\documentclass[unknownkeysallowed]{amsart}
\usepackage{etex}
\usepackage[style=alphabetic,doi=false,isbn=false,url=false,maxbibnames=99,backend=bibtex]{biblatex}
\usepackage{fixltx2e}
\DeclareFieldFormat[article,incollection,unpublished]{title}{#1}\renewbibmacro{in:}{\ifentrytype{article}{}{\printtext{\bibstring{in}\intitlepunct}}}
\bibliography{derivedpatching}

\newcommand{\Tor}{\mathrm{Tor}}
\DeclareMathOperator{\cInd}{c-Ind}

\usepackage{amssymb,amsmath,amsfonts,amsthm,epsfig,amscd}
\usepackage{stmaryrd}
\usepackage[all,cmtip,poly]{xy}
\usepackage{color}

\newcommand{\Rbar}{\overline{R}}
\newcommand{\Rdef}{R^{\operatorname{def}}}
\newcommand{\loc}{\operatorname{loc}}
\newcommand{\ad}{\operatorname{ad}}
\newcommand{\diag}{\operatorname{diag}}
\newcommand{\tr}{\operatorname{tr}}

\newcommand{\CNL}{\operatorname{CNL}}
\newcommand{\Sets}{\operatorname{Sets}}

\newcommand{\gF}{{\mathfrak{F}}}

\newcommand{\To}{\longrightarrow}
\newcommand{\isoto}{\stackrel{\sim}{\To}}

\newcommand{\wotimes}{\widehat{\otimes}}

\newcommand{\Rbarinfty}{{\bar{R}_\infty}}

 \newcommand{\sigmabar   }{\overline{\sigma}}

\newcommand{\id}{\operatorname{id}}

\newcommand{\RHom}{{\mathrm{RHom}}}

\newtheorem{lemma}[subsubsection]{Lemma}
\newtheorem{lem}[subsubsection]{Lemma}

\newtheorem{ithm}{Theorem}

\newtheorem{cor}[subsubsection]{Corollary}
\newtheorem{conj}[subsubsection]{Conjecture}

\newtheorem{prop}[subsubsection]{Proposition}

\newtheorem{alemma}[subsection]{Lemma}
\newtheorem{alem}[subsection]{Lemma}
\newtheorem{adefn}[subsection]{Definition}

\newtheorem{acor}[subsection]{Corollary}
\newtheorem{aprop}[subsection]{Proposition}

\theoremstyle{definition}

\newtheorem{defn}[subsubsection]{Definition}

\theoremstyle{remark}
\newtheorem{remark}[subsubsection]{Remark}
\newtheorem{rem}[subsubsection]{Remark}
\newtheorem{aremark}[subsection]{Remark}

\newtheorem{hypothesis}[subsubsection]{Hypothesis}

\def\numequation{\addtocounter{subsubsection}{1}\begin{equation}}
\def\nummultline{\addtocounter{subsubsection}{1}\begin{multline}}
\def\anumequation{\addtocounter{subsection}{1}\begin{equation}}
\def\anummultline{\addtocounter{subsection}{1}\begin{multline}}

\renewcommand{\theequation}{\arabic{section}.\arabic{subsection}.\arabic{subsubsection}}

\newif\iffinalrun
\iffinalrun
\else
 \fi

\iffinalrun
  \newcommand{\need}[1]{}
  \newcommand{\mar}[1]{}
\else
  \newcommand{\need}[1]{{\tiny *** #1}}
  \newcommand{\mar}[1]{\marginpar{\raggedright\tiny IUTT #1}}\fi

\newcommand{\A}{\AA}

\newcommand{\F}{\FF}

\newcommand{\Q}{\QQ}
\newcommand{\R}{\RR}
\newcommand{\Z}{\ZZ}

\newcommand{\m}{\frakm}

\renewcommand{\AA}{{\mathbb A}}

\newcommand{\FF}{{\mathbb F}}

\newcommand{\LL}{{\mathbb L}}

\newcommand{\NN}{{\mathbb N}}

\newcommand{\QQ}{{\mathbb Q}}
\newcommand{\RR}{{\mathbb R}}

\newcommand{\TT}{{\mathbb T}}

\newcommand{\ZZ}{{\mathbb Z}}

\newcommand{\ba}{\ensuremath{\mathbf{a}}}
\newcommand{\bb}{\ensuremath{\mathbf{b}}}

\renewcommand{\bf}{\ensuremath{\mathbf{f}}}

\newcommand{\cC}{{\mathcal C}}
\newcommand{\cD}{{\mathcal D}}

\newcommand{\cF}{{\mathcal F}}
\newcommand{\cG}{{\mathcal G}}
\newcommand{\cH}{{\mathcal H}}

\newcommand{\cJ}{{\mathcal J}}

\newcommand{\cO}{{\mathcal O}}

\newcommand{\cS}{{\mathcal S}}
\newcommand{\cT}{{\mathcal T}}

\newcommand{\frakm}{\mathfrak{m}}

\newcommand{\frakp}{\mathfrak{p}}

\newcommand{\Qbar}{\overline{\Q}}

\newcommand{\Fp}{\F_p}

\newcommand{\Zp}{\Z_p}

\newcommand{\Qp}{\Q_p}

\newcommand{\Qpbar}{\Qbar_p}

\DeclareMathOperator{\coker}{coker}

\DeclareMathOperator{\End}{End}
\DeclareMathOperator{\Ext}{Ext}

\DeclareMathOperator{\GL}{GL}

\DeclareMathOperator{\Hom}{Hom}
\DeclareMathOperator{\im}{im}

\DeclareMathOperator{\Mod}{Mod}

\DeclareMathOperator{\PGL}{PGL}

\DeclareMathOperator{\PSL}{PSL}
\DeclareMathOperator{\rank}{rank}
\DeclareMathOperator{\SL}{SL}

\DeclareMathOperator{\Spec}{Spec}

\DeclareMathOperator{\Sym}{Sym}

\newcommand{\Frob}{\mathrm{Frob}}

\newcommand{\rhobar}{\overline{\rho}}

 \newcommand{\dirlim}{\varinjlim}
\newcommand*{\invlim}{\varprojlim}                               
\newcommand{\onto}{\twoheadrightarrow}

\newcommand{\Art}{{\operatorname{Art}}}

\newcommand{\epsilonbar}{\overline{\epsilon}}

\newcommand{\Res}{\operatorname{Res}}

\DeclareMathOperator{\pd}{pd}
\DeclareMathOperator{\dph}{depth}

\usepackage[bookmarksopen,bookmarksdepth=2]{hyperref}

\begin{document}
\title{Patching and the completed homology of locally symmetric spaces}

\author[T. Gee]{Toby Gee} \email{toby.gee@imperial.ac.uk} \address{Department of
  Mathematics, Imperial College London,
  London SW7 2AZ, UK}

\author[J. Newton]{James Newton} \email{j.newton@kcl.ac.uk} \address{Department of
  Mathematics, King's College London,
  London WC2R 2LS, UK}

\thanks{The first author was
  supported in part by a Leverhulme Prize, EPSRC grant EP/L025485/1, Marie Curie Career
  Integration Grant 303605, and by
  ERC Starting Grant 306326. The second author was supported by ERC Starting Grant 306326.}

\maketitle
\paragraph{\textbf{Abstract}}
        Under an assumption on the existence of $p$-adic Galois
representations, we carry out Taylor--Wiles patching (in the derived
category) for the completed homology of the locally symmetric spaces
associated to $\GL_n$ over a number field. We use our construction,
and some new results in non-commutative algebra, to
show that standard conjectures on completed homology imply `big $R$ =
big~$\TT$' theorems in situations where one cannot hope to appeal to
the Zariski density of classical points (in contrast to all previous
results of this kind). In the case that $n=2$
and~$p$ splits completely in the number field, we relate our construction to the $p$-adic local Langlands
correspondence for~$\GL_2(\Qp)$.

\setcounter{tocdepth}{1}
\tableofcontents
\section{Introduction}In this paper we give a common generalisation of two recent
extensions of the Taylor--Wiles patching method, namely the
extension in~\cite{1207.4224} to cases where it is necessary to patch chain complexes rather than homology groups,
and the idea of patching completed homology explained
in~\cite{CEGGPSBreuilSchneider}. We begin by
explaining why this is a useful thing to do. Our main motivations come
from the $p$-adic
Langlands program, which is well understood for~$\GL_2/\Q$, but is
very mysterious beyond this case; and from the problem of proving
automorphy lifting theorems for $p$-adic automorphic forms (``big  $R = \TT$ theorems'') in situations
where classical automorphic forms are no longer dense (for example,
$\GL_n/\Q$ for any~$n>2$).

The  local $p$-adic Langlands correspondence for~$\GL_2(\Qp)$
has been established by completely local methods (see in
particular~\cite{MR2642409,paskunasimage}), and local-global
compatibility for~$\GL_2/\Q$ was established
in~\cite{emerton2010local} (which goes on to deduce many cases of the
Fontaine--Mazur conjecture). It has proved difficult to generalise the
local constructions for~$\GL_2/\Q$, and the
paper~\cite{CEGGPSBreuilSchneider} proposed instead (by analogy with
the original global proof of local class field theory) to construct a
candidate correspondence globally, by patching the completed homology
of unitary groups over CM fields.

This construction has the disadvantage that it seems to be very
difficult to prove that it is independent of the global situation, and
of the choices involved in Taylor--Wiles patching. However, in the
case of~$\GL_2(\Qp)$, the sequel~\cite{CEGGPSGL2} showed (without
using the results of~\cite{MR2642409,paskunasimage}) that the patching
construction is independent of global choices, and therefore uniquely
determines a local correspondence. 

It is natural to ask whether similar constructions can be carried out
for~$\GL_n$ over a number field~$F$. Until
recently it was believed that Taylor--Wiles patching only applied
to groups admitting discrete series (which would limit such a
construction to the case~$n=2$ and~$F$ totally real), but Calegari and
Geraghty showed in~\cite{1207.4224} that by patching chain complexes
rather than homology groups one can overcome this obstruction,
provided that one admits natural conjectures on the existence and
properties of Galois representations attached to torsion classes in (uncompleted)
homology. For a general~$F$ these conjectures are  open, but
for~$F$ totally real or CM the existence of the Galois representations
is known by~\cite{scholze-torsion}, and most of the necessary
properties are expected to be established in the near future (with the
possible exception of local-global compatibility at places
dividing~$p$, which we discuss further below).

The patching construction in~\cite{1207.4224} is sometimes a little ad
hoc, and it was refined in~\cite{1409.7007}, where the patching is
carried out in the derived category. The construction
of~\cite{CEGGPSBreuilSchneider} was improved upon in~\cite{scholze}, which uses
ultrafilters to significantly reduce the amount of bookkeeping needed
in the patching argument. We combine these two approaches, and
use ultrafilters to patch complexes in the derived category. In fact,
we take a different approach to~\cite{1409.7007}, by directly patching
complexes computing homology, rather than minimal resolutions of such complexes;
this has the advantage that our patched complex naturally has actions
of the Hecke algebras and $p$-adic analytic groups. The use of
ultrafilters streamlines this construction, and most of our constructions are natural, resulting in cleaner statements
and proofs. (We still make use of the existence of minimal resolutions
to show that our ultraproduct constructions are well behaved.)

To explain our results we
introduce some notation. Write  
$K_0=\prod_{v|p}\PGL_n(\cO_{F_v})$ and let $K_1$ denote a pro-$p$ Sylow subgroup of $K_0$. We consider locally symmetric spaces $X_{U}$ for $\PGL_n/F$, with level $U = U_pU^p \subset \PGL_n(\AA_F^\infty)$ where $U^p$ is some fixed tame level and $U_p$ is a compact open subgroup of $K_0$. Let~$\cO$ be the ring of integers in some
finite extension~$E/\Qp$, and write~$k$ for the residue field
of~$\cO$. We write~$\cO_\infty$ for a power series ring over~$\cO$
and~$R_\infty$ for a power series ring over the (completed) tensor
product of the local Galois deformation rings at the places~$v|p$
of~$F$. These power series rings are in some numbers of variables which
depend on the choice of Taylor--Wiles primes; these power series variables are unimportant for
the present discussion. For the purposes of this introduction, we will
also ignore the role of the local Galois deformation rings at places
$v\nmid p$ where our residual Galois representation is
ramified. 

The output of our patching
construction is a perfect chain complex $\widetilde{\cC}(\infty)$ of
$\cO_\infty[[K_0]]$-modules, equipped with an $\cO_\infty$-linear
action of $\prod_{v|p}\PGL_n(F_v)$ and an $\cO_\infty$-algebra homomorphism
 \[R_\infty
   \rightarrow \End_{D(\cO_\infty)}(\widetilde{\cC}(\infty))\]
 (where~$D(\cO_\infty)$ is the unbounded derived category
 of~$\cO_\infty$-modules).  The action of $R_\infty$ on
 $\widetilde{\cC}(\infty)$ commutes with the action of
 $\prod_{v|p}\PGL_n(F_v)$ (and with that of $\cO_\infty[[K_0]]$). Reducing the 
 complex $\widetilde{\cC}(\infty)$ modulo the ideal $\ba$ of $\cO_\infty$ generated 
 by the power series variables, we obtain a complex which computes the completed 
 homology groups \[\widetilde{H}_*(X_{U^p},\cO)_\m :=
   \invlim_{U_p}H_*(X_{U_pU^p},\cO)_\m\] localised at a non-Eisenstein
 maximal ideal~$\m$
 of a `big' Hecke algebra $\TT^S(U^p)$ which acts on completed homology.

Our first main result is to show that, assuming  a
vanishing conjecture of~\cite{1207.4224} (which says that homology
groups vanish outside of the expected range of
degrees~$[q_0,q_0+l_0]$ after localising at $\m$),
and a conjecture of~\cite{MR2905536} on
the codimension of completed homology, then the homology
of~$\widetilde{\cC}(\infty)$ vanishes outside of a single
degree~$q_0$, and $H_{q_0}(\widetilde{\cC}(\infty))$ is
Cohen--Macaulay over both~$\cO_\infty[[K_0]]$ and~$R_\infty[[K_0]]$ of
the expected projective dimensions. 

One novel feature of our work appears here: since we are working with finitely 
generated modules over the non-commutative algebras $\cO_\infty[[K_0]]$ and 
$R_\infty[[K_0]]$ we are forced to establish non-commutative 
analogues of the commutative algebra techniques which are applied in 
\cite{1207.4224}. The first crucial 
result is Lemma~\ref{lem:CG} (a generalisation of \cite[Lem.~6.2]{1207.4224}) 
which, as in \emph{op.~cit}~is used to establish vanishing of the homology of 
the patched complex outside degree $q_0$. The second is 
Corollary~\ref{cor:stayCM}, which is used to deduce the Cohen--Macaulay 
property for the patched module over $R_\infty[[K_0]]$ from the Cohen--Macaulay 
property over $\cO_\infty[[K_0]]$. 

If  $A$ is a ring and $M$ is
an $A$-module, then we write~$\pd_A(M)$ for the projective dimension
of $M$ over $A$, and $j_A(M)$ for its grade (also known as its
codimension; see Definition~\ref{def: gradedepth} and Remark~\ref{rem: grade 
and codimension}).
\begin{ithm}[Theorem~\ref{patchedCM}]\label{introCM}
Suppose that 
\begin{enumerate}
\item[(a)] ${H}_i(X_{U^pK_1},k)_\m = 0$ for $i$ outside the 
range $[q_0,q_0+l_0]$. 
\item[(b)]  
$j_{\cO[[K_0]]}\left(\bigoplus_{i\ge 
	0}\widetilde{H}_i(X_{U^p},\cO)_\m\right) \ge l_0$.\end{enumerate} 

Then \begin{enumerate}
\item $\widetilde{H}_i(X_{U^p},\cO)_\m = 0$ for $i \ne q_0$ and  
$\widetilde{H}_{q_0}(X_{U^p},\cO)_\m$ is a Cohen--Macaulay 
$\cO[[K_0]]$-module with
\[\mathrm{pd}_{\cO[[K_0]]}(\widetilde{H}_{q_0}(X_{U^p},\cO)_\m) = 
j_{\cO[[K_0]]}(\widetilde{H}_{q_0}(X_{U^p},\cO)_\m) = l_0.\] 
\item ${H}_i(\widetilde{\cC}(\infty)) = 0$ for $i \ne q_0$ and 
${H}_{q_0}(\widetilde{\cC}(\infty))$ is a Cohen--Macaulay 
$\cO_\infty[[K_0]]$-module with 	
\[\mathrm{pd}_{\cO_\infty[[K_0]]}\left({H}_{q_0}(\widetilde{\cC}(\infty))\right)
= 
j_{\cO_\infty[[K_0]]}\left({H}_{q_0}(\widetilde{\cC}(\infty))\right) = 
l_0.\] 
\item  ${H}_{q_0}(\widetilde{\cC}(\infty))$ is a Cohen--Macaulay 
$R_\infty[[K_0]]$-module with 	
\[\mathrm{pd}_{R_\infty[[K_0]]}\left({H}_{q_0}(\widetilde{\cC}(\infty))\right)
= 
j_{R_\infty[[K_0]]}\left({H}_{q_0}(\widetilde{\cC}(\infty))\right) 
= 
\dim(B)\] where $\dim(B) = 
(\frac{n(n+1)}{2}-1)[F:\QQ]$.
\end{enumerate}
\end{ithm}

The conjectures
 of~\cite{MR2905536} and~\cite{1207.4224} are open in general, but
 they are known if~$n=2$ and~$F$ is imaginary quadratic. 

 In Section~\ref{subsec: miracle flatness} we take this analysis
 further. Here it is essential for us to assume that $R_\infty$ is regular. 
 Under a natural condition on the codimension (over
 $k[[K_0]]$) of the fibre of completed homology at $\m$, we prove the
 following result, which shows that the Hecke algebra $\TT^S(U^p)_\m$ is isomorphic to a Galois deformation ring $R$ (a `big
 $R=\TT$' theorem), making precise the heuristics discussed in
 \cite[\S 3.1.1]{emertonICM} which compare the Krull dimensions of
 Hecke algebras and the Iwasawa theoretic dimensions of completed
 homology modules and their fibres. 
 
 \begin{ithm}[Proposition~\ref{prop: big R equals T}]\label{ibigRT}
 Suppose that the assumptions of Theorem~\ref{introCM} hold, that $R_\infty$ is 
 a power series ring over $\cO$, and 
 that we 
 moreover have
 \[j_{k[[K_0]]}(\widetilde{H}_{q_0}(X_{U^p},\cO)_\m/\m\widetilde{H}_{q_0}(X_{U^p},\cO)_\m)
 \ge \dim(B).\]
 
 Then we have the following: \begin{enumerate}
 \item $H_{q_0}(\widetilde{\cC}(\infty))$ is a flat $R_\infty$-module.
 \item The ideal $R_\infty\ba$ is generated by a regular sequence in 
 $R_\infty$.
 \item The surjective maps \[R_\infty/\ba \rightarrow R 
 \rightarrow \TT^S(U^p)_\m\] are all isomorphisms and 
 $\widetilde{H}_{q_0}(X_{U^p},\cO)_\m$ is a faithfully flat 
 $\TT^S(U^p)_\m$-module.
 \item The rings $R \cong \TT^S(U^p)_\m$ are local complete 
 intersections with Krull dimension equal to $1 + \dim(B) - l_0$.
 \end{enumerate}
 \end{ithm}
We note here a crucial difference between our set-up and the situation in which 
Taylor--Wiles patching (and its variants) is usually applied --- the patched 
module $H_{q_0}(\widetilde{\cC}(\infty))$ is not finitely generated over 
$R_\infty$. The patched module is finitely generated over $R_\infty[[K_0]]$ but 
is not free over this Iwasawa algebra (it has codimension $\dim(B)$). So the 
usual techniques to establish `$R=\TT$' do not apply. 

Moreover, even if we 
could establish that $H_{q_0}(\widetilde{\cC}(\infty))$ is a faithful 
$R_\infty$-module, this would not be enough to conclude that the map $R 
\rightarrow \TT^S(U^p)_\m$ has nilpotent kernel. Instead we need to establish 
the stronger result that $H_{q_0}(\widetilde{\cC}(\infty))$ is a \emph{flat} 
$R_\infty$-module. The main novelty of Theorem~\ref{ibigRT} is that the simple 
codimension inequality appearing in the statement is enough to guarantee this 
flatness. This follows from a version of the miracle flatness criterion in 
commutative algebra (Prop.~\ref{miracleflatness} --- again we must modify 
things to 
handle the fact that our 
modules are only finitely generated over a non-commutative algebra). 

Establishing the codimension inequality seems to require substantial 
information about the mod $p$ representations of $\prod_{v|p}\PGL_n(F_v)$ 
appearing in completed cohomology. Even in $l_0 = 0$ situations, we do not know 
how to establish this codimension inequality (in contrast to the assumptions 
made in Theorem~\ref{introCM}, which become trivial when working in an 
appropriate $l_0=0$ setup) --- if we did, our methods would 
give a new approach to proving big $R = \TT$ theorems in these situations. In 
the case $n=2$, 
$F=\Q$, the codimension inequality follows from Emerton's $p$-adic 
local--global 
compatibility theorem, together with known properties of the $p$-adic local 
Langlands 
correspondence. In Section~\ref{sec: p splits completely} we show that 
some conjectural local--global compatibility statements when~$n=2$ and~$p$ 
splits completely in~$F$ also imply that this codimension inequality holds.

This strategy for 
 establishing big $R = \TT$ theorems seems to be the only way known at present 
 to handle the $l_0 > 0$ situation (Emerton, in a personal communication, tells 
 us that this was the initial motivation for him and Calegari to consider the 
 codimension of completed homology and compare it with dimensions of Galois 
 deformation rings and Hecke algebras). Existing results in the $l_0 = 0$ 
 case
 (\cite{gouveamazur,boecklebigrt,cheninffern,allenbigrt}) rely on 
 establishing Zariski density of (characteristic 0) automorphic points in the 
 unrestricted Galois deformation ring $R$, using generalisations of the 
 Gouv\^ea--Mazur infinite fern. When $l_0 > 0$ characteristic 0 automorphic 
 points are not expected to be Zariski dense in $R$, and they are not Zariski 
 dense in the relevant eigenvarieties (see \cite{calmaz} and work of Serban 
 described in \cite{persblog}), so this approach breaks down.

 In Section~\ref{sec: p splits completely} we specialise to the case
 that~$n=2$ and~$p$ splits completely in~$F$, where we can relate our
 constructions to the $p$-adic local Langlands correspondence
 for~$\GL_2(\Qp)$. We formulate a natural conjecture
 (Conjecture~\ref{conj: p-adic local Langlands gives us the action for
   general F}) saying that the patched module~
 $H_{q_0}(\widetilde{\cC}(\infty))$ is determined by (and in fact determines) this
 correspondence; in the case~$F=\Q$ this conjecture is proved
 in~\cite{CEGGPSGL2}, and is essentially equivalent to the
 local-global compatibility result of~\cite{emerton2010local}. We show
 that this conjecture implies a local-global compatibility result (in
 the derived category) for the complexes computing finite level homology modules
 with coefficients in an algebraic representation; this compatibility is
 perhaps somewhat surprising, as it is phrased in terms of crystalline
 deformation rings, which are not obviously well-behaved integrally. 
 
 Conversely, we show that if we assume (in addition to the assumptions
 made in Section~\ref{sec: applications of algebra to patched
 	homology}) that this local-global compatibility holds
 at finite level, then Conjecture~\ref{conj: p-adic local Langlands
 	gives us the action for general F} holds. Our proof is an
 adaptation of the methods of~\cite{CEGGPSGL2}, although some
 additional arguments are needed in our more general setting.
 
 We 
 moreover show that Conjecture~\ref{conj: p-adic local Langlands gives us the 
 action for
 	general F} has as consequences an automorphy lifting theorem and a `small 
 	$R[1/p] = \TT[1/p]$' result (Corollary~\ref{cor: modularity lifting}). 
 	Therefore, our local-global compatibility conjecture entails many new cases of the Fontaine--Mazur conjecture. 
 	The application to Fontaine--Mazur was established by 
 	\cite{emerton2010local} in the case $F=\Q$, and although our argument looks 
 	rather different it is closely related to that of \emph{loc.~cit.} (but see 
 	also Remark~\ref{emertonfmarg}).

 While our main results are all conditional on various natural
 conjectures about (completed) homology groups, in the case that~$n=2$
 and~$F$ is an imaginary quadratic field in which~$p$ splits it seems
 that the only serious obstruction is our finite level local-global
 compatibility conjecture (Conjecture~\ref{conj: crystalline local global}), as 
 we explain in Section~\ref{subsec:
   totally real or imaginary quadratic}.

We end this introduction by briefly explaining the contents of the
sections that we have not already described. In
Section~\ref{sec:patchingI} we introduce the complexes that we will
patch and the Hecke algebras that act on them, and prove some standard
results about minimal resolutions of complexes. We also prove some
basic results about ultraproducts of complexes. In
Section~\ref{sec:patchingII} we introduce the Galois deformation
rings, carry out our patching construction, and prove its basic
properties (for example, we establish its compatibility with completed
homology).

In Appendix~\ref{sec: non commutative algebra} we establish analogues
for Iwasawa algebras of various classical results in commutative
algebra, which we apply to our patched complexes in~Section~\ref{sec:
  applications of algebra to patched homology}. Finally in
Appendix~\ref{appendix: tensor products projective covers} we prove
some basic results about tensor products and projective envelopes of
pseudocompact modules that we use in Section~\ref{sec: p splits
  completely}.

\subsection{Acknowledgements}\label{subsec:acknowledgements}We would
like to thank Frank Calegari, Matt Emerton, Christian Johansson, Vytas Pa{\v{s}}k{\=u}nas,
Jack Thorne and Simon Wadsley for helpful conversations. We would also like to 
thank the anonymous referee for helpful comments.

\subsection{Notation}\label{subsec:notation}

Let~$F$ be a number field, and fix an algebraic closure $\overline{F}$
of $F$, as well as algebraic closures $\overline{F}_v$ of the
completion $F_v$ of~$F$ at $v$ for each place~$v$ of~$F$,
and embeddings $\overline{F} \hookrightarrow \overline{F}_v$ extending
the natural embeddings $F \hookrightarrow F_v$.  These choices
determine embeddings of absolute Galois groups
$G_{F_v} \hookrightarrow G_F$. If $v$ is a finite place of $F$, then
we write $I_{F_v} \subset G_{F_v}$ for the inertia group, and
$\Frob_v \in G_{F_v}/I_{F_v}$ for a geometric Frobenius element; we
normalise the local Artin maps $\Art_{F_v}$ to send uniformisers to
geometric Frobenius elements. We write $\A_F$ for the adele ring
of~$F$, and $\A_F^\infty$ for the finite adeles.

We fix a prime~$p$ throughout, and write $\epsilon:G_F\to\Zp^\times$ for
the $p$-adic cyclotomic character. Let~$\cO$ be the ring of integers
in a finite extension~$E/\Qp$ with residue field~$k$; our Galois
representations will be valued in~$\cO$-algebras (but we will feel
free to enlarge~$E$ where necessary). If $R$ is a complete
Noetherian local $\cO$-algebra with residue field $k$, then we write
$\mathrm{CNL}_R$ for the category of complete Noetherian local
$R$-algebras with residue field $k$.

If~$R$ is a ring, we write~$Ch(R)$ for the abelian category of chain
complexes of $R$-modules. If~$C_\bullet\in Ch(R)$ then we
write~$H_*(C_{\bullet}):=\oplus_{n\in\Z}H_n(C_\bullet)$.  We write
$D(R)$ for the (unbounded) derived category of $R$-modules --- for us,
the objects of $D(R)$ are \emph{cochain} complexes of $R$-modules, but
we regard a chain complex $C_\bullet\in Ch(R)$ as a cochain complex
$C^\bullet$ by setting $C^i = C_{-i}$. We write $D^-(R)$ for the
bounded-above derived category of $R$-modules. The objects of $D^-(R)$
are cochain complexes of $R$-modules with bounded-above cohomology, or
(equivalently) chain complexes of $R$-modules with bounded-below
homology.  Similarly, we write $D^+(R)$ for the bounded-below derived
category.  

An object $C^\bullet$ of $D(R)$ is called a \emph{perfect
  complex} if there is a quasi-isomorphism
$P^\bullet \rightarrow C^\bullet$ where $P^\bullet$ is a bounded
complex of finite projective $R$-modules. In fact, $C^\bullet$ is
perfect if and only if it is isomorphic in $D(R)$ to a bounded complex
$P^\bullet$ of finite projectives: if we have another complex
$D^\bullet$ and quasi-isomorphisms $P^\bullet \rightarrow D^\bullet$,
$C^\bullet \rightarrow D^\bullet$, then there is a quasi-isomorphism
$P^\bullet \rightarrow C^\bullet$
(\cite[\href{http://stacks.math.columbia.edu/tag/064E}{Tag
  064E}]{stacks-project}).

If~$K$ is a compact $p$-adic analytic group, we have the Iwasawa algebra
 $\cO[[K]]:=\invlim_U\cO[K/U]$, where~$U$ runs
over the open normal subgroups of~$K$. This is a (non-commutative)
Noetherian ring, some of whose properties we recall in Appendix~\ref{sec: non commutative
  algebra}. If~$R$ is a formally smooth (commutative) $\cO$-algebra,
then we write $R[[K]]:=R\wotimes_\cO\cO[[K]]$; note that if $R$ has
relative dimension~$d$ over~$\cO$, then
$R[[K]]\cong\cO[[K\times\Zp^d]]$, so general properties of~$\cO[[K]]$
are inherited by~$R[[K]]$.

   For technical reasons, we will sometimes assume that~$K$ is a
   uniform pro-$p$ group in the sense explained in~\cite[\S
   1.2]{Venjakob}; as explained there, this can always be achieved by
   replacing~$K$ with a normal open subgroup. The group~$\Zp^d$ is
   uniform pro-$p$, so properties of~$\cO[[K]]$ for~$K$ a uniform
   pro-$p$ group are again inherited by~$R[[K]]$.

If $M$ is a pseudocompact (i.e.~profinite) $\cO$-module, we write $M^\vee:=\Hom_\cO^{cts}(M,E/\cO)$ for the
Pontryagin dual of $M$.

\section{Patching I: Completed homology complexes and
  ultrafilters}\label{sec:patchingI}In this section and the following
one we explain our
patching construction. For the convenience
of the reader, we will generally follow the notation
of~\cite{1409.7007}. 

\subsection{Arithmetic quotients}\label{subsec: manifolds }
We begin by introducing the manifolds whose homology we will patch. We
follow~\cite{1207.4224} in patching arithmetic quotients
for~$\PGL_n$, rather than~$\GL_n$; this is a minor issue in practice, as the
connected components of the arithmetic quotients are the same for either
choice, and we are for the most part able to continue to
follow~\cite{1409.7007}, although we caution the reader that because
of this change, it is sometimes the case that we use the same notation
to mean something slightly different to the corresponding definition
in~\cite{1409.7007}.

 Let $G =
\PGL_{n, F}$, let $G_\infty = G(F \otimes_\Q \R)$, and let $K_\infty
\subset G_\infty$ be a maximal compact subgroup. Write $X_G:=G_\infty/K_\infty$. If $U \subset G(\A_F^\infty)$ is an open 
compact subgroup, then we
define

\[X_U = G(F) \backslash (G(\A_F^\infty) / U \times X_G),\]

If $U \subset G(\A_F^\infty)$ is an open compact subgroup of the form $U = \prod_v U_v$, we say that $U$ is \emph{good} if it satisfies the following conditions:
\begin{itemize}
\item For each $g \in G(\A_F^\infty)$, the group
  $\Gamma_{U,g}:=gUg^{-1}\cap G(F)$ is neat, and in particular
  torsion-free. (By definition, $\Gamma_{U,g}$ is neat if for each
  $h\in\Gamma_{U,g}$, the eigenvalues of~$h$ generate a torsion-free
  group.)
\item For each finite place $v$ of $F$, $U_v \subset \PGL_n(\cO_{F_v})$.
\end{itemize}
We write $U=U_pU^p$, where $U_p=\prod_{v|p}U_v$, $U^p=\prod_{v\nmid p}U_v$. If~$S$ is a finite set of finite places of~$F$, then we say that~$U$
is $S$-\emph{good} if $U_v = \PGL_n(\cO_{F_v})$ for all~$v\notin S$. 

By the proof of~\cite[Lem.\ 6.1]{1409.7007}, if $U$ is good, then~$X_U$ is a smooth
manifold, and if $V\subset U$ is a normal compact open subgroup,
then~$V$ is also good, and
$X_V\to X_U$ is a Galois cover of smooth
manifolds. 

 Let $r_1,$ $r_2$ denote the number of real and complex places of $F$, respectively. Then
\numequation  \dim X_U = \frac{r_1}{2}(n-1)(n+2) + r_2(n^2 - 1).\end{equation}

The \emph{defect} is
\numequation\label{eqn: defect} l_0 = \rank G_\infty - \rank  K_\infty = \left\{ \begin{array}{ll} r_1 (\frac{n-2}{2}) + r_2(n-1) & n \text { even;} \\ r_1(\frac{n-1}{2}) + r_2(n-1) &n \text{ odd,} \end{array}\right. 
\end{equation}
and we also set 
\numequation q_0 = \frac{d-l_0}{2} = \left\{ \begin{array}{ll} r_1 (\frac{n^2}{4}) + r_2\frac{n(n-1)}{2} & n \text { even;} \\ r_1 (\frac{n^2-1}{4}) + r_2\frac{n(n-1)}{2} &n \text{ odd.} \end{array}\right. 
\end{equation}
In particular, if $F$ is an imaginary quadratic field and~$n=2$, then
$\dim X_U=3$, $l_0=1$, and $q_0=1$. The notation $l_0,q_0$ comes
from~\cite{MR1721403}, and $[q_0,q_0+l_0]$ is the range of degrees in
which tempered cuspidal automorphic representations of~$G$ contribute
to the cohomology of the~$X_U$.

Let $C_{\A,\bullet}$ denote the complex of singular chains with
$\Z$-coefficients which are valued in $G(\A_F^\infty)\times
{X}_G$, where $G(\A_F^\infty)$ is given the discrete topology. We equip 
$G(\A_F^\infty)\times
{X}_G$ with the right $G(F)\times G(\A^\infty_F)$ action 
\[(h^\infty,x)\cdot(\gamma,g^\infty) = (\gamma^{-1}h^\infty g^\infty, 
\gamma^{-1}x)\]

 which makes $C_{\A,\bullet}$ a complex of right $\Z[G(F)\times
G(\A^\infty_F)]$-modules. 
If~$U$ is good and $M$ is a left~$\Z[U]$-module, then we
 set \[\cC(U,M):=C_{\A,\bullet}\otimes_{\Z[G(F)\times
    U]}M.\] As in~\cite[Prop.\ 6.2]{1409.7007}, there is a
natural isomorphism \[H_*(X_U,M)\cong
  H_*(\cC(U,M)).\]
  
If $U=U_pU^p$ is good, then we
have the completed homology groups in the sense of~\cite{MR2905536}
which by definition are given
by \[\widetilde{H}_*(X_{U^p},\cO):=\varprojlim_{U'_p}H_*(X_{U'_pU^p},\cO),\] the
limit being taken over open subgroups~$U'_p$ of~$U_p$.

We note here that the homology groups $H_*(X_{U},\cO)$ are all finitely 
generated $\cO$-modules. This follows from the existence of the Borel–Serre 
compactification \cite{MR0387495}, or the earlier work of Raghunathan 
\cite{raghunathan}.

\subsubsection{Hecke operators}\label{subsubsec: Hecke operators on complexes} Our complexes have a natural Hecke action in the usual way, as
described in~\cite[\S 6.2]{1409.7007}. We recall some of the
details. Suppose that~$U,V$ are good subgroups, that~$S$ is a finite
set of places of~$F$ with $U_v=V_v$ if $v\in S$, and that $M$ is a
$\Z[G(\A_F^{\infty,S})\times U_S]$-module. Then for each $g\in G(\A_F^{\infty,S})$
there is a Hecke operator \[[UgV]_*:\cC(V,M)\to
  \cC(U,M)\] given by the
formula \[([UgV]_*((h\times \sigma)\otimes
  m)=\sum_i
  (hg_i\times\sigma)\otimes g_i^{-1}m,\]where $h\in G(\A_F^\infty)$, $\sigma:\Delta^j\to
{X}_G$ is a singular simplex, $m\in M$, and $UgV = \coprod_i g_iV$.

In practice, we will take~$S=S_p$ to be the set of places of~$F$ lying
over~$p$, and we take~$M$ to be a finite~$\Zp$-module with a
continuous action of~$\prod_{v|p}\PGL_n(\cO_{F_v})$, with the action
of~$G(\A_F^{\infty,S})\times U_S$ on~$M$ being via projection
to~$U_S\subset \prod_{v|p}\PGL_n(\cO_{F_v}) $. (In fact, we will
usually take the action on~$M$ to be the trivial action.) If~$v\notin
S_p$ is a finite place of~$F$, then
we choose a uniformiser~$\varpi_v$ of~$\cO_{F_v}$, and for each $1\le
i\le n$ we set~$\alpha_{v,i}=\diag(\varpi_v,\dots,\varpi_v,1,\dots,1)$
(with~$i$ occurrences of~$\varpi_v$). 

If~$v\notin S$ is a place for which~$U_v=\PGL_n(\cO_{F_v})$, we set
$T_v^i:=[U\alpha_{v,i}U]_*$, where by an abuse of notation we denote by 
$\alpha_{v,i}$ the element of $G(\A_F^{\infty,S})$ which is equal to 
$\alpha_{v,i}$ in the $v$ component and the identity elsewhere; these operators 
are independent of the
choice of~$\varpi_v$, and pairwise commute. We also consider places at
which~$U_v$  is a normal subgroup of the standard Iwahori
subgroup which contains the standard pro-$\varpi_v$-Iwahori
subgroup. At these places we
will set $\mathbf{U}^i_v=[U\alpha_{v,i}U]_*$; these operators now
depend on the choice of~$\varpi_v$, but (for the particular~$U_v$ that
we use) they still pairwise commute. They also commute with the
diamond operators~$\langle\alpha\rangle=[U\alpha U]_*$,
where~$\alpha$ is an element of the standard Iwahori subgroup whose
diagonal entries are all equal modulo~$\varpi_v$.

Note that it is immediate from the definitions that the actions of the
operators~$T^i_v$ and $\mathbf{U}^i_v$ are equivariant for the natural
morphisms of complexes arising from shrinking the level~$U$ away from $v$.

\subsubsection{Minimal resolutions}
We recall some standard material on minimal resolutions of complexes. Since we 
work over non-commutative rings, there don't seem to be any standard references.

Let $R$ be a Noetherian local ring (possibly non-commutative). We denote the 
maximal ideal by $\m$ and assume that $R/\m = k$ is a field.

\begin{defn}
	Let $\cF_\bullet$ be a chain complex of finite free $R$-modules. The 
	complex $\cF_\bullet$ is \emph{minimal} if for all $i$ the boundary map 
	$d_i : \cF_{i+1} \rightarrow \cF_{i}$ satisfies \[d_i(\cF_{i+1}) \subset \m 
	\cF_{i}.\]
\end{defn}
Note that if $\cF_\bullet$ is minimal, the complex $k\otimes_R \cF_\bullet$ has 
boundary maps equal to zero.

\begin{lemma}\label{lem: minimal complex rank}
	Let $\cF_\bullet$ be a minimal complex of finite free $R$-modules with 
	bounded below homology, so that thinking of $\cF_\bullet$ as 
	an object of the derived category $D^-(R)$, we have a well-defined object 
	$k\otimes^\LL_R \cF_\bullet \in D^-(k)$. Then for each 
	$n$ we have \[\mathrm{rank}_R(\cF_i) = \dim_k(H_i(k\otimes_R 
	\cF_\bullet)) 
	= 
	\dim_k(H_i(k\otimes^\LL_R \cF_\bullet)).\]
	
	In particular, the ranks of the modules $\cF_i$ depend only on the 
	isomorphism class of $\cF_\bullet$ in $D^-(R)$.
\end{lemma}
\begin{proof}
	We have $\mathrm{rank}_R(\cF_i) = \dim_k(k\otimes_R\cF_i)$, and since $\cF$ 
	is minimal we have \[k\otimes_R\cF_i = H_i(k\otimes_R\cF_\bullet).\qedhere\]
\end{proof}
\begin{defn}
	Let $\cC_\bullet \in Ch(R)$ with bounded below homology. If $\cF_\bullet$ 
	is a minimal complex (necessarily bounded below)
	with a quasi-isomorphism $\cF_\bullet \rightarrow \cC_\bullet$, we say that 
	$\cF_\bullet$ is a \emph{minimal resolution} of $\cC_\bullet$.
\end{defn}
If $\cF_\bullet$ is a minimal resolution of $\cC_\bullet$, then by
Lemma~\ref{lem: minimal complex rank} we have 
\[\mathrm{rank}_R(\cF_i) = \dim_k(H_i(k\otimes^\LL_R 
\cC_\bullet)).\]

\begin{prop}\label{prop: minimal resolutions exist}
	Let $\cC_\bullet \in Ch(R)$  be a chain complex with bounded
        below homology, and assume further that
	$H_i(\cC_\bullet)$ is a finitely generated 
	$R$-module for all $i$. Then there 
	exists a minimal resolution $\cF_\bullet$ of $\cC_\bullet$,
        and any two minimal 
	resolutions of $\cC_\bullet$ are isomorphic \emph{(}although the isomorphism is 
	not necessarily unique\emph{)}.
\end{prop}
\begin{proof}
	By considering the canonical truncation $\tau_{\ge N}\cC_\bullet 
	\rightarrow\cC_\bullet$ (which is an isomorphism for sufficiently negative 
	$N$), we may assume 
	that the complex $\cC_\bullet$ is bounded below. The proof in the 
	commutative case from \cite[\S2, Theorem 2.4]{roberts} applies 
	without change (the proof in \emph{loc.~cit.}~assumes that the complex has 
	bounded homology, but this is not necessary). For the reader's convenience, 
	we sketch the proof. 
	
	First we check the uniqueness of the minimal resolution: suppose we have 
	two minimal resolutions $\cF_{1,\bullet},\cF_{2,\bullet}$ of $\cC_\bullet$. 
	Then $\cF_{1,\bullet},\cF_{2,\bullet}$ are isomorphic in $D(R)$. Since 
	$\cF_{1,\bullet}$ is a bounded below chain complex of projective modules 
	there is a quasi-isomorphism $\cF_{1,\bullet} \rightarrow \cF_{2,\bullet}$ 
	(by \cite[\href{http://stacks.math.columbia.edu/tag/0649}{Tag 
	0649}]{stacks-project}). This map induces a quasi-isomorphism 
	$k\otimes_R\cF_{1,\bullet} \rightarrow k\otimes_R\cF_{2,\bullet}$, and 
	minimality implies that this quasi-isomorphism is actually an isomorphism 
	of complexes. Nakayama's lemma now implies that $\cF_{1,\bullet} 
	\rightarrow \cF_{2,\bullet}$ is an isomorphism of complexes.
	
	Now we show existence of the minimal resolution. First, by a
        standard argument (see for example \cite[Lem.~1, 
	pp.47--49]{mumford}), there is a (not 
	necessarily minimal) bounded below complex of finite free modules 
	$\cG_\bullet$ with a quasi-isomorphism $\cG_\bullet \rightarrow 
	\cC_\bullet$. 
	
	We now inductively suppose that the complex $\cG_\bullet$ satisfies $d_m(\cG_{m+1}) 
	\subset \m 
	\cG_{m}$ for $m < i$. (Note that this is certainly true
        for~$i\ll 0$.) We will construct a new bounded below complex 
	$\cG'_\bullet$ of finite free modules with $\cG'_m = \cG_m$ for $m < i$, 
	together with a quasi-isomorphism $\cG'_\bullet \rightarrow 
	\cG_\bullet$, such that $d_m(\cG'_{m+1}) \subset \m 
	\cG'_{m}$ for $m \le i$. Iterating this procedure constructs the minimal 
	resolution $\cF_\bullet$.
	
	So, we suppose that $d_{i}(\cG_{i+1}) 
	\not\subset \m\cG_{i}$. We let $Y$ be a subset of $\cG_{i+1}$ which lifts a 
	linearly independent subset of $k\otimes_R\cG_{i+1}$ mapping 
	(injectively) to a basis for $d_{i}(k\otimes_R\cG_{i+1}) \subset 
	k\otimes_R\cG_{i}$. Then the acyclic complex (with non-zero terms in degree 
	$i+1$ and $i$) \[\cC(Y)  = \oplus_{y\in 
	Y}\left(0\rightarrow 
	Ry \overset{d_i}{\rightarrow} Rd_i(y)\rightarrow 0\right)\] is a direct 
	summand of 
	$\cG_\bullet$ (a splitting of $\oplus_{y\in 
		Y}Rd_i(y) \subset \cG_i$ induces a compatible splitting of 
		$\oplus_{y\in 
		Y}Ry \subset \cG_{i+1}$ and such a splitting exists since $d_i(Y)$ 
		extends to a basis of $\cG_i$ by Nakayama's lemma), 
	and we set $\cG_\bullet' = \cG_\bullet/\cC(Y)$. Since $\cG'_\bullet$ is a 
	direct 
	summand of $\cG_\bullet$ we may choose a splitting $\cG'_\bullet 
	\rightarrow 
	\cG_\bullet$ of the projection map. This splitting is a 
	quasi-isomorphism, since $\cC(Y)$ is acyclic. 
	It is easy to check that $\cG'_\bullet$ has the other desired properties, 
	so we are done. 
\end{proof}

\subsubsection{Big Hecke algebras}\label{subsubsec: big hecke 
algebras}Write~$\cC(U,s):=\cC(U,\cO/\varpi^s)$.
\begin{defn}
	Let $S$ be a finite set of finite places of $F$ which contains $S_p$. Let $U 
	= U_pU^p$ be an
	$S$-good 
	subgroup, with $U_p$ a compact open normal subgroup of $K_0$. We define 
	$\TT^S(U,s)$ to 
	be the image of the abstract Hecke 
	algebra 
	$\TT^S$ (generated over $\cO$ by $T^i_v$ for $v \notin S$) in 
	$\End_{D(\cO/\varpi^s[K_0/U_p])}(\cC(U,s))$.

	We let \[\TT^S(U^p) = \invlim_{U_p,s}\TT^S(U_pU^p,s)\] where the limit is 
	over compact open normal subgroups $U_p$ of $K_0$ and $s \in 
	\Z_{\ge 1}$, and the (surjective) transition maps come from the functorial 
	maps 
	\[\End_{D(\cO/\varpi^{s'}[K_0/U'_p])}(\cC(U'_pU^p,s')) \rightarrow 
	\End_{D(\cO/\varpi^s[K_0/U_p])}
	(\cO/\varpi^{s}[K_0/U_p]\otimes_{\cO/\varpi^{s'}[K_0/U'_p]}\cC(U'_pU^p,s'))\]
	 for $s' \ge s$ and $U_p' \subset U_p$ and the natural identification 
	\[\cO/\varpi^{s}[K_0/U_p]\otimes_{\cO/\varpi^{s'}[K_0/U'_p]}\cC(U'_pU^p,s') 
	\cong \cC(U_pU^p,s).\]
	
	 We equip $\TT^S(U^p)$ 
	with the inverse limit
        topology.
\end{defn}

\begin{rem}\label{bigheckeactsalllevels}
Now suppose that $U_p$ is any compact open subgroup of $K_0$ (not necessarily 
normal) and $s \ge 1$. Let $V_p$ be a compact open normal subgroup of $U_p$ 
which is also normal in $K_0$. Then the natural map $\TT^S(U^p)\rightarrow 
\End_{D(\cO/\varpi^s[K_0/V_p])}(\cC(V_pU^p,s))$ induces a map 
$\TT^S(U^p)\rightarrow \End_{D(\cO/\varpi^s[U_p/V_p])}(\cC(V_pU^p,s))$	and 
therefore induces a natural map $\TT^S(U^p)\rightarrow 
\End_{D(\cO/\varpi^s)}(\cC(U_pU^p,s))$, using the identification 
\[\cO/\varpi^{s}\otimes_{\cO/\varpi^{s}[U_p/V_p]}\cC(V_pU^p,s) 
\cong \cC(U_pU^p,s).\]
\end{rem}

For each $U$ and $s$, $\TT^S(U,s)$ is a finite $\cO$-algebra, since  $\cC(U,s)$ 
is perfect as a complex of $\cO/\varpi^s[K_0/U_p]$-modules. Moreover, 
the 
natural map $\TT^S(U,s) \rightarrow 	
\End_{\cO}(H_*(\cC(U,s)))$ has nilpotent kernel by~\cite[Lem.\ 2.5]{1409.7007}, and therefore $\TT^S(U,s)$ is 
a finite ring.

\begin{rem}\label{rem: O-Hecke algebras at finite level}
          Similarly, for each compact open normal subgroup $U_p$ of
        $K_0$, we can define  \[\TT^S(U_pU^p) =
          \invlim_{s}\TT^S(U_pU^p,s).\] Then
        $\TT^S(U_pU^p)$ is a finite $\cO$-algebra, and we have
        $\TT^S(U^p) = \invlim_{U_p}\TT^S(U_pU^p)$, equipped with the
        inverse limit topology (where each $\TT^S(U_pU^p)$ has its
        natural $p$-adic topology). \end{rem}

The big Hecke algebra $\TT^S(U^p)$ is naturally equipped with a map 
\[\TT^S(U^p) \rightarrow 
\End_{\cO[[K_0]]}(\widetilde{H}_i(X_{U^p},\cO))\]which 
commutes with the action of 
 $\prod_{v|p}G(F_v)$.\begin{lemma}\label{lem: big Hecke is 
 semilocal}
 The profinite $\cO$-algebra $\TT^S(U^p)$ is semilocal. Denote its finitely 
 many 
 maximal 
 ideals by $\m_1,\ldots,\m_r$ and let $J = J(\TT^S(U^p)) = \cap_{j=1}^r\m_j$ 
 denote the Jacobson radical. Then $\TT^S(U^p)$ is $J$-adically complete and 
 separated, and we have \[\TT^S(U^p) = \TT^S(U^p)_{\m_1}\times\cdots\times 
 \TT^S(U^p)_{\m_r}.\] For each maximal ideal $\m$ of $\TT^S(U^p)$, the 
 localisation $\TT^S(U^p)_{\m}$ is an $\m$-adically 
 complete and separated local ring with residue field a finite extension of $k$.
\end{lemma}
\begin{proof}
 First we note that if $U_p$ is a pro-$p$ group and $V_p$ is a normal open 
 subgroup of $U_p$, then for each $s \ge 1$ the surjective map \[ 
 \TT^S(V_pU^p,s)\rightarrow \TT^S(U_pU^p,1)\] induces a bijection of maximal 
 ideals.
 Indeed, we have (by~\cite[Thm.\ 5.6.4]{weibel}) a spectral sequence of $\TT^S(V_pU^p,s)$-modules \[E^2_{i,j}:
 \mathrm{Tor}_i^{\cO/\varpi^s[U_p/V_p]}(k,H_j(\cC(U^pV_p,s))) \Rightarrow 
 H_{i+j}(\cC(U^pU_p,1)).\] Localising at a maximal ideal $\m$ of 
 $\TT^S(V_pU^p,s)$ and considering the largest $q$ such that 
 $H_q(\cC(U^pV_p,s))_{\m}$ is non-zero shows that 
 $\m$ is in the support of $H_q(\cC(U^pU_p,1))$, and therefore $\m$ is the 
 inverse image of a maximal ideal in $\TT^S(U_pU^p,1)$. (Here we have
 used that  $\TT^S(U,s) \rightarrow 	
\End_{\cO}(H_*(\cC(U,s)))$ has nilpotent kernel.)
 
 Now it is not hard to show that the maximal ideals of $\TT^S(U^p)$ are in 
 bijection with the maximal ideals of $\TT^S(U_pU^p,1)$. Indeed, we have shown 
 that for every open $V_p \triangleleft U_p$ and $s \ge 1$ the kernel of \[ 
 \TT^S(V_pU^p,s)\rightarrow \TT^S(U_pU^p,1)\] is contained in the Jacobson 
 radical 
 of $\TT^S(V_pU^p,s)$. If $x \in \TT^S(U^p)$ maps to a unit in 
 $\TT^S(V_pU^p,s)$ 
 for every open $V_p \triangleleft U_p$ and $s \ge 1$ then $x$ is a unit. We 
 deduce that the kernel of \[ 
 \TT^S(U^p)\rightarrow \TT^S(U_pU^p,1)\] is contained in the Jacobson radical 
 $J$ 
 of 
 $\TT^S(U^p)$, and it follows that $\TT^S(U^p)$ is semilocal.
 
 For every open $V_p \triangleleft U_p$ and $s \ge 1$ the image of $J$ in 
 $\TT^S(V_pU^p,s)$ is nilpotent. It follows that $\TT^S(U^p)$ is $J$-adically 
 complete and separated. The remainder of the lemma follows from \cite[Theorem 
 8.15]{MR1011461}.
\end{proof}

\subsection{Ultrafilters}\label{sec:ultra}

In this section we let $A$ be a commutative finite (cardinality) local ring of characteristic 
$p$, denote the maximal ideal of $A$ by $\m_A$, and let $k = A/\m_A$. We 
let $B$ be a finite (but possibly non-commutative) augmented $A$-algebra. 
Denote the augmentation ideal $\ker(B\rightarrow A)$ by $\ba$. The example we 
have in mind is $B = A[\Gamma]$ where 
$\Gamma$ is a finite group.

Given an index set $I$, we define $A_I = \prod_{i \in I}A$, and similarly 
$B_I = 
\prod_{i \in I}B$. $B_I$ is an augmented $A_I$-algebra, with augmentation ideal 
$\ba_I = \prod_{i \in I}\ba = \ker(B_I\rightarrow A_I)$. Note that $\ba_I$ is a 
finitely generated ideal of $B_I$, so $A_I$ is finitely presented as a 
$B_I$-module. More generally, if $\bb \subset B$ is a two-sided ideal which 
contains $\ba$, and $\bb_I = \prod_{i \in I}\bb$, then $B_I/{\bb_I} = 
(B/\bb)_I$ is finitely presented as a $B_I$-module.

\begin{remark}
	If $B = A[\Gamma]$ then we have $B_I = A_I[\Gamma]$.
\end{remark} 

\begin{lemma}\label{lem:ultrafilter ideals}
	$\Spec(A_I)$ can be naturally identified with the set of ultrafilters on 
	$I$.  We have $A_{I,x} =
	A$ for each $x \in \Spec(A_I)$. We also have $A_{I,x} \otimes_{A_I} B_I = 
	B$.
\end{lemma}
\begin{proof}
	The bijection between ultrafilters and prime ideals is given
        by taking an ultrafilter
	$\gF$ to the ideal whose elements $(a_i)$ satisfy $\{i:a_i \in \m_A\}\in 
	\gF$. Since the map $A_I \rightarrow k_I$ has nilpotent kernel, the 
	fact that this gives a bijection follows from the case when $A$ is a field 
	\cite[Lemma 8.1]{scholze}.
	
For $x \in \Spec(A_I)$ the associated ultrafilter $\gF_x$ induces a map $B_I 
\rightarrow 
B$ by sending $(b_i)_{i \in I} \mapsto b$ where $b\in B$ is the unique
element with the property that $\{i: b_i = b\} \in \gF_x$. Since $B_I = A_I 
\otimes_A B$ (because $B$ is finitely presented as an $A$-module), this map 
induces an isomorphism $A_{I,x} 
\otimes_{A_I} B_I 
\cong B.$
\end{proof}

We have a natural inclusion $I \subset \Spec(A_I)$ given by taking the 
principal ultrafilter associated to an element of $I$. Given a point $x \in 
\Spec(A_I)\backslash I$  and a set of chain 
complexes of 
$B$-modules $\{\cC(i)\}_{i \in I}$, we define a chain complex of $B$-modules
\[\cC(\infty):= 
A_{I,x}\otimes_{A_I}\left(\prod_{i \in I}\cC(i)\right).\]

\begin{lemma}\label{patchingflat}
	Let $\{\cC(i)\}_{i \in I}$ be a set of chain complexes of flat 
	$B$-modules. Then $\prod_{i\in I}\cC(i)$ is a chain complex of flat 
	$B_I$-modules and $\cC(\infty)$ is a chain complex of flat $B$-modules.
\end{lemma}
\begin{proof}
	The fact that $\prod_{i\in I}\cC(i)$ is a chain complex of flat 
	$B_I$-modules follows from \cite[Thm.~1.13]{sweedler} (condition (d) in 
	Sweedler's Theorem is automatically satisfied because $B$ is a finite ring). 
	We deduce immediately that the localisation $\cC(\infty)$ is also a chain 
	complex of flat $B$-modules.
\end{proof}
\begin{lemma}\label{patchingmodaug}
	Let $\{\cC(i)\}_{i \in I}$ be a set of chain complexes of $B$-modules. Let 
	$\bb \subset B$ be a two-sided ideal which contains $\ba$, and
        let $\{\overline{\cC}(i)\}_{i\in I} = \{(B/\bb)\otimes_B
        \cC(i)\}_{i\in I}$. Then we have a 
	natural 
	isomorphism \[(B/\bb)\otimes_B \cC(\infty) = \overline{\cC}(\infty).\]
\end{lemma}
\begin{proof}
	We have \[(B/\bb)\otimes_B \cC(\infty) = 
	(B/\bb)_{I,x}\otimes_{(B/\bb)_I}(B/\bb)_I\otimes_{B_I}\prod_{i \in 
	I}\cC(i).\] Since $(B/\bb)_I$ is 
	finitely presented as a (right) $B_I$-module, we have (by~\cite[Ex.\ I.\S 2.9]{MR1727221})
	\[(B/\bb)_I\otimes_{B_I}\prod_{i \in I}\cC(i) = \prod_{i \in 
	I}\overline{\cC}(i)\] and we obtain the desired equality.
\end{proof}
In the rest of this subsection, we are going to assume that $B$ is a 
\emph{local} $A$-algebra. The example we have in mind is $B= A[\Gamma]$ where 
$\Gamma$ is a finite $p$-group.
\begin{defn}
	Suppose $B$ is a \emph{local} $A$-algebra and fix a set $\{\cC(i)\}_{i \in 
	I}$ of perfect chain complexes of 
	$B$-modules. For each $i$ fix a minimal resolution $\cF(i)$ of
        $\cC(i)$. Suppose we have integers $a \le b$ and $D \ge 0$. We 
	say that the set $\{\cC(i)\}_{i \in I}$ has \emph{complexity bounded by} 
	$(a,b,D)$ if the minimal complexes 
	$\cF(i)$ are all concentrated in degrees between $a$ and $b$ and every term 
	in these complexes has rank $\le D$. 
	
	If there exists some $a,b,D$ such that $\{\cC(i)\}_{i \in I}$ has 
	complexity bounded by $(a,b,D)$, we say that $\{\cC(i)\}_{i \in I}$ has 
	\emph{bounded complexity}.
\end{defn}

\begin{lem}\label{productperf}Suppose $B$ is a local $A$-algebra, and let 
$\{\cC(i)\}_{i \in I}$ be a set of perfect 
chain complexes of 
	$B$-modules with bounded complexity. Then the complex $\prod_{i \in 
	I}\cC(i)$ is a perfect complex of $B_I$-modules. 
\end{lem}
\begin{proof}
	Fix a minimal resolution $\cF(i)$ of each perfect complex $\cC(i)$. Since 
	products are exact in the category of Abelian groups, it suffices to check 
	that the complex $\prod_{i \in I}\cF(i)$ is a bounded complex of finite 
	projective $B_I$-modules. Boundedness follows immediately from the bounded 
	complexity assumption. It remains to show that if we have a set $\{F_i\}_{i \in 
	I}$ of finite free $B$-modules with ranks all $\le D$, then the product 
	$\prod_{i \in I}F_i$ is a finite projective $B_I$-module.
	
	We have a decomposition $I = \coprod_{d=0}^D I_d$ such that $F_i \cong B^d$ 
	for $i \in I_d$. Then $M_d \cong \prod_{i \in I_d} B^d \cong B_{I_d}^d$ is 
	a finite free $B_{I_d}$-module. Each $M_d$ is a finite projective 
	$B_I$-module (they are direct summands of finite free modules), and we have 
	\[\prod_{i \in I}F_i = 
	\bigoplus_{d = 0}^D M_d,\] so $\prod_{i \in I}F_i$ is a finite projective 
	$B_I$-module, as required.
\end{proof}

\begin{cor}
  Let $x \in \Spec(A_I)\backslash I$ and suppose that $\{\cC(i)\}_{i \in I}$ is a set of
  perfect chain complexes of $B$-modules with
  bounded complexity. Then $\cC(\infty)$ is a perfect
  complex of $B$-modules.
\end{cor}
\begin{proof}
  This follows from Lemmas~\ref{lem:ultrafilter
    ideals} and~ \ref{productperf}.
\end{proof}

\begin{rem}\label{rem: patched minimal resolutions}
  In fact there is another way of phrasing the proof that this complex
  is perfect. If we fix $a,b$ and $D$ then there are finitely many
  isomorphism classes of minimal complex with complexity bounded by
  $(a,b,D)$ (since $B$ is a finite ring). Let
  $x \in \Spec(A_I)\backslash I$, corresponding to the non-principal
  ultrafilter $\gF$ on $I$. Then there is an $I' \in \gF$ such that
  the minimal resolutions of $\cC(i)$ are isomorphic for all
  $i \in I'$. We can therefore take a single minimal complex
  $\cF(\infty)$ which is a minimal resolution of $\cC(i)$ for all
  $i \in I'$. We then have a quasi-isomorphism of complexes of
  $B_{I'}$-modules:
  \[A_{I'}\otimes_A\cF(\infty) \rightarrow
    A_{I'}\otimes_{A_I}\left(\prod_{i \in I}\cC(i)\right) = \prod_{i
      \in I'}\cC(i)\] which induces a quasi-isomorphism
  \[\cF(\infty) \rightarrow \cC(\infty),\] so that~$\cF(\infty)$ is a
  minimal resolution of~$\cC(\infty)$.
\end{rem}

\section{Patching II: Galois representations and Taylor--Wiles primes}\label{sec:patchingII}
\subsection{Deformation theory}
We fix a continuous
absolutely irreducible representation $\rhobar:G_F\to\GL_n(k)$. We assume
from now on that
$p>n\ge 2$. Fix also a continuous character $\mu:G_F\to\cO^\times$
lifting $\det\rhobar$, and a finite set of finite places $S$ of~$F$,
which contains the set~$S_p$ of places of~$F$ lying over~$p$, as well
as the places at which~$\rhobar$ or~$\mu$ are ramified.

For each $v\in S$, we fix a ring $\Lambda_v\in\CNL_\cO$, let
$\cD_v^\square:\CNL_{\Lambda_v}\to\Sets$ be the functor associating
to~$R\in\CNL_{\Lambda_v}$ the set of all continuous liftings of
$\rhobar|_{G_{F_v}}$ to~$\GL_n(R)$ which have
determinant~$\mu|_{G_{F_v}}$. This is represented by the universal
lifting ring $R_v\in\CNL_{\Lambda_v}$. We let $\Lambda = 
\widehat{\bigotimes}_{v \in S, \cO}\Lambda_v \in \CNL_\cO$.

Then as in~\cite[\S 4]{1409.7007} we have the following notions.
\begin{itemize}
\item For $v\in S$, a \emph{local deformation problem} for
  $\rhobar|_{G_{F_v}}$ is a subfunctor $\cD_v\subset \cD_v^\square$
  which is stable under conjugation by elements of $\ker(\GL_n(R)\to\GL_n(k))$, and is represented by a
  quotient~$\Rbar_v$ of~$R_v^\square$.
\item A \emph{global deformation problem} is a
  tuple \[\cS=(\rhobar,\mu,S,\{\Lambda_v\}_{v\in S},\{\cD_v\}_{v\in
    S})\]consisting of the objects defined above.
\item If $R\in\CNL_\Lambda$, then a lifting of $\rhobar$ to a
  continuous homomorphism $\rho:G_F\to\GL_n(R)$ is of type~$\cS$ if it
  is unramified outside~$S$, has determinant~$\mu$, and for each $v\in
  S$, $\rho|_{G_{F_v}}$ is in $\cD_v(R)$.
\item We say that two liftings are \emph{strictly equivalent} if they
  are conjugate by an element of $\ker(\GL_n(R)\to\GL_n(k))$.
\item If $T\subset S$ and $R\in\CNL_\Lambda$, then a \emph{$T$-framed
    lifting of $\rhobar$ to~$R$} is a tuple $(\rho,\{\alpha_v\}_{v\in
    T})$ where $\rho$ is a lifting of $\rhobar$ to a
  continuous homomorphism $\rho:G_F\to\GL_n(R)$, and each $\alpha_v$
  is an element of $\ker(\GL_n(R)\to\GL_n(k))$. Two $T$-framed
  liftings $(\rho,\{\alpha_v\}_{v\in
    T}), (\rho',\{\alpha'_v\}_{v\in
    T})$ are \emph{strictly equivalent} if there is an element
  $a\in\ker(\GL_n(R)\to\GL_n(k))$ such that $\rho'=a\rho a^{-1}$ and
  each $\alpha'_v=a\alpha_v$.
\item The functors of liftings of type~$\cS$, strict equivalences of
  liftings of type~$\cS$, and strict equivalence classes of $T$-framed
  liftings of type~$\cS$, are representable by objects
  $R^\square_\cS$, $R_\cS$, $R^T_\cS$ respectively of
  $\CNL_\Lambda$. (See~\cite[Thm.\ 4.5]{1409.7007}.)
\end{itemize}

Write~$\Lambda_T:=\widehat{\otimes}_{v\in T,\cO}\Lambda_v$. For each $v \in S$, let $\Rbar_v \in \CNL_{\Lambda_v}$ denote the
representing object of $\cD_v$, and write
$R_\cS^{T,\loc} := \widehat{\otimes}_{v \in T,\cO} \Rbar_v$. The natural
transformation
$(\rho, \{ \alpha_v \}_{v \in T}) \mapsto ( \alpha_v^{-1}
\rho|_{G_{F_v}} \alpha_v)_{v \in T}$
induces a canonical homomorphism of $\Lambda_T$-algebras
$R_\cS^{T,\loc} \to R_\cS^T$.

\subsection{Enormous image}\label{subsec: enormous}Let
$H\subset\GL_n(k)$  be a subgroup which acts irreducibly on the
natural representation. We assume that~$k$ is chosen large enough to
contain all eigenvalues of all elements of~$H$.
\begin{defn}\label{defn:enormous image}
We say that $H$ is enormous if it satisfies the following conditions:
\begin{enumerate}
\item $H$ has no non-trivial $p$-power order quotient.
\item $H^0(H, \ad^0) = H^1(H, \ad^0) = 0$ (for the adjoint action of $H$).
\item For all simple $k[H]$-submodules $W \subset \ad^0$, there is an
  element $h \in H$ with $n$ distinct eigenvalues and $\alpha \in k$
  such that $\alpha$ is an eigenvalue of $h$ and
  $\tr e_{h, \alpha} W \neq 0$, where $e_{h, \alpha} \in M_n(k) = \ad$
  denotes the unique $h$-equivariant projection onto the
  $\alpha$-eigenspace of $h$.
\end{enumerate}
\end{defn}

\begin{rem}
  \label{rem: big adequate enormous}By definition, an enormous
  subgroup is big in the sense of~\cite[Defn.\ 2.5.1]{cht}, and thus
  adequate in the sense of~\cite[Defn.\ 2.3]{jack}. Indeed, the only
  differences between these notions is that in the definition of big,
  the condition that $h$ has $n$ distinct eigenvalues is relaxed to
  demanding that the generalised eigenspace of $\alpha$ is
  one-dimensional, and in the definition of adequate, it is further
  relaxed to ask only that $\alpha$ is an eigenvalue of~$h$ (but the
  definition of $e_{h, \alpha}$ is now the projection onto the
  generalised eigenspace for~$\alpha$).
\end{rem}
\begin{lem}
  \label{lem: big adequate enormous n=2}If $n=2$, the notions of
  enormous, big and adequate are all equivalent. In particular, if
   $H$ acts irreducibly on $k^2$, then $H$ is enormous unless $p=3$ or
   $p=5$,
   and the image of $H$ in $\PGL_2(k)$ is conjugate to $\PSL_2(\Fp)$.
\end{lem}
\begin{proof}
  The second statement follows from the first statement
  and~\cite[Prop.\ A.2.1]{blggu2}. By Remark~\ref{rem: big adequate
    enormous}, it is therefore enough to show that
if we have a simple  $k[H]$-submodule $W \subset \ad^0$ and an element
$h\in H$  with an eigenvalue~$\alpha$ such that   $\tr e_{h, \alpha} W
\neq 0$, then $h$ necessarily has distinct eigenvalues. If not, then
$e_{h,\alpha}=1$ by definition (as $e_{h,\alpha}$ is projection onto
the generalised eigenspace for~$\alpha$), which is a contradiction as $W \subset \ad^0$.
\end{proof}
We now give two examples of classes of enormous subgroups
of~$\GL_n(k)$ when~$n>2$, following~\cite[\S 2.5]{cht} (which shows
that the same groups are big).

\begin{lem}
  \label{SLn big}If $n>2$ and there is a subfield $k'\subset k$ such that
  $k^\times\GL_n(k')\supset H\supset\SL_n(k')$, then $H$ is enormous.
\end{lem}
\begin{proof}
  Examining the proof of~\cite[Lem.\ 2.5.6]{cht} (which shows that~$H$
  is big), we see that it is enough to check that $\SL_n(k')$ contains
  an element with~$n$ distinct eigenvalues. Since we are assuming
  that~$p>n$, we can use an element with characteristic
  polynomial~$X^n+(-1)^n$ (for example, the matrix~$(a_{ij})$
  with~$a_{i+1,i}=1$, $a_{1,n}=(-1)^{n-1}$, and all other $a_{ij}=0$).
\end{proof}
\begin{lem}
  \label{Sym sl2 big}If $p>2n+1$ and there is a subfield $k'\subset k$
  such that
  $k^\times\Sym^{n-1}\GL_2(k')\supset H\supset\Sym^{n-1}\SL_2(k')$,
  then $H$ is enormous.
\end{lem}
\begin{proof}
  The proof of~\cite[Cor.\ 2.5.4]{cht} (which shows that~$H$ is big)
  in fact shows that~$H$ is enormous (note that in the proof
  of~\cite[Lem.\ 2.5.2]{cht} it is shown that the eigenspaces of the
  element denoted~$t$ are 1-dimensional). (Note also that as explained
  after~\cite[Prop.\ 2.1.2]{BLGGT}, the hypothesis that $p>2n-1$
  in~\cite[Cor.\ 2.5.4]{cht} should be~$p>2n+1$.)
\end{proof}

\subsection{Taylor--Wiles primes}\label{subsec: TW primes}
Suppose that $v$ is a finite place of~$F$ such that $\# k(v) \equiv 1 \pmod{p}$, that $\rhobar|_{G_{F_v}}$
is unramified, and that $\rhobar(\Frob_v)$ has $n$ distinct
eigenvalues $\gamma_{v, 1}, \dots, \gamma_{v, n} \in k$. Let
$\Delta_v = (k(v)^\times(p))^{n-1}$ (where $k(v)^\times(p)$ is the Sylow
$p$-subgroup of $k(v)^\times$),  and let
$\Lambda_v = \cO[\Delta_v]$.

We define $\cD_v^\mathrm{TW}$ to be the functor of liftings over
$R \in \CNL_{\Lambda_v}$ of the form
\[ r \sim \chi_1 \oplus \dots \oplus \chi_n, \]
where
$\chi_1, \dots, \chi_n : G_{F_v} \to R^\times$  are continuous
characters such that  for each $i = 1, \dots, n-1$, we have
\begin{itemize}
\item $(\chi_i \pmod{\m_R})(\Frob_v) = \gamma_{v, i}$, and
\item   $\chi_i|_{I_{F_v}}$ agrees, on composition with the Artin map, with
  the $i^\text{th}$ canonical character $k(v)^\times(p) \to R^\times$.
\end{itemize}

(This definition depends on the ordering of the $\gamma_{v,i}$, but
this does not affect any of our arguments.) The functor
$\cD_v^\mathrm{TW}$ is represented by a formally smooth
$\Lambda_v$-algebra.

Suppose that
$\cS = (\rhobar, \mu, S, \{ \Lambda_v \}_{v \in S}, \{ \cD_v \}_{v \in
  S})$
is a deformation problem. Let $Q$ be a set of places disjoint from $S$
of the form considered above (that is,
$\# k(v) \equiv 1 \pmod{p}$ and $\rhobar(\Frob_v)$ has $n$
distinct eigenvalues). We refer to the tuple
\[(Q, (\gamma_{v, 1}, \dots, \gamma_{v, n})_{v \in Q})\] as a
\emph{Taylor--Wiles datum}, and define the augmented deformation problem
\[ \cS_Q = (\rhobar, \mu, S \cup Q, \{ \Lambda_v \}_{v \in S} \cup \{ \cO[\Delta_v] \}_{v \in Q}, \{ \cD_v \}_{v \in S} \cup \{ \cD_v^\mathrm{TW} \}_{v \in Q}). \]
Let $\Delta_Q = \prod_{v \in Q} \Delta_v = \prod_{v \in Q} k(v)^\times(p)^{n-1}$. Then $R_{\cS_Q}$ is naturally a $\cO[\Delta_Q]$-algebra. If $\ba_Q \subset \cO[\Delta_Q]$ is the augmentation ideal, then there is a canonical isomorphism $R_{\cS_Q}/\ba_Q \cong R_\cS$. 

Recall that~$\rhobar$ is \emph{totally odd} if if for each complex conjugation $c \in G_F$, we have 
\[ \overline{\rho}(c) \sim \diag(\underbrace{1, \dots, 1}_a, \underbrace{-1, \dots, -1}_b), \]
with $|a - b| \leq 1$. (Of course, if~$F$ is totally complex, this is
a vacuous condition.) Let~$l_0$ be the integer defined in~(\ref{eqn:
  defect})  (which
only depends on $F$ and~$n$).
\begin{lem}
  \label{lem: existence of TW primes}Let
  $(\rhobar, \mu, S, \{ \Lambda_v \}_{v \in S}, \{ \cD_v \}_{v \in
    S})$
  be a global deformation problem. Suppose that:
  \begin{itemize}
  \item $\rhobar$ is totally odd.\item 
  $\rhobar\not\cong\rhobar\otimes\epsilonbar$.\item $\rhobar({G_{F(\zeta_p)}})$ 
  is enormous.
  \end{itemize}

   Then for every $q\gg 0$ and every $N \geq 1$,
 there exists a Taylor--Wiles datum
  $(Q_N, (\gamma_{v, 1}, \dots, \gamma_{v, n})_{v \in Q_N})$
  satisfying the following conditions:
\begin{enumerate}
\item $\# Q_N = q$.
\item For each $v \in Q_N$, $q_v \equiv 1 \pmod{p^N}$.
\item The ring $R^{S}_{\cS_{Q_N}}$ is a quotient $R_\cS^{S,\loc}$-algebra of $R_\infty := R_\cS^{S,\loc} \llbracket X_1, \dots, X_g \rrbracket$, where 
\[ g = (n-1)q - n(n-1)[F : \Q]/2 - l_0 - 1 + \# S .\]
\end{enumerate}
\end{lem}\begin{proof}
  This follows from~\cite[Lem.\ 4.12]{1409.7007} and a standard argument using 
  Poitou--Tate duality, compare the proof of \cite[Thm.~6.29]{1409.7007}. 
\end{proof}Fix a choice of place $v_0\in T$ and an integer~$q\gg 0$ as in Lemma~\ref{lem: existence of TW primes}, and set~$\cT=\cO[[X_v^{i,j}]]_{v\in
  S,1\le i,j\le
  n}/(X_{v_0}^{1,1})$. Set $\Delta_{Q_N}:=\prod_{v\in Q_N}\Delta_v$,
$\cO_N:=\cT[\Delta_{Q_N}]$, and~$\cO_{\infty}:=\cT[[\Delta_{\infty}]]$,
where~$\Delta_\infty=\Zp^{(n-1)q}$. For each~$N$ we fix a
surjection~$\Delta_\infty\onto\Delta_N$,  and thus a surjection of~$\cT$-algebras
~$\cO_\infty\onto\cO_N$.

We now examine the behaviour of the Hecke operators at Taylor--Wiles primes. Fix~$U^p$ such that ~$U^pK_0$
is $S$-good. We begin by setting up some notation. Let $(Q, (\gamma_{v, 1}, \dots, \gamma_{v, n})_{v \in Q})$ be a
Taylor--Wiles datum. We define compact open subgroups
$U^p_0(Q)=\prod_{v\nmid p}U_0(Q)_v$ and $U^p_1(Q)=\prod_{v\nmid
  p}U_1(Q)_v$ of $U^p=\prod_{v\nmid p}U_v$ by:
\begin{itemize}
\item if $v\notin Q$, then $U_0(Q)_v=U_1(Q)_v=U_v$.
\item If $v\in Q$ then $U_0(Q)_v$ is the standard Iwahori subgroup of
  $\PGL_n(\cO_{F,v})$, and $U_1(Q)_v$ is the minimal subgroup of
  $U_0(Q)_v$ for which $U_0(Q)_v/U_1(Q)_v$ is a $p$-group.\end{itemize}
In particular~$U_1(Q)_v$ contains the pro-$v$ Iwahori subgroup
of~$U_0(Q)_v$, so we can and do identify~$\prod_{v\in
  Q}U_0(Q)_v/U_1(Q)_v$ with~$\Delta_Q$. We now introduce some natural
variants of the Hecke algebras that we introduced in
Section~\ref{subsubsec: big hecke algebras}. 

 For each compact open normal
subgroup~$U_p$ of $K_0$ we  define $\TT^{S\cup
  Q,Q}(U_pU_0^p(Q),s)$ to be the image in 
  $\End_{D(\cO/\varpi^s[K_0/U_p])}(\cC(U_pU_0^p(Q),s))$ of the abstract Hecke 
	algebra $\TT^{S\cup Q,Q}$ generated by the operators $T^i_v$ for $v \notin S\cup
        Q$
        and~$\mathbf{U}^i_v$ for~$v\in Q$, where the operators~
        $\mathbf{U}^i_v$ act as explained in Section~\ref{subsubsec: Hecke operators on complexes}. Similarly, we let $\TT^{S\cup
  Q,Q}(U_pU_1^p(Q),s)$ be the image of~$\TT^{S\cup Q,Q}$ in
$\End_{D(\cO/\varpi^s[\Delta_Q\times K_0/U_p])}(\cC(U_pU_1^p(Q),s))$ (as 
explained in
Section~\ref{subsubsec: Hecke operators on complexes}, the
operators~$\mathbf{U}^i_v$ commute with the action
of~$\Delta_Q$).

Note that we have a natural isomorphism of complexes \numequation\label{eqn: going from level 1 to level 0}\cC(U_pU_1^p(Q),s)\otimes_{\cO[\Delta_Q]}\cO\cong\cC(U_pU_0^p(Q),s).\end{equation}
  We then set (for each compact open normal subgroup~$U_p$ of $K_0$)	
  \[\TT^{S\cup Q,Q}(U_pU^p_i(Q)) = \invlim_{s}\TT^{S\cup
             Q,Q}(U_pU^p_i(Q),s), \]
	 \[\TT^{S\cup Q,Q}(U^p_i(Q)) = \invlim_{U_p,s}\TT^{S\cup
             Q,Q}(U_pU^p_i(Q),s),\] for $i=0,1$, equipped with their inverse limit topologies.
We now need to assume the existence of Galois representations
associated to completed homology, as in the following conjecture.
\begin{conj}\label{galconj}
	Let $\m \subset \TT^S(U^p)$ be a maximal ideal with residue field $k$. 
	
	\begin{enumerate}
		\item There exists a continuous semi-simple representation 
		\[\rhobar_\m: G_{F,S} \rightarrow \GL_n(\TT^S(U^p)/\m)\] satisfying the 
		following conditions: $\rhobar_\m$ is totally odd, and for any finite place $v \notin S$ of $F$, 
		$\rhobar_\m(\Frob_v)$ has characteristic polynomial \[X^n - T_v^1 
		X^{n-1} + \cdots + 
		(-1)^iq_v^{i(i-1)/2}T_v^iX^{n-i}+\cdots+(-1)^nq_v^{n(n-1)/2}T_v^n \in 
		(\TT^S(U^p)/\m)[X]\] 

		\item Suppose that $\rhobar_\m$ is absolutely irreducible. Then there 
		exists a lifting of $\rhobar_\m$ to a continuous homomorphism 
		\[\rho_\m: G_{F,S} \rightarrow \GL_n(\TT^S(U^p)_\m)\] satisfying the 
		following condition: for any finite place $v \notin S$
                of $F$, $\rho_\m(\Frob_v)$ has characteristic
                polynomial
                \[X^n - T_v^1 X^{n-1} + \cdots +
                  (-1)^iq_v^{i(i-1)/2}T_v^iX^{n-i}+\cdots+(-1)^nq_v^{n(n-1)/2}T_v^n
                  \in \TT^S(U^p)_\m[X]\] \emph{(}In particular, since for each
                $v\notin S$ we have $T_v^n=1$, we have
                $\det\rho_\m=\epsilon^{n(1-n)/2}$.\emph{)}
            	\end{enumerate}
\end{conj}
\begin{remark}\label{rem: existence of Galois repns in CM case}
	If $F$ is a CM or totally real field, the first part of the conjecture 
	holds by the main results of~\cite{scholze-torsion} and~\cite{2014arXiv1409.2158C}. It also follows from Scholze's work (again 
	with the 
	assumption that $F$ is CM or totally real) that there is a 
	lifting of $\rhobar_\m$ valued in $\TT^S(U^p)_\m/I$ for some nilpotent 
	ideal $I \subset \TT^S(U^p)_\m$, and in fact we may assume that $I^4 = 0$ 
	by \cite[Theorem 1.3]{new-tho}. Moreover, the nilpotent ideal has been 
	eliminated entirely when $F$ is CM and $p$ splits completely in $F$ 
	\cite{arizona}.
\end{remark}
\begin{defn}
	Let $\m$ be a maximal ideal of $\TT^S(U^p)$. For sufficiently small $U_p$ 
	(for example, if $U_p$ is pro-$p$), 
	$\m$ is the inverse image of a maximal ideal of $\TT^S(U_pU^p,1)$, which we 
	also denote by $\m$. The abstract Hecke algebra $\TT^S$ surjects onto 
	$\TT^S(U_pU^p,1)$ and we again denote by $\m$ the inverse image of $\m$ in 
	$\TT^S$. 
	
	Finally, for any module $M$ for $\TT^S$ (or complex of such modules) we 
	denote by $M_\m$ the localisation $\TT^S_\m
        \otimes_{\TT^S}M$. Note that the idea of patching singular chain 
        complexes localised with respect to the action of the abstract Hecke 
        algebra appears in \cite{hansen}.

We make an analogous definition for maximal ideals of the Hecke
algebras $\TT^{S\cup
  Q}(U_i^p(Q))$ and~$\TT^{S\cup Q,Q}(U_i^p(Q))$.
\end{defn}

We assume Conjecture~\ref{galconj} from now on, and recall that we have 
fixed~$U^p$ such that ~$U^pK_0$
is $S$-good. We now fix a maximal ideal $\m$ of $\TT^S(U^p)$, and
assume that
\begin{itemize}
	\item~$\rhobar_\m$ is absolutely irreducible
	\item~$\rhobar_\m(G_{F(\zeta_p})$ is enormous, and 
	\item~$\rhobar_\m\not\cong\rhobar_\m\otimes\epsilonbar$.
\end{itemize} 

Enlarging our coefficient field~$E$ if necessary, we
assume further that~$\m$ has residue field~$k$, and that~$k$ contains
the eigenvalues of all elements of the image of~$\rhobar_\m$.  We
fix \[\cS=(\rhobar_\m,\epsilon^{n(1-n)/2},S,\{\cO\}_{v\in
                  S},\{\cD_v^\square\}_{v\in
                  S}).\]

The following is the analogue of~\cite[Prop.\ 6.26]{1409.7007} in our
context, and the proof is essentially identical.

\begin{prop}\label{prop: Hecke action at TW primes}Let $(Q, (\gamma_{v, 1}, \dots, \gamma_{v, n})_{v \in Q})$ be a
Taylor--Wiles datum.
\begin{enumerate}
\item There are natural inclusions $\TT^{S\cup Q}(U^p)\subset \TT^S(U^p)$ and $\TT^{S\cup
  Q}(U^p_0(Q))\subset \TT^S(U^p_0(Q))$, and natural surjections $\TT^{S\cup
  Q}(U^p_0(Q)) \onto \TT^{S\cup Q}(U^p)$, $\TT^{S\cup Q}(U^p_1(Q))\onto
\TT^{S\cup Q}(U^p_0(Q))$ and $\TT^{S\cup Q,Q}(U^p_1(Q))\onto
\TT^{S\cup Q,Q}(U^p_0(Q))$.
\item Let $\m_{Q,0}\subset \TT^{S\cup Q,Q}(U^p_0(Q))$ denote the ideal generated by the
  pullback of~$\m$ to $\TT^{S\cup
    Q}(U^p_0(Q))$ and the elements
  $\mathbf{U}^i_v-\prod_{j=1}^i\gamma_{v,i}$. Then $\m_{Q,0}$ is a
  maximal ideal.
\item Write $\m_{Q,1}$ for the pullback of~$\m_{Q,0}$ to $\TT^{S\cup
    Q,Q}(U^p_1(Q))$, and~$\m'$ for the pullback of~$\m$ to $\TT^{S\cup
  Q}(U^p)$. Then there is a quasi-isomorphism \[
    \cC(U_pU_0^p(Q),s)_{\m_{Q,0}}\to \cC(U_pU^p,s)_{\m}\] and an isomorphism \[
  \cC(U_pU_1^p(Q),s)_{\m_{Q,1}}\otimes_{\cO[\Delta_Q]}\cO\cong\cC(U_pU_0^p(Q),s)_{\m_{Q,0}}\]
which are both equivariant for the actions of the operators~$T_v^i$,
$v\notin S\cup Q$. Consequently,  if we write~$\TT^{S\cup
  Q}(U_pU^p_1(Q),s)_{\m_{Q,1}}$ for the $\cO[\Delta_Q]$-subalgebra
of~$\End_{D(\cO/\varpi^s[\Delta_Q\times 
K_0/U_p])}(\cC(U_pU_1^p(Q),s)_{\m_{Q,1}})$ generated
by the operators~$T_v^i$,
$v\notin S\cup Q$, then there are natural maps \[\TT^{S\cup
  Q}(U_pU^p_1(Q),s)_{\m_{Q,1}}\onto \TT^{S\cup Q}(U_pU^p,s)_{\m'}\cong
  \TT^{S}(U_pU^p,s)_\m .\]\end{enumerate}
\end{prop}
\begin{proof}The inclusions $\TT^{S\cup Q}(U^p)\subset \TT^S(U^p)$ and
  $\TT^{S\cup Q}(U^p_0(Q))\subset \TT^S(U^p_0(Q))$ exist by
  definition. The surjections
  $\TT^{S\cup Q}(U^p_1(Q))\onto \TT^{S\cup Q}(U^p_0(Q))$ and
  $\TT^{S\cup Q,Q}(U^p_1(Q))\onto \TT^{S\cup Q,Q}(U^p_0(Q))$ are induced
  by~(\ref{eqn: going from level 1 to level 0}), while the surjection
  $\TT^{S\cup Q}(U^p_0(Q)) \onto \TT^{S\cup Q}(U^p)$ comes from the
  splitting by the trace map of the natural
  map \[\cC(U_pU_0^p(Q),s)\to \cC(U_pU^p,s) \](note that for~$v\in Q$, since
  $p>n$ and~$\# k(v)\equiv 1\pmod{p}$, the index of $U_0(Q)_v$ 
  in $\PGL_n(\cO_{F,v})$ is congruent to $n!$ mod $p$, by the Bruhat 
  decomposition, and hence this index
  is prime to~$p$).

For the second part, we need to show that~$\m_{Q,0}$ is in the support
of~$\cC(U_pU_0^p(Q),1)$. As in the proof of~\cite[Lem.\
6.25]{1409.7007}, it is enough to prove the corresponding statement
for cohomology groups, which follows from~\cite[Lem.\
5.3]{1409.7007}.

The isomorphism
$\cC(U_pU_1^p(Q),s)_{\m_{Q,1}}\otimes_{\cO[\Delta_Q]}\cO\cong\cC(U_pU_0^p(Q),s)_{\m_{Q,0}}$
is induced by~(\ref{eqn: going from level 1 to level 0}). The 
quasi-isomorphism is the composite of quasi-isomorphisms  \[
    \cC(U_pU_0^p(Q),s)_{\m_{Q,0}}\to    \cC(U_pU^p,s)_{\m'}\to \cC(U_pU^p,s)_{\m}\] 
which are induced by the obvious natural maps of
complexes (and the morphisms of Hecke algebras from part~(1)); to see
that they are indeed quasi-isomorphisms, one uses
respectively~\cite[Lem.\ 5.4]{1409.7007} and the argument
of~\cite[Lem.\ 6.20]{1409.7007}. Finally the isomorphism $\TT^{S\cup 
  Q}(U_pU^p,s)_\m \cong \TT^{S}(U_pU^p,s)_\m$ again follows from the argument
of~\cite[Lem.\ 6.20]{1409.7007} and \cite[Cor.~3.4.5]{cht}.
\end{proof}
As usual, we set $\TT^{S\cup
  Q}(U^p_1(Q))_{\m_{Q,1}}:=\varprojlim_{U_p,s}\TT^{S\cup
  Q}(U_pU^p_1(Q),s)_{\m_{Q,1}}$,\\ $\TT^{S\cup
  Q}(U_pU^p_1(Q))_{\m_{Q,1}}:=\varprojlim_{s}\TT^{S\cup
  Q}(U_pU^p_1(Q),s)_{\m_{Q,1}}$, equipped with their inverse limit
topologies. (These are local rings, as can easily be checked as in the
proof of Lemma~\ref{lem: big Hecke is semilocal}.)
We will need to assume the following refinement of Conjecture~\ref{galconj}.
\begin{conj}\label{conj: compatibility at Q}
   Suppose that $\rhobar_\m$ is absolutely irreducible,
              and let~ $(Q, (\gamma_{v, 1}, \dots, \gamma_{v, n})_{v \in Q})$ be a
Taylor--Wiles datum. Then there exists a lifting of~$\rhobar_\m$ to a
continuous  homomorphism 
		\[\rho_{\m,Q}: G_{F,S\cup Q} \rightarrow \GL_n(\TT^{S\cup Q}(U^p_1(Q))_{\m_{Q,1}})\] satisfying the 
		following conditions: for any finite place $v \notin
                S\cup Q$ of $F$, 
		$\rho_{\m,Q}(\Frob_v)$ has characteristic polynomial \[X^n - T_v^1 
		X^{n-1} 
		+ \cdots + 
		(-1)^iq_v^{i(i-1)/2}T_v^iX^{n-i}+\cdots+(-1)^nq_v^{n(n-1)/2}T_v^n \in 
		\TT^{S\cup Q}(U^p_1(Q))_{\m_{Q,1}}[X]\] and~$\rho_{\m,Q}$ is of
              type~$\cS_Q$. 

\end{conj}
\begin{rem}\label{rem: compatibility at Q in CM case}
  The requirement that~$\rho_{\m,Q}$ be of type~$\cS_Q$ is a form of
  local-global compatibility at the places in~$Q$. If~$F$ is CM, this property 
  is verified in \cite{10author} (under a technical assumption which permits 
  the use of Shin's unconditional base change and up 
  to a nilpotent 
  ideal, see Remark~\ref{rem:
  	existence of Galois repns in CM case}).
\end{rem}
We assume Conjecture~\ref{conj: compatibility at Q} from now on, so
that in particular $\rho_{\m,Q}$ determines an~$\cO[\Delta_Q]$-algebra
homomorphism \numequation\label{eqn:R to T}R_{\cS_{Q}}\to \TT^{S\cup 
Q}(U^p_1(Q))_{\m_{Q,1}},\end{equation} and the
choice of~$\rho_{\m,Q}$ in its strict equivalent class determines an
isomorphism \numequation\label{eqn: framed deformation ring
  isomorphism}
R_{\cS_{Q}}^S\isoto\cT\wotimes_{\cO}R_{\cS_Q}.\end{equation} 

\subsection{Patching}
For each
$N\ge 1$, we let
$(Q_N,(\gamma_{v, 1}, \dots, \gamma_{v, n})_{v \in Q_N}$ be a choice
of Taylor--Wiles datum as in Lemma~\ref{lem: existence of TW primes} (for some fixed
choice of~$q\gg 0$). We fix a surjective $R_\cS^{S,\loc}$-algebra map $R_\infty 
\rightarrow R_{\cS_{Q_N}}^S$ for each $N$. We also fix a non-principal 
ultrafilter $\gF$ on
the set $\NN = \{N \ge 1\}$.

\begin{rem}
  \label{rem: the ultrafilter is the only choice we make}With the
  exception of Remark~\ref{liftOtoR}, the choice of~$\gF$ is the only
  choice we make in our patching argument. This has the pleasant
  effect of making many of the constructions below natural, although
  the reader should bear in mind that they are only natural relative
  to our fixed choice of~$\gF$.
\end{rem}

\begin{defn}

Let $U_p$ be a compact open subgroup of $K_0$, and let $J$ be an open 
ideal in 
$\cO_\infty$. Let~$I_J$ be the (cofinite) subset of~$N\in\NN$ such that $J$ contains the kernel of~$\cO_\infty 
\rightarrow \cO_{N}.$ For $N \in I_J$, we define \[\cC(U_p,J,N) 
= 
\cO_\infty/J\otimes_{\cO[\Delta_{Q_N}]}\cC(U^p_1(Q_N)U_p,\cO)_{\m_{Q_N,1}}.\]

\end{defn}
\begin{rem}\phantomsection\label{prepatchingremarks}
	\begin{enumerate}
\item We have a map $R^{S_p}_{\cS_{Q_N}}\rightarrow 
\cT\wotimes_{\cO}\TT^{S\cup Q_N}(U^p_1(Q_N))_{\m_{Q_N,1}}$  by 
~(\ref{eqn:R to T}) and~(\ref{eqn: framed deformation ring
  isomorphism}), and a map \[\cT\wotimes_{\cO}\TT^{S\cup 
  Q_N}(U^p_1(Q_N))_{\m_{Q_N,1}}\rightarrow 
\End_{D(\cO_\infty/J)}(\cC(U_p,J,N))\] by definition of~$\TT^{S\cup 
Q_N}(U^p_1(Q_N))_{\m_{Q_N,1}}$ together with Remark 
\ref{bigheckeactsalllevels}. In 
particular, for all $J$ and $N\in I_J$ we have a ring homomorphism \[R_\infty 
\rightarrow \End_{D(\cO_\infty/J)}(\cC(U_p,J,N))\] which factors through our 
chosen quotient map $R_\infty \rightarrow R^S_{\cS_{Q_N}}$and the 
$\cO_N$-algebra map \[R^S_{\cS_{Q_N}}\rightarrow 
\cT\wotimes_{\cO}\TT^{S\cup Q_N}(U^p_1(Q_N))_{\m_{Q_N,1}}.\]
\item If $U'_p$ is an open normal subgroup of $U_p$, $\cC(U'_p,J,N)$ is a 
complex of flat $\cO_\infty/J[U_p/U'_p]$-modules.
\item Let $\ba = \ker(\cO_\infty 
\rightarrow \cO)$. Suppose that $\ba \subset 
J$. Then $\cC(U_p,J,N) = \cC(U^p_0(Q_N)U_p,s(J))_{\m_{Q_N,0}}$ where 
$\cO_\infty/J \cong 
\cO/\varpi^{s(J)}$ and the natural map 
$\cC(U^p_0(Q_N)U_p,s(J))_{\m_{Q_N,0}}\rightarrow \cC(U^pU_p,s(J))_{\m}$ is a 
quasi-isomorphism.\end{enumerate}
\end{rem}
\begin{defn}
	For $d \ge 1$, $J$ an open ideal in $\cO_\infty$ and $N \in I_J$, we define 
	\[R(d,J,N) = 
	\left(R^{S}_{\cS_{Q_N}}/\m_{R^{S}_{\cS_{Q_N}}}^d\right)\otimes_{\cO_N}\cO_\infty/J.\]
	\end{defn}
	\begin{rem}
		Each ring $R(d,J,N)$ is a finite commutative local 
		$\cO_\infty/J$-algebra, equipped with a surjective $\cO$-algebra map 
		$R_\infty \rightarrow R(d,J,N)$. As in the beginning of the proof of 
		\cite[Prop.~3.1]{1409.7007}, for $d$ sufficiently large (depending 
		on $J$ and $U_p$), the map \[R_\infty 
		\rightarrow \End_{D(\cO_\infty/J)}(\cC(U_p,J,N))\] factors through the 
		quotient $R(d,J,N)$ and the map \[R(d,J,N)\rightarrow 
		\End_{D(\cO_\infty/J)}(\cC(U_p,J,N))\] is an $\cO_\infty$-algebra 
		homomorphism. We have an isomorphism \[R(d,J,N)/\ba \cong 
		R_{\cS}/(\m_{R_\cS}^d,\varpi^{s(\ba+J)})\] induced by the canonical 
		isomorphism $R_{\cS_{Q_N}}/\ba_{Q_N} \cong R_\cS$. \end{rem}
\begin{lem}\phantomsection\label{boundedpatchinginput}\begin{enumerate}
\item For all open ideals $J' \subset J$ and open normal subgroups $U'_p 
\subset U_p$ we have surjective maps of complexes 
\[\cC(U'_p,J',N)\rightarrow 
\cC(U_p,J,N)\] inducing isomorphisms \emph{(}of complexes\emph{)}
\[\cO_\infty/J\otimes_{\cO_\infty/J'[U_p/U'_p]}\cC(U'_p,J',N)\rightarrow\cC(U_p,J,N).\]

\item Let $K_1$ be a pro-$p$ Sylow subgroup of $K_0$ and let $U_p$ be an 
open 
normal subgroup of $K_1$. Then $\{\cC(U_p,J,N)\}_{N \in I_J}$ is a set of 
perfect 
chain complexes of $\cO_\infty/J[K_1/U_p]$-modules with bounded complexity.
\end{enumerate}
\end{lem}
\begin{proof}The maps of complexes $\cC(U'_p,J',N)\rightarrow 
\cC(U_p,J,N)$ are those induced by the natural maps
$\cO_\infty/J'\to\cO_\infty/J$ and
$\cC(U^p_1(Q_N)U_p',\cO)\to \cC(U^p_1(Q_N)U_p,\cO)$.

To see that $\cC(U_p,J,N)$ is perfect, we first observe that by part (1) we 
have an isomorphism $k\otimes_{\cO_\infty/J[K_1/U_p]}\cC(U_p,J,N) \cong 
\cC(K_1,\m_{\cO_\infty},N)$ --- note that $k$ is the residue field of the local 
ring $\cO_\infty/J[K_1/U_p]$ and $\cC(U_p,J,N)$ is a bounded-below complex of 
flat $\cO_\infty/J[K_1/U_p]$-modules with finitely generated homology. It  
follows from Proposition \ref{prop: minimal resolutions exist} that 
$\cC(U_p,J,N)$ has a minimal resolution, and since $\cC(K_1,\m_{\cO_\infty},N)$ 
has bounded homology we deduce that $\cC(U_p,J,N)$ is perfect.

It follows immediately from
the quasi-isomorphism 
\[
  \cC(U_pU_1^p(Q),s)_{\m_{Q,1}}\otimes_{\cO[\Delta_Q]}\cO\to\cC(U_pU^p,s)_{\m}\](which
comes from Proposition~\ref{prop: Hecke action
  at TW primes}~(3))
that the set of complexes has bounded complexity, as required.\end{proof}

\begin{defn}
Applying the construction of section \ref{sec:ultra}, we let $x \in \Spec((\cO_\infty/J)_{I_J})$ correspond to $\gF$
(here we use that~$\gF$ is non-principal, and therefore defines an
ultrafilter on~$I_J$), and define 
\[\cC(U_p,J,\infty) = 
(\cO_\infty/J)_{I_J,x}\otimes_{(\cO_\infty/J)_{I_J}}\left(\prod_{N \in I_J}\cC(U_p,J,N)\right)
.\]\end{defn}
\begin{rem}\phantomsection\label{patchremark}
	\begin{enumerate}
	\item  It follows from Lemma \ref{patchingflat} that if $U'_p$ is an 
		open normal subgroup of $U_p$, $\cC(U'_p,J,\infty)$ is a 
		complex of flat $\cO_\infty/J[U_p/U'_p]$-modules.
	\item It follows from Remark \ref{prepatchingremarks}(3) that if 
	$\ba\subset J$ 
	there is a natural quasi-isomorphism $\cC(U_p,J,\infty) \rightarrow 
	\cC(U^pU_p,s(J))_\m$.\end{enumerate}
\end{rem}
\begin{defn}
	Similarly, we define \[R(d,J,\infty) = 
	(\cO_\infty/J)_{I_J,x}\otimes_{(\cO_\infty/J)_{I_J}}\left(\prod_{N \in 
	I_J}R(d,J,N)\right)
	.\]
\end{defn}
\begin{remark}\label{Rpatchremark}
For $d$ sufficiently large (depending 
	on $J$ and $U_p$), the map \[R_\infty 
	\rightarrow \End_{D(\cO_\infty/J)}(\cC(U_p,J,\infty))\] factors through 
	$R(d,J,\infty)$ and the map \[R(d,J,\infty)\rightarrow 
	\End_{D(\cO_\infty/J)}(\cC(U_p,J,\infty))\] is an $\cO_\infty$-algebra 
	homomorphism. By Lemma \ref{patchingmodaug}, we have an isomorphism 
	\[R(d,J,\infty)/\ba \cong 
	R_{\cS}/(\m_{R_\cS}^d,\varpi^{s(\ba+J)})\] induced by the isomorphisms 
	$R(d,J,N)/\ba \cong 
	R_{\cS}/(\m_{R_\cS}^d,\varpi^{s(\ba+J)})$.

\end{remark}
\begin{lemma}\phantomsection\label{patchedcomplexesarecompatible}
	\begin{enumerate}
	\item For all open ideals $J' \subset J$ and open normal subgroups $U'_p 
	\subset U_p$, the natural maps of complexes 
	\[\cC(U'_p,J',\infty)\rightarrow 
	\cC(U_p,J,\infty)\] are surjective, and induce isomorphisms of complexes
	\[\cO_\infty/J\otimes_{\cO_\infty/J'[U_p/U'_p]}\cC(U'_p,J',\infty)
	\rightarrow\cC(U_p,J,\infty).\]
			
	\item Let $U_p$ be an open normal subgroup of $K_1$, and let $J$ be an open 
	ideal in $\cO_\infty$. Then $\cC(U_p,J,\infty)$ is a perfect complex of 
	$\cO_\infty/J[K_1/U_p]$-modules. If $U_p$ is moreover normal in $K_0$, then 
	$\cC(U_p,J,\infty)$ is a perfect complex of 
	$\cO_\infty/J[K_0/U_p]$-modules. 
\end{enumerate}
	\end{lemma}
	\begin{proof}
		The surjectivity claim of the first part follows immediately from 
		Lemma~\ref{boundedpatchinginput}(1), since taking the the direct 
		product over $N\in I_J$ 
		and localising at $x$ preserves surjectivity. It follows from 
		Lemma~\ref{patchingmodaug} and Lemma~\ref{boundedpatchinginput}(1) that 
		we 
		obtain an isomorphism of complexes 	
	\[\cO_\infty/J\otimes_{\cO_\infty/J'[U_p/U'_p]}\cC(U'_p,J',\infty)
	\rightarrow\cC(U_p,J,\infty).\]
		For the second part, the fact that $\cC(U_p,J,\infty)$ is a perfect 
		complex of 
		$\cO_\infty/J[K_1/U_p]$-modules follows from Lemma~\ref{productperf} 
		and Lemma~\ref{boundedpatchinginput}(2). To get perfectness over 
		$\cO_\infty/J[K_0/U_p]$ we apply (an obvious variant
                of) Lemma~\ref{perfectoversuperring}.
	\end{proof}
\begin{defn}\label{def: Cinfty}
	We define a complex of $\cO_\infty[[K_0]]$-modules \[\widetilde{\cC}(\infty) := 
	\invlim_{J,U_p} \cC(U_p,J,\infty).\]
\end{defn}
\begin{remark}
	The complex $\widetilde{\cC}(\infty)$ is naturally equipped with an 
	$\cO_\infty$-linear action 
	of $\prod_{v|p}G(F_v)$ (on each term of the complex), which extends the 
	$K_0$-action coming from the $\cO_\infty[[K_0]]$-module structure. 
	Explicitly, 
	for $g \in \prod_{v|p}G(F_v)$, right multiplication by $g$ gives a map of 
	complexes \[\cdot g: \cC(U_p,J,N) \rightarrow 
	\cC(g^{-1}U_pg,J,N)\] 
	for each $U_p$, $J$ and $N$. Supposing that $g^{-1}U_pg \subset K_0$, 
	applying our (functorial) patching 
	construction gives a map \[\cdot g: \cC(U_p,J,\infty) \rightarrow 
	\cC(g^{-1}U_pg,J,\infty)\] As $U_p$ runs over the cofinal subset of open 
	subgroups of  $K_0$ with $g^{-1}U_pg \subset K_0$, the subgroups 
	$g^{-1}U_pg$ also run over a cofinal subset of open subgroups of $K_0$, so 
	we can identify $\invlim_{J,U_p}\cC(g^{-1}U_pg,J,\infty)$ with 
	$\widetilde{\cC}(\infty)$. Therefore, taking the inverse limit over $J$ and 
	$U_p$ 
	gives the action of $g$ on $\widetilde{\cC}(\infty)$. 
\end{remark}
To verify that $\widetilde{\cC}(\infty)$ has good properties, we will need 
several technical Lemmas. 
\begin{lemma}\label{goodinvlim}
	Let $I$ be a countable directed poset. Let $\cC = 
	(\cC(i))_{i \in I}$ be an inverse system with $\cC(i)\in Ch(\cO)$. 
	Suppose that the following two conditions hold: \begin{enumerate}
		\item for every $i \in I$ and $m \in \Z$, the homology group 
		$H_m(\cC(i))$ is an Artinian $\cO$-module.
		\item Either the entries of $\cC(i)$ are Artinian $\cO$-modules for 
		every $i \in I$, or for every pair $i \le j$ in $I$ the transition map 
		$\cC(j)\rightarrow \cC(i)$ is surjective.
	\end{enumerate} 
	Then for every $m \in \Z$ there are natural isomorphisms \[H_m(\invlim_I 
	\cC) = \invlim_I H_m(\cC(i)).\]
\end{lemma}
\begin{proof}
	Since $I$ is direct and countable, it has a cofinal subset which is 
	isomorphic 
	(as a poset) to $\mathbb{N}$ with its usual ordering. So we can assume $I = 
	\mathbb{N}$. The proposition is then a consequence of \cite[Theorem 
	3.5.8]{weibel} (as assumption~(1) guarantees the Mittag-Leffler
	property for the~$H_m(\cC(i))$, and assumption~(2) guarantees it for 
	the~$\cC(i)$).
\end{proof}
\begin{lemma}\label{perfectoversuperring}
	Let $K$ be a compact $p$-adic analytic group, and let $K_1$ be a pro-$p$ 
	Sylow subgroup of $K$. Let $\cC$ be a bounded 
	below chain complex 
	of $\cO[[K]]$-modules. Suppose that $\cC|_{K_1}$ is perfect when 
	regarded as a complex 
	of $\cO[[K_1]]$-modules. Then $\cC$ is a perfect complex of $\cO[[K]]$-modules.
\end{lemma}
\begin{proof}
	We can assume that $\cC$ is a bounded below complex of projective 
	$\cO[[K]]$-modules. Let $\cF$ be a bounded complex of finite free 
	$\cO[[K_1]]$-modules with 
	a quasi-isomorphism $\alpha:\cF\rightarrow\cC|_{K_1}$. We have a homotopy 
	inverse $\beta: \cC|_{K_1} \rightarrow \cF$ to $\alpha$. We obtain maps of 
	complexes of $\cO[[K]]$-modules
	\begin{align*}\widetilde{\alpha}&:\cO[[K]]\otimes_{\cO[[K_1]]}\cF
\rightarrow \cC\\
	\widetilde{\beta}&:\cC \rightarrow 
	\cO[[K]]\otimes_{\cO[[K_1]]}\cF\end{align*} where 
	$\widetilde{\alpha}$ is given by the usual adjunction and 
	$\widetilde{\beta}$ is given by \[\widetilde{\beta}(x) = \sum_{gK_1 \in 
		K_0/K_1}[g]\otimes \beta(g^{-1}x).\] The composite 
	$\widetilde{\alpha}\circ\widetilde{\beta}$ is homotopic to 
	$[K_0:K_1]\id_\cC$, and $[K_0:K_1]$ is invertible in $\Z_p$, so 
	$\cC$ is a retract (in the homotopy category) of 
	$\cO[[K]]\otimes_{\cO[[K_1]]}\cF$. Since $	
\cO[[K]]\otimes_{\cO[[K_1]]}\cF$ is perfect, it follows that 
$\cC$ 
	is also perfect, since perfect complexes form a thick (or \'{e}paisse) 
	subcategory of 
	$D(\cO[[K]])$ (this follows from \cite[Prop.~6.4]{bn}, which 
	identifies perfect complexes with compact objects in $D(R)$) and therefore 
	the retraction of a perfect complex is perfect (thick subcategories of 
	triangulated categories are closed under retraction, by 
	definition).\end{proof}

As promised, we can now show that $\widetilde{\cC}(\infty)$ has various 
desirable properties.
\begin{prop}\phantomsection\label{completionofpatchedcomplexes}
	\begin{enumerate}
		
			\item For all open ideals $J\subset \cO_\infty$ and compact open 
			subgroups 
			$U_p$ of $K_0$ we have surjective maps of
                        complexes \emph{(}induced by the maps in \emph{Lemma 
                        \ref{patchedcomplexesarecompatible}(1))} \[\widetilde{\cC}(\infty)\rightarrow 
			\cC(U_p,J,\infty)\] inducing isomorphisms of complexes
			\[\cO_\infty/J\otimes_{\cO_\infty[[U_p]]}\widetilde{\cC}(\infty)
			\rightarrow\cC(U_p,J,\infty),\] and $\widetilde{\cC}(\infty)$ is a 
			complex of flat $\cO_\infty[[U_p]]$-modules. 
			
			\item $\widetilde{\cC}(\infty)$ is a perfect complex of 
		$\cO_\infty[[K_0]]$-modules.

\item There is a ring homomorphism $R_\infty \rightarrow 
\End_{D(\cO_\infty)}(\widetilde{\cC}(\infty))$ which factors as the composite 
of maps $R_\infty \rightarrow \invlim_{J,d} R(d,J,\infty)$ and $\invlim_{J,d} 
R(d,J,\infty) \rightarrow \End_{D(\cO_\infty)}(\widetilde{\cC}(\infty))$ \emph{(}the 
latter map is an $\cO_\infty$-algebra map\emph{)} given 
by the limit of the maps discussed in Remark \ref{Rpatchremark}.
	\end{enumerate}
\end{prop}
\begin{proof}
	The first part follows from Lemma 
	\ref{patchedcomplexesarecompatible}(1) and Lemma \ref{lem:flatnessinlimit}. 
	To see this, fix an open uniform pro-$p$ subgroup  $U'_p$  of
        $U_p$, and note that if $\cJ$ is the two-sided ideal in 
	$\cO_\infty[[U_p]]$ generated by the maximal ideal of $\cO_\infty[[U'_p]]$, 
	where  the $\cJ$-adic 
	topology on $\cO_\infty[[U_p]]$ is equivalent to the canonical profinite 
	topology. We set $K = \Z_p^g \times U_p$ in Lemma 
	\ref{lem:flatnessinlimit}, where $g$ is chosen so that 
	$\cO[[K]] = \cO_\infty[[U_p]]$.
	
	For $m \ge 1$ we can define a complex of flat
	$\cO_\infty[[U_p]]/\cJ^m$-modules by choosing $J$ and $V_p \subset U_p$ 
	sufficiently small so that $\cJ^m$ contains the kernel of the map 
	\[\cO_\infty[[U_p]] \rightarrow \cO_\infty/J[U_p/V_p]\] and considering the 
	complex 
	$\cC(V_p,J,\infty)\otimes_{\cO_\infty/J[U_p/V_p]}\cO_\infty[[U_p]]/\cJ^m$. 
	This complex is independent of the choice of $J$ and $V_p$, by Lemma 
	\ref{patchedcomplexesarecompatible}(1). In particular, by choosing $J$ and 
	$V_p$ sufficiently small, we get a natural surjective map 
	\[\cC(V_p,J,\infty)\otimes_{\cO_\infty/J[U_p/V_p]}\cO_\infty[[U_p]]/\cJ^{m+1}
	 \rightarrow 
	\cC(V_p,J,\infty)\otimes_{\cO_\infty/J[U_p/V_p]}\cO_\infty[[U_p]]/\cJ^m.\] 
	Taking the terms of these 
	complexes in fixed degree as $m$ varies gives a system of modules to which 
	Lemma \ref{lem:flatnessinlimit} applies, and 
          taking the inverse limit over 
	$m$ gives the complex $\widetilde{\cC}(\infty)$.

For the second part,  we first note that by Lemma~\ref{perfectoversuperring} it 
suffices to show that $\widetilde{\cC}(\infty)$ is perfect over 
$\cO_{\infty}[[K_1]]$, where $K_1$ is a pro-$p$ Sylow subgroup of $K_0$. For 
each $J$ and each $U_p$ open normal in $K_1$, there is by Remark~\ref{rem: patched minimal resolutions} a minimal resolution 
$\cF(U_p,J,\infty)$ of $\cC(U_p,J,\infty)$, which is isomorphic to the minimal 
resolutions of $\cC(U_p,J,N)$ for all $N \in I'$ where $I' \in \gF$. For each 
$J'\subset J$ and $U'_p$ open normal in $U_p$, we choose compatible maps
$\cF(U'_p,J',\infty)\rightarrow \cF(U_p,J,\infty)$ which are also compatible 
with the map 
\[\cC(U'_p,J',\infty) \rightarrow \cC(U_p,J,\infty)\] and induce isomorphisms 
\[\cO_\infty/J\otimes_{\cO_\infty/J'[U_p/U'_p]}\cF(U'_p,J',\infty)\cong 
\cF(U_p,J,\infty).\] In fact, rather than choosing maps for all $J$ and $U_p$, 
it suffices to choose maps between minimal resolutions $\cF_m$ of the complexes 
$\cC(U_p,J,\infty)\otimes_{\cO_\infty/J[K_1/U_p]}\cO_\infty[[K_1]]/\cJ^m$ 
discussed in the proof of the first part.  It follows from Lemma 
\ref{goodinvlim} that there is a 
quasi-isomorphism 
 \[\invlim_{m}\cF_m \rightarrow \widetilde{\cC}(\infty)\] and 
 $\invlim_{m}\cF_m$ is a bounded complex of finite free 
 $\cO_{\infty}[[K_1]]$-modules by construction, as required.

The third part follows from (the proof of) \cite[Lemma 
2.13(3)]{1409.7007}.\end{proof}
\begin{remark}\label{liftOtoR}
	Since the image of the map $\alpha: R_\infty \rightarrow \invlim_{J,d} 
	R(d,J,\infty)$ contains (the image 
	of) $\cO_\infty$, $\alpha(R_\infty)$ is naturally an $\cO_\infty$-algebra. 
	Since $\cO_\infty$ is formally smooth, we can choose a lift of the map 
	$\cO_\infty \rightarrow \alpha(R_\infty)$ to a map $\cO_\infty \rightarrow 
	R_\infty$. We make such a choice, and regard $R_\infty$ as an 
	$\cO_\infty$-algebra and $\alpha$ as an $\cO_\infty$-algebra
        map. \end{remark}
\begin{remark}
	With some more careful bookkeeping, it should be possible to show 
	that there is a natural map $R_\infty \rightarrow 
	\End_{D(\cO_\infty[[K_0]])}(\widetilde{\cC}(\infty))$ lifting the map 
	$R_\infty \rightarrow \End_{D(\cO_\infty)}(\widetilde{\cC}(\infty))$ which 
	we have described above. However, in our applications below, the complex 
	$\widetilde{\cC}(\infty)$ will have homology concentrated in a single 
	degree, so this doesn't give any additional information.
\end{remark}

The following Proposition shows that we can think of
$\widetilde{\cC}(\infty)$ as `patched completed homology'.\begin{prop}
  \label{prop: killing patching variables gives completed homology} If
  we let $\ba = \ker(\cO_\infty \rightarrow \cO)$, we have natural \emph{(}in
  particular, $\prod_{v|p}G(F_v)$-equivariant\emph{)} isomorphisms
  \[H_i(\cO_\infty/\ba \otimes_{\cO_\infty}\widetilde{\cC}(\infty))
    \cong \widetilde{H}_i(X_{U^p},\cO)_\m.\] 
    
  There are surjective maps $R_\infty/\ba \rightarrow R_{\cS} \rightarrow 
  \TT^S(U^p)_\m$ and the above isomorphism intertwines the action of $R_\infty$ 
  on the left hand side with the action of $\TT^S(U^p)_\m$ on the right.
\end{prop}\begin{proof}We have natural maps
\[\widetilde{\cC}(\infty) = \invlim_{J,U_p} \cC(U_p,J,\infty) \rightarrow 
\invlim_{\ba\subset J,U_p} \cC(U_p,J,\infty) \rightarrow  
\invlim_{s,U_p}\cC(U^pU_p,s)_\m.\]It follows 
from 
Lemma~\ref{goodinvlim} and Remark~\ref{patchremark}(2) that the natural map 
\[\invlim_{\ba\subset J,U_p} \cC(U_p,J,\infty) \rightarrow  
\invlim_{s,U_p}\cC(U^pU_p,s)_\m\] is a quasi-isomorphism and by Lemma 
\ref{goodinvlim}
we have natural isomorphisms \[H_n(\invlim_{s,U_p}\cC(U^pU_p,s)_\m) \cong 
\invlim_{s,U_p}{H}_n(X_{U^pU_p},\cO/\varpi^s)_\m.\] The natural map 
\[\alpha: \widetilde{H}_n(X_{U^p},\cO)_\m \rightarrow 
\invlim_{s,U_p}{H}_n(X_{U^pU_p},\cO/\varpi^s)_\m\] is also an isomorphism: 
indeed, we have short exact sequences \[0 \rightarrow 
{H}_n(X_{U^pU_p},\cO)_\m/\varpi^s \rightarrow {H}_n(X_{U^pU_p},\cO/\varpi^s)_\m 
\rightarrow  {H}_{n-1}(X_{U^pU_p},\cO)_\m[\varpi^s] \rightarrow 0\] so taking 
the limit over $(U_p, s)$ shows that the map $\alpha$ is an injection with 
$\varpi$-divisible cokernel. On the other hand, this cokernel is a finitely 
generated $\cO[[K_0]]$-module, so if it is $\varpi$-divisible it 
must be zero.

To finish the proof, by Proposition~\ref{completionofpatchedcomplexes}~(1), 
it suffices to 
show that the map  \[\cO_\infty/\ba 
\otimes_{\cO_\infty}\widetilde{\cC}(\infty) \rightarrow \invlim_{\ba\subset 
J,U_p} \cC(U_p,J,\infty) =
\invlim_{J,U_p}\cO_\infty/(\ba + 
J)\otimes_{\cO_\infty[[U_p]]}\widetilde{\cC}(\infty)\] is an isomorphism of 
complexes. As in the 
proof of Proposition~\ref{completionofpatchedcomplexes}~(1), we easily reduce to the following 
claim, where $\cJ = 
\m_{\cO_\infty[[U_p]]}\cO_\infty[[K_0]]$ for $U_p \subset K_0$ an open uniform 
pro-$p$ subgroup: suppose 
we have flat 
$\cO_\infty[[K_0]]/\cJ^m$-modules $M_m$ for each $m\ge 1$, with $M_m = 
M_{m+1}/\cJ^m$. Let $M=\invlim_m M_m$. Then $M/\ba M = \invlim_m M_m/\ba M_m$.

This claim follows from Lemma \ref{lem:flatnessinlimit}, taking $K = 
\Z_p^{g}\times K_0$ (where $g$ is chosen so that $\cO[[K]] = 
\cO_\infty[[K_0]]$), 
and $Q = \cO_\infty[[K_0]]/\ba$. 

The final claim of the Proposition follows from the fact that the isomorphisms 
$R(d,J,\infty)/\ba \cong R_\cS/(\m^d_{R_\cS},\varpi^{s(\ba+J)})$ of Remark
 \ref{Rpatchremark} induce an isomorphism 
\[\left(\invlim_{d,J}R(d,J,\infty)\right)/\ba \cong R_\cS.\qedhere\] \end{proof}

\begin{lem}\label{heckeonedegree}
	Let $\m \subset \TT^S(U^p)$ be a maximal ideal and suppose that 
	$\widetilde{H}_i(X_{U^p},\cO)_\m$ is non-zero for a single 
	$i$, which we denote by $q$. Then the map \[\alpha: \TT^S(U^p)_\m 
	\rightarrow \End_{\cO}(\widetilde{H}_{q}(X_{U^p},\cO)_\m)\] is an 
	injection.\end{lem}
\begin{proof}
    The map $\alpha$ 
    factors through the inclusion \[ 
	\End_{\cO[[K_0]]}(\widetilde{H}_{q}(X_{U^p},\cO)_\m)\subset 
	\End_{\cO}(\widetilde{H}_{q}(X_{U^p},\cO)_\m).\] Suppose $T$ is in 
	the 
	kernel of $\alpha$. Then, as an endomorphism in $D(\cO[[K_0]])$, 
	$T$ 
	acts on
	$\cO_\infty/\ba \otimes_{\cO_\infty}\widetilde{\cC}(\infty)$ as $0$ (by 
	Proposition \ref{prop: killing patching variables gives completed 
	homology}), and so for any $s \ge 1$ and $U_p$ compact open normal in $K_0$
	it acts, as an 
	endomorphism in $D(\cO/\varpi^s[K_0/U_p])$, as $0$ on 
	\[\cO/\varpi^s[K_0/U_p]\otimes_{\cO[[K_0]]}\cO_\infty/\ba 
	\otimes_{\cO_\infty}\widetilde{\cC}(\infty).\] 
	By 
	Proposition 
	\ref{completionofpatchedcomplexes} and Remark \ref{patchremark}, we deduce 
	that $T$ maps to $0$ in 
	$\TT^S(U_pU^p,s)_\m$. 
	Since $U_p$ and $s$ were arbitrary, we deduce that $T$ is equal to $0$. Of 
	course we don't require the patched complex $\widetilde{\cC}(\infty)$ to 
	prove this Lemma --- we can replace $\cO_\infty/\ba 
	\otimes_{\cO_\infty}\widetilde{\cC}(\infty)$ by any suitable complex 
	computing completed homology.
\end{proof}

\section{Applications of noncommutative algebra to patched completed
  homology}\label{sec: applications of algebra to patched homology}In
this section we apply the non-commutative algebra developed in Appendix~\ref{sec: non commutative
  algebra} to the output of the patching construction in Section
\ref{sec:patchingII}.

\subsection{Formally smooth local deformation rings}\label{subsec:
  formally smooth local def rings} We begin by recalling some
of the notation, assumptions and results of
Section~\ref{sec:patchingII}, and we then make an additional
assumption.

We assume Conjectures~\ref{galconj}
and~\ref{conj: compatibility at Q}. We work with a fixed~$U^p$ such
that~$U^pK_0$ is good, and we further assume that
\begin{itemize}
\item $p>n\ge 2$, 
\item $\rhobar_\m(G_{F(\zeta_p)})$ is enormous, and
\item $\rhobar_\m\not\cong\rhobar_\m\otimes\epsilonbar$.
\end{itemize}

We have two rings~$\cO_\infty$ and~$R_\infty$. The former is a power
series ring over~$\cO$, and the latter is a power series ring
over~$R_\cS^{S,\loc}$. More precisely, we have fixed an integer $q\gg
0$, and
$\cO_\infty$ is a power series ring in \[n^2\# S-1+(n-1)q\] variables
over~$\cO$, while $R_\infty$ is a power series ring in \[(n-1)q - n(n-1)[F : \Q]/2 - l_0 - 1 + \# S\]  variables over~$R_\cS^{S,\loc}$.
\begin{lem}\label{lem:dim Rinfty}
Suppose that for each place~$v|p$ of~$F$ there is no non-zero 
$k[G_{F_v}]$-equivariant
  map~$\rhobar|_{G_{F_v}}\to\rhobar|_{G_{F_{v}}}(1)$. Then~$R_\infty$ is equidimensional of dimension $\dim \cO_\infty + (n(n+1)/2 - 1)[F : \Q] - l_0$. 
\end{lem}
\begin{proof}For places $v | p$  we have~$H^2(G_{F_v},\ad^0\rhobar)=0$ by Tate local
  duality, and a standard calculation shows that~$R_v$ is formally smooth of
  dimension $1 + (n^2-1)[F_v:\Q_p] + (n^2-1)$
  (see e.g.\ \cite[Lem.\
  3.3.1]{allenbigrt}). If $v\nmid p$ then~$R_v$ is
  equidimensional of dimension~$n^2$ by~\cite[Thm.\ 2.5]{shottonGLn}. The claim then follows immediately
  (using~\cite[Lem.\ 3.3]{blght} to compute the dimension
  of~$R_\cS^{S,\loc}$).
\end{proof}

\begin{remark}\label{rem: numerology of dim Rinfty}Note that $(n(n+1)/2 - 1)[F : \Q]$ is equal to the dimension of the
  Borel subgroup $B$ in $G$. It
  follows from Lemma~\ref{lem:dim Rinfty} that we have
  \[\dim R_\infty + \dim(G/B) = \dim\cO_\infty[[K_0]] - l_0.\] See
  \cite[Equation (1.6)]{MR2905536} and the surrounding discussion for
  the same numerology.
 \end{remark}

Under the assumptions of Lemma~\ref{lem:dim Rinfty}, the local deformation 
ring~$R_v$ are formally smooth 
over~$\cO$ for $v|p$. We could make a similar assumption at places $v\nmid p$,
but it seems more reasonable to instead make the following more
general assumption.

\begin{hypothesis}\label{hyp: smooth local deformation rings}\begin{itemize}
		\item For each place~$v|p$ of~$F$ there is no non-zero 
		$k[G_{F_v}]$-equivariant
		map~$\rhobar|_{G_{F_v}}\to\rhobar|_{G_{F_{v}}}(1)$.
		\item For each
  place $v\in S$ with $v\nmid p$, we let~$\Rbar_v$ be an 
  irreducible component of~$R_v$ which is formally
  smooth. Let~$\cD_v$ be the local deformation problem corresponding
  to~$\Rbar_v$. Let \[\overline{\cS}=(\rhobar_\m,\epsilon^{n(1-n)/2},S,\{\cO\}_{v\in
                  S},\{\cD_v^\square\}_{v|p}\cup\{\cD_v\}_{v\in
                  S,v\nmid p}).\] 
 Then we further assume that for any set of Taylor--Wiles
  primes~$Q$, the representation ~$\rho_{\m,Q}$ of
  Conjecture~\ref{conj: compatibility at Q} is of type~$\overline{\cS}_Q$.  
  \end{itemize}
\end{hypothesis}

\begin{rem}
  \label{rem: discussion of smooth local deformation rings and local
    global compatibility}If $v\nmid p$ is such that there is
  no non-zero $k[G_{F_v}]$-equivariant
  map~$\rhobar|_{G_{F_v}}\to\rhobar|_{G_{F_{v}}}(1)$, then~$R_v$ is
  formally smooth and we can take~$\Rbar_v=R_v$.  Under the expected
  local-global compatibility, the question of whether ~$\rho_{\m,Q}$
  is of type~$\overline{\cS}_Q$ for a given choice of
  components~$\Rbar_v$ is governed by the local Langlands
  correspondence, and therefore depends on the choices of compact open
  subgroups~$U_v$. 

  Since our primary interest is in the behaviour at the places~$v|p$,
  we content ourselves with mentioning one important example. For
  any~$v\nmid p$ there is
  always at least one choice of irreducible component~$\Rbar_v$ which
  is formally smooth, namely the component corresponding to minimally
  ramified lifts; see~\cite[Lem.\ 2.4.19]{cht}. In general we do not
  expect to be able to make a choice of~$U_v$ compatible with the
  minimally ramified lifts; this is not a problem, as instead one
  should be able to consider a type (in the sense of Henniart's appendix to 
  \cite{breuil-mezard}) at each place~$v\nmid
  p$. Doing so would take us too far afield, so we content ourselves
  with noting that if~$n=2$, and~$v$ is not a vexing prime in the
  sense of~\cite{MR1638490} (so in particular if~$\# k(v)\not\equiv
  -1\mod{p}$), then we expect to be able to take~$U_v$ to be given by the image
  in~$\PGL_2(\cO_{F_v})$ of the subgroup of matrices
  in~$\GL_2(\cO_{F_v})$ whose last row is congruent to $(0,1)$
  modulo~$v^{n_v}$, where~$n_v$ is the conductor
  of~$\rhobar_\m|_{G_{F_v}}$. As in Remark~\ref{rem: compatibility at
    Q in CM case}, in the case that~$F$ is totally real or~CM, this
  compatibility should follow from forthcoming work of Varma.
\end{rem}

We assume Hypothesis~\ref{hyp: smooth local deformation rings} from
now on. If $v|p$ then we set~$\Rbar_v=R_v$; we then write
$\Rbar_{\cS}^{S,\loc}:=\wotimes_{v\in S}\overline{R}_v$, and
set~$\Rbar_\infty:=R_\infty\otimes_{R_{\cS}^{S,\loc}}\Rbar_{\cS}^{S,\loc}$. 
Under
our assumptions, $\Rbar_\infty$ is a power series ring over~$\cO$, and
has the same dimension as~$R_\infty$ (indeed, it is an irreducible
component of~$R_\infty$).

\begin{rem}
  \label{rem: local rings are complete intersections}If $v\nmid p$
  then~$R_v$ is in fact a reduced complete intersection, and is flat
  over~$\cO$ (\cite[Thm.\ 2.5]{shottonGLn}). In particular $R_v$ is Gorenstein. 
  It seems reasonable to
  imagine that these properties should be sufficient to carry out our
  analysis below without making any assumption at the places
  $v\nmid p$, but this would require a substantial generalisation of
  the material in Appendix~\ref{sec: non commutative algebra} (to
  Iwasawa algebras over more general rings than~$\cO$), so we have not
  pursued this. Note however that the `miracle flatness' result used in the 
  proof of Proposition \ref{prop: big R equals T} requires $\Rbar_\infty$ to be 
  regular --- moreover, in the $\GL_2/\Q$ case the conclusion of part 
  (\ref{Rflat}) of this Proposition does not hold when $R_p$ is not 
  regular (see \cite[Remark 7.7]{CEGGPSGL2}).
  \end{rem}
\subsection{Patched completed homology is
  Cohen--Macaulay}\label{subsec: patched CM}
 We return to the notation and set-up of section \ref{sec:patchingII}, and 
 recall that we have a perfect chain complex $\widetilde{\cC}(\infty)$ of 
 $\cO_\infty[[K_0]]$-modules (see Definition~\ref{def: Cinfty} and 
 Proposition~\ref{completionofpatchedcomplexes}), equipped with an 
 $\cO_\infty$-linear action of $\prod_{v|p}G(F_v)$ and an $\cO_\infty$-algebra 
 homomorphism 
 \[R_\infty \rightarrow \End_{D(\cO_\infty)}(\widetilde{\cC}(\infty)).\] 
 The action of $R_\infty$ on $\widetilde{\cC}(\infty)$ commutes with the action 
 of $\prod_{v|p}G(F_v)$ (and with that of $\cO_\infty[[K_0]]$). By 
 Hypothesis~\ref{hyp: smooth local
 	deformation rings} together with
      Remark~\ref{prepatchingremarks} and Proposition~\ref{completionofpatchedcomplexes}~(3), this map factors through the quotient $\Rbar_\infty$ of 
 	$R_\infty$. Recall that $\Rbar_\infty$ is a formal power series ring over 
 	$\cO$. The action of $\Rbar_\infty$ induces an $\cO_\infty[[K_0]]$-algebra 
 	homomorphism $\Rbar_\infty[[K_0]] \rightarrow 
 	\End_{D(\cO_\infty)}(\widetilde{\cC}(\infty))$, and in particular each 
 	homology group $H_i(\widetilde{\cC}(\infty))$ is a finitely generated 
 	$\Rbar_\infty[[K_0]]$-module. We refer to Definition~\ref{def: gradedepth} 
 	for the notion of the grade $j_A(M)$ of a module $M$ over a ring $A$ and to
        Definition~\ref{defn: CM module} for the notion of a
        Cohen--Macaulay module over~$\cO[[K_0]]$; this also
        gives us the definition of a Cohen--Macaulay module
        over~$\cO_\infty[[K]]$ or~$\Rbar_\infty[[K]]$ for any compact open $K 
        \subset K_0$.
 
 We have natural isomorphisms (for every $i\ge 0$) 
 \[H_i(\cO_\infty/\ba \otimes_{\cO_\infty}\widetilde{\cC}(\infty))
 \cong \widetilde{H}_i(X_{U^p},\cO)_\m,\]where $\ba=\ker(\cO_\infty\to\cO)$. Recall that $K_1$ denotes a pro-$p$ 
 Sylow subgroup of $K_0$, and~$B$ is the Borel subgroup of~$G$.
 \begin{prop}\label{patchedCM}
	Suppose that 
	\begin{enumerate}
		\item[(a)] ${H}_i(X_{U^pK_1},k)_\m = 0$ for $i$ outside the 
	range $[q_0,q_0+l_0]$ (note that ${H}_*(X_{U^pK_1},k)_\m$ is non-zero). 
	 \item[(b)]  $j_{\cO[[K_0]]}\left(\bigoplus_{i\ge 
	0}\widetilde{H}_i(X_{U^p},\cO)_\m\right) \ge l_0$.\end{enumerate} 

Then \begin{enumerate}
		\item $\widetilde{H}_i(X_{U^p},\cO)_\m = 0$ for $i \ne q_0$ and  
		$\widetilde{H}_{q_0}(X_{U^p},\cO)_\m$ is a Cohen--Macaulay 
		$\cO[[K_0]]$-module with
		\[\mathrm{pd}_{\cO[[K_0]]}(\widetilde{H}_{q_0}(X_{U^p},\cO)_\m) = 
		j_{\cO[[K_0]]}(\widetilde{H}_{q_0}(X_{U^p},\cO)_\m) = l_0.\] 
		\item ${H}_i(\widetilde{\cC}(\infty)) = 0$ for $i \ne q_0$ and 
		${H}_{q_0}(\widetilde{\cC}(\infty))$ is a Cohen--Macaulay 
		$\cO_\infty[[K_0]]$-module with 	
		\[\mathrm{pd}_{\cO_\infty[[K_0]]}\left({H}_{q_0}(\widetilde{\cC}(\infty))\right)
		 = 
		j_{\cO_\infty[[K_0]]}\left({H}_{q_0}(\widetilde{\cC}(\infty))\right) = 
		l_0.\] 
		\item  ${H}_{q_0}(\widetilde{\cC}(\infty))$ is a Cohen--Macaulay 
		$\Rbar_\infty[[K_0]]$-module with 	
		\[\mathrm{pd}_{\Rbar_\infty[[K_0]]}\left({H}_{q_0}(\widetilde{\cC}(\infty))\right)
		 = 
		j_{\Rbar_\infty[[K_0]]}\left({H}_{q_0}(\widetilde{\cC}(\infty))\right) 
		= 
		\dim(B)\] where $\dim(B) = 
		(\frac{n(n+1)}{2}-1)[F:\QQ]$.
	\end{enumerate}
	
	If we moreover suppose that \begin{enumerate}\item[(c)] 
	$j_{k[[K_0]]}\left(\bigoplus_{i\ge 
		0}\widetilde{H}_i(X_{U^p},k)_\m\right) \ge l_0$,\end{enumerate} then  
		$\widetilde{H}_i(X_{U^p},k)_\m = 0$ for $i \ne q_0$ and 
		both $\widetilde{H}_{q_0}(X_{U^p},\cO)_\m$ and 
		${H}_{q_0}(\widetilde{\cC}(\infty))$ are $\varpi$-torsion free.
\end{prop}

\begin{proof}We have 
$H_i(k\otimes^\LL_{\cO_\infty[[K_1]]}\widetilde{\cC}(\infty)) \cong 
{H}_i(X_{U^pK_1},k)_\m$ by Proposition~\ref{completionofpatchedcomplexes}(1). 
So the assumption that ${H}_i(X_{U^pK_1},k)_\m = 0$ for $i$ outside the 
	range $[q_0,q_0+l_0]$ implies (Lemma \ref{lem: minimal complex rank}) that 
	the minimal resolution $\cF$ of 
	$\widetilde{\cC}(\infty)$ (viewed as a complex of 
	$\cO_\infty[[K_1]]$-modules) is concentrated in degrees $[q_0,q_0+l_0]$.
	
Fix $H\subset K_1$ a normal compact open subgroup of $K_0$ 
which is uniform pro-$p$. We now apply Lemma 
\ref{lem:CG} to the shifted complex $\cO_\infty/\ba 
\otimes_{\cO_\infty}\cF[-q_0]$ of finite free $\cO[[H]]$-modules to 
deduce 
that $H_i(\cO_\infty/\ba \otimes_{\cO_\infty}\cF)\cong 
\widetilde{H}_i(X_{U^p},\cO)_\m$ vanishes for $i \ne q_0$ and 
$\mathrm{pd}_{\cO[[H]]}(\widetilde{H}_{q_0}(X_{U^p},\cO)_\m) = 
j_{\cO[[H]]}(\widetilde{H}_{q_0}(X_{U^p},\cO)_\m) = l_0$. Lemma 
\ref{lem:samepd} gives the first claim of the proposition: note that the 
perfect complex $\cO_\infty/\ba \otimes_{\cO_\infty}\widetilde{\cC}(\infty)$ of 
$\cO[[K]]$-modules has homology equal to 
$\widetilde{H}_{q_0}(X_{U^p},\cO)_\m$ concentrated in a single degree, so 
$\widetilde{H}_{q_0}(X_{U^p},\cO)_\m$ has finite projective dimension as a 
$\cO[[K]]$-module.

Now we move on to the second claim of the proposition. We begin by showing that $\widetilde{\cC}(\infty)$ has non-zero homology only in degree $q_0$. As we will explain, this follows from the fact (which we have just established) that $\cO_\infty/\ba \otimes_{\cO_\infty}\widetilde{\cC}(\infty)$ has non-zero homology only in degree $q_0$. To see this, we recall that $\cO_\infty = \cO[[x_1,\ldots,x_g]]$, and begin by showing that $\cO_\infty/(x_1,\ldots,x_{g-1}) \otimes_{\cO_\infty}\widetilde{\cC}(\infty)$ has non-zero homology only in degree $q_0$. For each $i$ we have an injective map (part of a short exact sequence coming from a degenerating $\Tor$ spectral sequence) \[H_i(\cO_\infty/(x_1,\ldots,x_{g-1})\otimes_{\cO_\infty}\widetilde{\cC}(\infty))/x_g \hookrightarrow  H_i(\cO_\infty/\ba \otimes_{\cO_\infty}\widetilde{\cC}(\infty))\] so it follows from Nakayama's lemma, applied to the finitely generated $\cO_\infty[[K_1]]$-module $H_i(\cO_\infty/(x_1,\ldots,x_{g-1})\otimes_{\cO_\infty}\widetilde{\cC}(\infty))$, that $H_i(\cO_\infty/(x_1,\ldots,x_{g-1})\otimes_{\cO_\infty}\widetilde{\cC}(\infty))$ is zero for $i \ne q_0$. Repeating this argument, we eventually deduce that $H_i(\widetilde{\cC}(\infty))$ is zero for $i \ne q_0$.

Since $\cF$ is a chain 
complex whose first nonzero term has degree $q_0$, we have
$\cO_\infty/\ba \otimes_{\cO_\infty}H_{q_0}(\cF)=H_{q_0}(\cO_\infty/\ba 
\otimes_{\cO_\infty}\cF)\cong \widetilde{H}_{q_0}(X_{U^p},\cO)_\m$, so Lemma \ref{lem:obviouscodimineq} 
implies that \[j_{\cO_\infty[[H]]}(H_{q_0}(\cF)) \ge 
j_{\cO[[H]]}(\widetilde{H}_{q_0}(X_{U^p},\cO)_\m)=l_0.\] Another application of 
Lemmas~\ref{lem:CG} and~\ref{lem:samepd} gives us the second claim 
of the proposition. 

Next, we establish the third claim. It follows from Corollary 
\ref{cor:stayCM} and what we have established above that 
${H}_{q_0}(\widetilde{\cC}(\infty))$ is a Cohen--Macaulay 
$\Rbar_\infty[[H]]$-module, with grade and projective dimension as in 
the claim. 
Lemma \ref{lem:samepd} and Lemma \ref{perfectoversuperring} establish the claim 
as stated once we verify that $H_{q_0}(\widetilde{\cC}(\infty))$ 
has finite projective dimension over $\Rbar_\infty[[K_1]]$. To verify this, let 
$\cG$ be a minimal resolution of $H_{q_0}(\widetilde{\cC}(\infty))$, viewed as 
a complex of $\Rbar_\infty[[K_1]]$-modules concentrated in degree $q_0$. The 
complex $\Rbar_\infty \otimes_{\Rbar_\infty[[K_1]]}\cG$ (we mod out by the 
augmentation 
ideal for $K_1$) has bounded and finitely generated homology, since 
$H_i(\Rbar_\infty \otimes_{\Rbar_\infty[[K_1]]}\cG) = 
H_i(\cO_\infty\otimes_{\cO_\infty[[K_1]]}\widetilde{\cC}(\infty))$ and 
$\widetilde{\cC}(\infty)$ is a perfect complex of $\cO_\infty[[K_1]]$-modules. 
Since $\Rbar_\infty$ is regular, $\Rbar_\infty 
\otimes_{\Rbar_\infty[[K_1]]}\cG$ is 
therefore a perfect complex of $\Rbar_\infty$-modules, so 
$k\otimes_{\Rbar_\infty[[K_1]]}\cG$ also has bounded homology. We deduce that 
the 
minimal complex $\cG$ is itself bounded, so $H_{q_0}(\widetilde{\cC}(\infty))$ 
has finite projective dimension over $\Rbar_\infty[[K_1]]$.

For the last part of the Proposition, if we assume that 
\[j_{k[[K_0]]}\left(\bigoplus_{i\ge 
	0}\widetilde{H}_i(X_{U^p},k)_\m\right) = l_0\] then we may apply Lemma 
\ref{lem:CG} to the shifted complex $\cO_\infty/\m_{\cO_\infty} 
\otimes_{\cO_\infty}\cF[-q_0]$ of finite free $k[[H]]$-modules to deduce that 
$\widetilde{H}_i(X_{U^p},k)_\m = 0$ for $i \ne q_0$. This shows that 
\[\mathrm{Tor}_i^\cO(\cO/\varpi,\widetilde{H}_{q_0}(X_{U^p},\cO)_\m) =
\widetilde{H}_{q_0+i}(X_{U^p},k)_\m = 0\] for $i > 0$, so 
$\widetilde{H}_{q_0}(X_{U^p},\cO)_\m$ is $\varpi$-torsion free. Arguing as for 
the second part, we deduce that ${H}_i(\widetilde{\cC}(\infty)/\varpi) = 0$ for 
$i \ne q_0$ and 
hence $H_{q_0}(\widetilde{\cC}(\infty))$ is also $\varpi$-torsion free.
\end{proof}

\begin{remark}\phantomsection\label{rem: assumptions abc for IQF}
\begin{enumerate}
	\item  Hypothesis (b) that $j_{\cO[[K_0]]}\left(\bigoplus_{i\ge 
		0}\widetilde{H}_i(X_{U^p},\cO)_\m\right) \ge l_0$ of the above
	Proposition is implied by the codimension conjecture of Calegari and
	Emerton \cite[Conjecture 1.5]{MR2905536} (indeed equality is conjectured to 
	hold here). For $\PGL_2$ over an
	imaginary quadratic field, this hypothesis holds (for example, by the 
	argument of \cite[Example
	1.12]{MR2905536}).
	\item Hypothesis (a), that ${H}_i(X_{U^pK_1},k)_\m = 0$ for $n$ outside the 
	range $[q_0,q_0+l_0]$, is conjectured in \cite[Conj.\ B(4)(a)]{1207.4224}. 
	Again, for $\PGL_2$ over an imaginary quadratic field, 
	the hypothesis holds: we have $l_0 = 1$, $q_0 = 1$ and the dimension of 
	$X_{U^pK_1}$ is equal to $3$, so it suffices to check that 
	${H}_0(X_{U^pK_1},k)_\m = {H}_3(X_{U^pK_1},k)_\m = 0$ which follows from 
	the fact that $\m$ is 
	non-Eisenstein.
	\item In contrast to the other hypotheses, hypothesis (c) seems difficult 
	to 
	verify even for 
	$\PGL_2$ over an imaginary quadratic field. We cannot rule out (for 
	example) $\widetilde{H}_1(X_{U^p},\cO)_\m$ containing a $\varpi$-torsion 
	submodule which is torsion free over $k[[K_0]]$, in which case  
	$j_{k[[K_0]]}\widetilde{H}_1(X_{U^p},k)_\m = 0$.
	\end{enumerate}
\end{remark}
\begin{remark}\label{rem: Rinfty action}
		It follows from the second part of the Proposition that the map 
		$R_\infty \rightarrow \End_{D(\cO_\infty)}(\widetilde{\cC}(\infty))$ 
		(which 
		commutes with the $G$ action) arises from a map $R_\infty \rightarrow 
		\End_{\cO_\infty[G]}(H_{q_0}(\widetilde{\cC}(\infty)))$. In particular, 
		the 
		action of $R_\infty$ on $\widetilde{\cC}(\infty)$ can be thought of as 
		taking place in, for example, the derived category of  
		$\cO_\infty[[K_0]]$-modules with compatible $G$-action.
\end{remark}

\subsection{Miracle flatness and `big $R = \TT$'}\label{subsec:
  miracle flatness}
We have a surjective map $R_{\overline{\cS}} \rightarrow \TT^S(U^p)_\m$. If 
this map is an isomorphism, the global Euler characteristic formula for Galois 
cohomology gives an expected dimension of $1 + (\frac{n(n+1)}{2}-1)[F:\QQ] - 
l_0$ for both these rings. See \cite[Conj.\ 3.1]{emertonICM}.

The following Proposition shows that this dimension formula, as well as the 
isomorphism $R_{\overline{\cS}} \cong \TT^S(U^p)_\m$, is implied by a
natural condition on the codimension (over $k[[K_0]]$) of the fibre of the completed 
homology module 
$\widetilde{H}_{q_0}(X_{U^p},\cO)_\m$ at the maximal ideal $\m$ of the Hecke 
algebra. The method of proof is in some sense a precise version of the 
heuristics discussed in \cite[\S 3.1.1]{emertonICM} which compare the Krull 
dimension of 
$\TT^S(U^p)_\m$ and the Iwasawa theoretic dimensions of 
$\widetilde{H}_{q_0}(X_{U^p},\cO)_\m$ and its mod $\m$ fibre. A
related argument was found independently by Emerton and
Pa{\v{s}}k{\=u}nas, and will appear in a forthcoming paper\footnote{This has 
now appeared: \cite{EPdensity}} of theirs.

\begin{prop}\label{prop: big R equals T}
Suppose that assumptions \emph{(a)} and \emph{(b)} of Proposition~\ref{patchedCM} hold, and 
that we 
moreover have
\[j_{k[[K_0]]}(\widetilde{H}_{q_0}(X_{U^p},\cO)_\m/\m\widetilde{H}_{q_0}(X_{U^p},\cO)_\m)
 \ge \dim(B).\]
 
Recall that we are assuming Hypothesis \ref{hyp: smooth local deformation 
rings}, which implies that $\Rbar_\infty$ is a power series ring over $\cO$.

Then we have the following: \begin{enumerate}
	\item\label{Rflat} $H_{q_0}(\widetilde{\cC}(\infty))$ is a flat 
	$\Rbar_\infty$-module.
	\item The ideal $\Rbar_\infty\ba$ is generated by a regular sequence in 
	$\Rbar_\infty$.
	\item The surjective maps \[\Rbar_\infty/\ba \rightarrow R_{\overline{\cS}} 
	\rightarrow \TT^S(U^p)_\m\] are all isomorphisms and 
	$\widetilde{H}_{q_0}(X_{U^p},\cO)_\m$ is a faithfully flat 
	$\TT^S(U^p)_\m$-module.
	\item The rings $R_{\overline{\cS}} \cong \TT^S(U^p)_\m$ are local complete 
	intersections with Krull dimension equal to $\dim(R_\infty) - 
	\dim(\cO_\infty) + 1 = 1 + (\frac{n(n+1)}{2}-1)[F:\QQ] - l_0$.\item If assumption \emph{(c)} of Proposition~\ref{patchedCM} holds, then 
	$\TT^S(U^p)_\m$ is 
	$\varpi$-torsion free.
	\end{enumerate}
\end{prop}
\begin{proof}
First we note that by Lemma~\ref{lem:obviouscodimineq} and 
Proposition~\ref{patchedCM} we have 
\[j_{k[[K_0]]}(\widetilde{H}_{q_0}(X_{U^p},\cO)_\m/\m\widetilde{H}_{q_0}(X_{U^p},\cO)_\m)
\le j_{\Rbar_\infty[[K_0]]}\left(H_{q_0}(\widetilde{\cC}(\infty))\right) = 
\dim(B),\] since 
\[\Rbar_\infty/\m_{\Rbar_\infty}\otimes_{\Rbar_\infty}H_{q_0}(\widetilde{\cC}(\infty))
= \widetilde{H}_{q_0}(X_{U^p},\cO)_\m/\m\widetilde{H}_{q_0}(X_{U^p},\cO)_\m.\] 
So our assumption implies that we have equality of codimensions 
here. The first claim then follows immediately from 
Propositions~\ref{patchedCM} and~ \ref{miracleflatness}.
 
 For the second part, write $\ba = (x_1,\ldots,x_g)$ where $\cO_\infty = 
 \cO[[x_1,\ldots,x_g]]$ (so $g = \dim(\cO_\infty) - 1$). Note that, by 
 Proposition \ref{patchedCM} (which in 
 particular says that the complexes $\widetilde{\cC}(\infty)$ and 
 $\cO_\infty/\ba\otimes_{\cO_\infty}\widetilde{\cC}(\infty)$ both have homology 
 concentrated in a single degree) we have 
 \[\mathrm{Tor}_i^{\cO_\infty}(\cO_\infty/\ba,H_{q_0}(\widetilde{\cC}(\infty))) 
 = 
 \widetilde{H}_{q_0+i}(X_{U^p},\cO)_\m = 0\] for $i > 0$. So (by considering 
 the Koszul complex for $(x_1,\ldots,x_g)$) we see that $(x_1,\ldots,x_g)$ is a 
 regular sequence on $H_{q_0}(\widetilde{\cC}(\infty))$. Since 
 $H_{q_0}(\widetilde{\cC}(\infty))$ is a flat $\Rbar_\infty$-module and its 
 reduction mod $\m_{\Rbar_\infty}$ is non-zero (by Nakayama, since the module 
 is finitely generated over $\Rbar_\infty[[K_0]]$), it follows 
 from \cite[Thm.~7.2]{MR1011461} that $H_{q_0}(\widetilde{\cC}(\infty))$ is 
 a faithfully flat $\Rbar_\infty$-module and we can conclude that 
 $(x_1,\ldots,x_g)$ is a regular sequence in $\Rbar_\infty$ --- this can be 
 seen by considering the Koszul homology groups 
 \[H_*^{\Rbar_\infty}((x_1,\ldots,x_g),H_{q_0}(\widetilde{\cC}(\infty))) \cong 
 H_*^{\Rbar_\infty}((x_1,\ldots,x_g),\Rbar_\infty)\otimes_{\Rbar_\infty} 
 H_{q_0}(\widetilde{\cC}(\infty)),\] and by faithful flatness we have 
 $H_i^{\Rbar_\infty}((x_1,\ldots,x_g),\Rbar_\infty) = 0$ for $i \ne 0$ and 
 therefore $(x_1,\ldots,x_g)$ is a regular sequence in $\Rbar_\infty$. This
 gives 
 the second part. 
 
 For the third part, since $H_{q_0}(\widetilde{\cC}(\infty))$ is a flat 
 ${\Rbar}_\infty$-module, $\widetilde{H}_{q_0}(X_{U^p},\cO)_\m = 
 \cO_\infty/\ba\otimes_{\cO_\infty}H_{q_0}(\widetilde{\cC}(\infty))$ is a flat 
 ${\Rbar}_\infty/\ba$-module. As before, it follows from 
 \cite[Thm.~7.2]{MR1011461} that 
 $\widetilde{H}_{q_0}(X_{U^p},\cO)_\m$ is a faithfully flat 
 ${\Rbar}_\infty/\ba$-module and is in particular faithful. It follows that the 
 surjective maps appearing in the third part must also be injective, since the 
 action of ${\Rbar}_\infty/\ba$ on $\widetilde{H}_{q_0}(X_{U^p},\cO)_\m$ 
 factors through these maps. This completes the proof of the third part.
 
 The fourth part follows immediately from the second and third parts.
 
 The fifth part follows from the fact that $\TT^S(U^p)_\m$ acts faithfully on 
 $\widetilde{H}_{q_0}(X_{U^p},\cO)_\m$ (by Lemma \ref{heckeonedegree}), which 
 is $\varpi$-torsion free (under our additional assumption) by Proposition 
 \ref{patchedCM}. Alternatively, one can redo the argument of part (2) of the 
 Proposition to show that $(\varpi,x_1,\ldots,x_g)$ is a regular sequence in 
 $\Rbar_\infty$, and so in particular $\varpi$ is not a zero-divisor in 
 $\Rbar_\infty/\ba \cong \TT^S(U^p)_\m$.
\end{proof}
\begin{remark}
	To explain the condition  
	\[j_{k[[K_0]]}(\widetilde{H}_{q_0}(X_{U^p},\cO)_\m/\m\widetilde{H}_{q_0}(X_{U^p},\cO)_\m)
	\ge \dim(B)\] we first note that the parabolic induction of a $k$-valued 
	character from $B$ to $G$ has codimension $\dim(B)$ over $k[[K_0]]$. We 
	moreover
	expect this to be the codimension of any `generic' irreducible admissible 
	smooth $k$-representation of $G$, with other irreducibles having at least 
	this codimension. In the case $G = \PGL_2(\Q_p)$ this is 
	true: any 
	infinite-dimensional irreducible smooth $k$-representation of $G$ has 
	codimension $\dim(B) = 2$ \cite[Proof of 
	Cor.~7.5]{schmidtstrauch}, whilst the finite-dimensional representations 
	have codimension $3$.
	
 It seems reasonable to expect that 
 the smooth representation 
\[\left(\widetilde{H}_{q_0}(X_{U^p},\cO)_\m/
\m\widetilde{H}_{q_0}(X_{U^p},\cO)_\m\right)^\vee\]
  is 
 a finite length representation of $G$, and therefore we expect it to have 
 codimension $\ge \dim(B)$ also.
 
 We also point out that our assumption that $\overline{R}_\infty$ is regular is 
 essential in order to apply Proposition \ref{miracleflatness}. See Remark 
 \ref{rem: local 
 rings are complete intersections}.
\end{remark}

\section{The $p$-adic local Langlands correspondence
  for~$\GL_2(\Qp)$}\label{sec: p splits completely}
In this section we specialise to the case that $n=2$ and~$p$ splits
completely in~$F$, and use the techniques of~\cite{CEGGPSGL2} to study
the relationship of our constructions to the $p$-adic local Langlands
correspondence for~$\GL_2(\Qp)$.

\subsection{A local-global compatibility conjecture}\label{subsec:
  local global compatibility conjecture PGL2}We continue to
make the assumptions made in Section~\ref{sec: applications of algebra
  to patched homology}, as well as assumptions~(a) and~(b) of
Proposition~\ref{patchedCM}.

In addition we
assume that
\begin{itemize}
\item $n=2$, 
\item $p$ splits completely in~$F$, 
\item if~$\rhobar_\m|_{G_v}$ is ramified for some place~$v\nmid p$, then~$v$ is
  not a vexing prime in the sense of~\cite{MR1638490}, and
\item for each place~$v|p$, $\rhobar_\m|_{G_v}$ is either
  absolutely irreducible, or is a non-split extension of characters,
  whose ratio is not the trivial character or the mod~$p$ cyclotomic character.
\end{itemize}

This last assumption allows us to use the results
of~\cite{CEGGPSGL2}; it guarantees in particular that each
$\rhobar_\m|_{G_{F_v}}$ admits a universal deformation
ring~$\Rdef_v$. Since $n=2$, $l_0$ is just equal to~$r_2$, the number
of complex places of~$F$.

From now on in a slight abuse of notation for each place~$v|p$ we
write~$G_v$ for~$\PGL_2(F_v)$ and~$K_v$ for~$\PGL_2(\cO_{F_v})$, and
we write~$G$ for~$\prod_{v|p}G_v$. Recall that~$K_0=\prod_{v|p}K_v$.

Since our interest is primarily in
phenomena at places dividing~$p$, we content ourselves with the
`minimal level' situation at places not dividing~$p$; that is, we
choose~$\Rbar_\infty$ and the level~$U^p$ as in the second paragraph
of Remark~\ref{rem: discussion of smooth local deformation rings and local
    global compatibility}, and assume Hypothesis \ref{hyp: smooth local 
    deformation rings} holds for this choice.  (The reader may object that this 
    level is
  not necessarily $S$-good; as usual in the Taylor--Wiles method, this
  difficulty is easily resolved by shrinking the level at an
  auxiliary place at which~$\rhobar_\m$ admits no ramified
  deformations, and for simplicity of exposition we ignore this point.)

We would like to understand the action of
$(\wotimes_{v|p,\cO}\Rdef_v)[G]$ on~$H_{q_0}(\widetilde{\cC}(\infty))$.
When $F=\Q$ it follows from the local-global compatibility theorem
of~\cite{emerton2010local} that this action is determined by the
$p$-adic local Langlands correspondence for~$\GL_2(\Qp)$, and it is
natural to expect that the same applies for general number fields~$F$.

More precisely, for each place~$v|p$, we can associate an absolutely
irreducible $k$-representation~$\pi_v$ of $\GL_2(F_v)$
to~$\rhobar_\m|_{G_{F_v}}$ via the recipe of~\cite[Lem.\
2.15~(5)]{CEGGPSGL2}; note that by~ \cite[Rem.\ 2.17]{CEGGPSGL2}, the
central character of~$\pi_v$ is trivial, so we can regard it as
a representation of~$G_v$.
\begin{defn}\label{defn: gothic C category}
  If~$H$ is a $p$-adic analytic group and~$A$ is a complete local
  Noetherian $\cO$-algebra, then we write $\mathfrak{C}_H(A)$ for the
  Pontryagin dual of the category of locally admissible
  $A$-representations of~$H$ (cf.\ Appendix~\ref{appendix: tensor products 
  projective covers} and \cite[\S 4.4]{CEGGPSGL2}).\end{defn}
We  let~$P_v\onto \pi_v^\vee$ be a
projective envelope of~$\pi_v^\vee$ in~$\mathfrak{C}_{G_v}(\cO)$. By~\cite[Prop.\ 6.3, Cor.\ 8.7]{paskunasimage} there is a natural 
isomorphism
$\Rdef_v\to\End_{\mathfrak{C}_{G_v}(\cO)}(P_v)$. (This is a large part
of the $p$-adic local Langlands correspondence for~$\GL_2(\Qp)$.)

Write $P:=\wotimes_{v|p,\cO}P_v$ which is naturally
an~$R_p^{\loc}:=\wotimes_{v|p,\cO}\Rdef_v$-module. For each~$v|p$ we
make a choice (in its strict equivalence class) of the universal
deformation of~$\rhobar_\m|_{G_{F_v}}$ to~$\Rdef_v$, so that we can
regard~$\Rbar_\infty$ as an $R_p^{\loc}$-module. For some $g\ge 0$ we
can and do choose an isomorphism of~$R_p^{\loc}$-algebras
$\Rbar_\infty\cong
R_p^{\loc}\wotimes_\cO\cO[[x_1,\dots,x_g]]$.

\begin{conj}\label{conj: p-adic local Langlands gives us the action for general
    F} For some~$m\ge 1$ there is an isomorphism of
  $\Rbar_\infty\left[G\right]$-modules \[H_{q_0}(\widetilde{\cC}(\infty))\cong
  \Rbar_\infty\wotimes_{R_p^{\loc}}  P^{\oplus m}.\]
  \end{conj}
\begin{rem}
  \label{rem: multiplicity at infinite places}We do not know what the
  value of~$m$ in Conjecture~\ref{conj: p-adic local Langlands gives us the action for general
    F} should be in general. The natural guess is that~$m=2^{r_1}$
  where~$r_1$ is the number of real places of~$F$, since this is the
  dimension of the~$(\mathfrak{g},K)$-cohomology in degree~$q_0$ of
  the trivial representation for the
  group~$\Res_{F/\Q}\PGL_2$. This guess is justified by Corollary~\ref{cor: 
  modularity lifting}. Indeed, if $H_{q_0}(X_{K_0U^p},\sigma)_\m$ is non-zero 
  for some irreducible $E$-representation of $K_0$, then Corollary~\ref{cor: 
  modularity lifting} shows that $m$ is equal to the 
  multiplicity of a system of Hecke eigenvalues (away from $S$) in 
  $H_{q_0}(X_{K_0U^p},\sigma)_\m$.
\end{rem}
We now explain some consequences of this conjecture for completed
homology and homology with coefficients. In the proof of the following
result we will briefly need the notion of the \emph{atome
  automorphe}~$\kappa_v$ associated
to~$\rhobar_\m|_{G_{F_v}}$; recall that if~$\rhobar_\m|_{G_{F_v}}$ is irreducible,
then $\kappa_v=\pi_v$ is an irreducible supersingular representation
of~$G_v$, while if $\rhobar_\m|_{G_{F_v}}$ is reducible, $\kappa_v$ is
a non-split extension of irreducible principal series representations with
socle~$\pi_v$ (see for example the beginning of~\cite[\S 8]{paskunasimage}).

\begin{prop}
  \label{prop: big R equals T and describing completed homology}Assume
  Conjecture~\ref{conj: p-adic local Langlands gives us the action for general
    F}. Then we have an isomorphism of local
  complete intersections $R_{\overline{\cS}} \cong \TT^S(U^p)_\m$  with Krull dimension equal to $1 + 2[F:\QQ] -
        l_0$. Furthermore, there is an isomorphism of
        $\TT^S(U^p)_\m\left[G\right]$-modules \[
          \widetilde{H}_{q_0}(X_{U^p},\cO)_\m\cong   \TT^S(U^p)_\m  \wotimes_{R_p^{\loc}}P^{\oplus m}.\]
If we moreover make assumption~\emph{(c)} of
Proposition~\ref{patchedCM} then $\TT^S(U^p)_\m$ is $\varpi$-torsion free.
\end{prop}
\begin{proof}The isomorphism $R_{\overline{\cS}} \cong \TT^S(U^p)_\m$
  and the properties of these rings will follow immediately from
  Proposition~\ref{prop: big R equals T} once we know
  that \[j_{k[[K_0]]}(\widetilde{H}_{q_0}(X_{U^p},\cO)_\m/\m\widetilde{H}_{q_0}(X_{U^p},\cO)_\m)
    =2[F:\Q].\]

Now, since we are assuming Conjecture~\ref{conj: p-adic local Langlands gives us the action for general
    F}, we have
  \begin{align*}
    \widetilde{H}_{q_0}(X_{U^p},\cO)_\m/\m\widetilde{H}_{q_0}(X_{U^p},\cO)_\m
    &=
      \Rbar_\infty/\m_{\Rbar_\infty}\otimes_{\Rbar_\infty}H_{q_0}(\widetilde{\cC}(\infty))\\
&= P^{\oplus m}\otimes_{R_p^{\loc}}R_p^{\loc}/\m_{R_p^{\loc}}\\ &=
                                                                 (
                                                                  \wotimes_{v|p}P_v\wotimes_{\Rdef_v}k)^{\oplus m}
                                                                  \\
    &= (\wotimes_{v|p}\kappa_v^\vee)^{\oplus m}
  \end{align*}(where in the last line we have used~\cite[Prop.\ 1.12,
  6.1, 8.3]{paskunasimage} and
  that~$\Rdef_v=\End_{\mathfrak{C}_{G_v}(\cO)}(P_v)$). By Lemma~\ref{tensorgrade} we
are therefore reduced to showing that for each~$v|p$,
\[j_{k[[K_v]]}(\kappa_v^\vee)=2. \] By the same argument as 
Lemma~\ref{changeofrings} it is enough to show that
$j_{k[[\GL_2(\cO_{F_v})]]}(\kappa_v^\vee)=3$ (we pass from 
$k[[\GL_2(\cO_{F_v})]]$ to $k[[\PGL_2(\cO_{F_v})]]$ by quotienting out by a 
central regular element which acts trivially on $\kappa_v^\vee$). By 
Lemma~\ref{gradeses} we
are reduced to the same statement for irreducible principal series and
supersingular representations of~$\GL_2(\Qp)$, which is proved in~\cite[Proof of 
  Cor.~7.5]{schmidtstrauch}. 

Finally, we have \begin{align*} \widetilde{H}_{q_0}(X_{U^p},\cO)_\m &=
                                                                      \Rbar_\infty/\mathbf{a}\Rbar_\infty\otimes_{\Rbar_\infty}H_{q_0}(\widetilde{\cC}(\infty))\\
&=R_{\overline{\cS}}\otimes_{\Rbar_\infty}H_{q_0}(\widetilde{\cC}(\infty))\\
&=\TT^S(U^p)_\m\otimes_{\Rbar_\infty}H_{q_0}(\widetilde{\cC}(\infty))\\
&\cong
  \TT^S(U^p)_\m  \wotimes_{R_p^{\loc}}P^{\oplus m},\end{align*}as required.
\end{proof}

We recall from~\cite[\S 2]{CEGGPSGL2} some notation for Hecke algebras
and crystalline deformation rings. (In fact our setting is slightly
different, as we are working with~$\PGL_2$ rather than~$\GL_2$, but
this makes no difference in practice and we will not emphasise this
point below.) Let~$\sigma$ be an irreducible $E$-representation
of~$K_0$. Any such representation is of the
form~$\otimes_{v|p}\sigma_v$, where~$\sigma_v$ is the
representation of~$G_v$ given by
$\sigma_v=\det^{a_v}\otimes\Sym^{b_v}E^2$ for integers~$a_v,b_v$
satisfying~$b_v\ge 0$ and~$2a_v+b_v=0$. We write~$\sigma^\circ$ for
the $K_0$-stable $\cO$-lattice~$\otimes_{v|p}\sigma_v^\circ$, where
$\sigma_v^\circ=\det^{a_v}\otimes\Sym^{b_v}\cO^2$. We have Hecke
algebras~$\cH(\sigma):=\End_G(\cInd_{K_0}^G\sigma)$,
$\cH(\sigma^\circ):=\End_G(\cInd_{K_0}^G\sigma^\circ)$. 

A \emph{Serre weight} is an irreducible $k$-representation
of~$K_0$. These are of the form~$\otimes_{v|p}\sigmabar_v$, where
~$\sigmabar_v=\det^{a_v}\otimes\Sym^{b_v}k^2$ for integers~$a_v,b_v$
satisfying~$0\le b_v\le p-1$ and~$2a_v+b_v=0$. Note that for
any~$\sigmabar$ there is a unique~$\sigma$
with~$\sigma^\circ\otimes_{\cO}k=\sigmabar$; we say
that~$\sigma^\circ$ lifts~$\sigmabar$.  As explained in the proof of
Lemma~\ref{locfinlocadm}, we have Hecke algebras
$\cH(\sigmabar)\cong\otimes_{v|p}\cH(\sigmabar_v) $,
where~$\cH(\sigmabar_v):=\End_{G_v}(\cInd_{K_v}^{G_v}\sigmabar_v)\cong
k[T_v]$ is a polynomial ring in one variable by~\cite[Prop.\
8]{barthel-livne}.

\subsubsection{Actions of Hecke algebras}\label{sec:actionsofhecke}We now 
describe how to define actions of the Hecke algebras $\cH(\sigmabar)$ and 
$\cH(\sigma^\circ)$ on objects of certain derived categories. 

Let $\sigmabar$ be a Serre weight. Suppose $M$ is a pseudocompact 
$A[[K_0]]$-module with a 
compatible action of $G$, where $A$ is a complete Noetherian local 
$\cO$-algebra with finite residue field which is flat over $\cO$. For example, 
$A$ could be either $\cO[\Delta_Q]$ or $\cO_\infty$.
Then the $A$-module $\sigmabar\otimes_{\cO[[K_0]]} M$ has a natural 
action of $\cH(\sigmabar)$. Indeed, we have isomorphisms
\[(\sigmabar\otimes_{\cO[[K_0]]} M)^\vee \cong 
\Hom^{cts}_{\cO[[K_0]]}(\sigmabar,M^\vee) 
= \Hom_G(\cInd_{K_0}^G\sigmabar,M^\vee)\] by Lemma \ref{lem: tensor dual 
Brumer} and 
Frobenius reciprocity (note that  $M^\vee \in \Mod^{sm}_G(A)$, where the 
definition of this category is recalled in Appendix~\ref{appendix: tensor 
products projective covers}), and 
$\cH(\sigmabar)$ naturally acts on $\Hom_G(\cInd_{K_0}^G\sigmabar,M^\vee)$.

We have a similar story in the derived category. If we let $M^\vee \rightarrow 
I^\bullet$ be an injective resolution of $M^\vee$ in $\Mod^{sm}_G(A)$, then 
each $(I^i)^\vee$ is projective as a pseudocompact $A[[K_0]]$-module (by 
\cite[Prop.~2.1.2]{emordtwo}), and is in particular a flat $\cO[[K_0]]$-module, 
so we have a natural action of $\cH(\sigmabar)$ 
on \[\sigmabar\otimes^\LL_{\cO[[K_0]]} M = 
\sigmabar\otimes_{\cO[[K_0]]}(I^\bullet)^\vee\] in $D(A)$.

Similarly, if $\sigma^\circ$ is a lattice in $\sigma$, we have a natural action 
of $\cH(\sigma^\circ)$ on \[\Hom^{cts}_{\cO[[K_0]]}(\sigma^\circ,I^i) = 
\dirlim_s\Hom_{K_0}(\sigma^\circ/\varpi^s,I^i) = 
\dirlim_s\Hom_{G}(\cInd_{K_0}^G(\sigma^\circ/\varpi^s),I^i)\] for each $n$, 
where the first equality uses Lemma \ref{brumercont}, and 
therefore a natural action of $\cH(\sigma^\circ)$ on 
 \[\sigma^\circ\otimes^\LL_{\cO[[K_0]]} M = 
\sigma^\circ\otimes_{\cO[[K_0]]}(I^\bullet)^\vee\] in $D(A)$.

As a particular example of this construction, we get a natural action of 
$\cH(\sigma^\circ)$ on $\cC(K_0U^p,\sigma^\circ)_\m$, in $D(\cO)$, since we 
have an 
isomorphism \[\cC(K_0U^p,\sigma^\circ)_\m \cong 
\sigma^\circ\otimes^\LL_{\cO[[K_0]]} 
\widetilde{H}_{q_0}(U^p,\cO)_\m.\] Here we are using the part of 
Prop.~\ref{patchedCM} which shows that $\widetilde{H}_i(U^p,\cO)_\m = 0$ for 
$i \ne q_0$. One can also define the action of $\cH(\sigma^\circ)$ on 
$\cC(K_0U^p,\sigma^\circ)$ directly, similarly to the definition of the Hecke 
action at places away from $p$, and this gives the same Hecke action.

We say that a representation~$r:G_{F_v}\to\GL_2(\Qpbar)$ is
crystalline of Hodge type~$\sigma_v$ if it is crystalline with
Hodge--Tate weights $(1-a_v,-a_v-b_v)$, and we
write~$\Rdef_v(\sigma_v)$ for the reduced, $p$-torsion free quotient
of~$\Rdef_v$ corresponding to crystalline deformations of Hodge
type~$\sigma_v$. We
write~$R_p^{\loc}(\sigma):=\wotimes_{v|p}\Rdef_v(\sigma_v)$ and
$\Rbar_\infty(\sigma):=\Rbar_\infty\otimes_{R_p^{\loc}}R_p^{\loc}(\sigma)$. By~\cite[Thm.\
3.3.8]{kisindefrings}, $\Rdef_v(\sigma_v)$ is equidimensional of
Krull dimension~$2$  less than~$\Rdef_v$, so by Lemma~\ref{lem:dim
  Rinfty}, $\Rbar_\infty(\sigma)$ is equidimensional of dimension~$\dim\cO_\infty-l_0$.

We have a homomorphism
$\cH(\sigma)\stackrel{\eta}{\to}R_p^{\loc}(\sigma)[1/p]$, which is the
  tensor product over the places~$v|p$ of the maps $\cH(\sigma_v)\to\Rdef_v(\sigma_v)[1/p]$ defined in
  \cite[Thm. 4.1]{CEGGPSBreuilSchneider}, which interpolates the
  (unramified) local Langlands correspondence.

\begin{prop}\label{prop: local global compatibility for finite level}Assume
	Conjecture~\ref{conj: p-adic local Langlands gives us the action for general
		F}. Then, for any irreducible $E$-representation $\sigma$ of $K_0$, the 
		action of~$R_p^{\loc}$
	on \[\cC(K_0U^p,\sigma^\circ)_\m \in D(\cO)\] factors
	through~$R_p^{\loc}(\sigma)$. Furthermore,
	if~$h\in\cH(\sigma^\circ)$ is such that $\eta(h)\in R_p^{\loc}(\sigma)$, 
	then~$h$ acts
	on~$\cC(K_0U^p,\sigma^\circ)_\m$ via~$\eta(h)$. 
	
	In particular, we get the same statements for the action of $R_p^{\loc}$ 
	and $\cH(\sigma^\circ)$ on the homology groups 
	$H_i(X_{K_0U^p},\sigma^\circ)_\m$ for any $i$.
\end{prop}
\begin{proof}
	As in the proof of Proposition \ref{prop: killing patching variables gives 
	completed homology}, it follows from Lemma
      \ref{lem:flatnessinlimit}  
	that we have a 
        natural quasi-isomorphism (where we regard $\sigma^\circ$ as a right 
        $\cO_\infty[[K_0]]$-module)
	\[\sigma^\circ\otimes_{\cO_\infty[[K_0]]}\widetilde{\cC}(\infty) \rightarrow
	\cC(K_0U^p,\sigma^\circ)_\m.\] Conjecture~\ref{conj: p-adic 
	local Langlands gives us the action for general
		F} implies that $H_{q_0}(\widetilde{\cC}(\infty))$ is a flat 
		$\cO[[K_0]]$-module, so we have an isomorphism in $D(\cO)$
		\[\sigma^\circ\otimes_{\cO_\infty[[K_0]]}\widetilde{\cC}(\infty)
		 = 
		\cO\otimes_{\cO_\infty}^\LL\left(\sigma^\circ\otimes_{\cO[[K_0]]}H_{q_0}(\widetilde{\cC}(\infty))\right)[+q_0]\]

Taking Conjecture~\ref{conj: p-adic 
	local Langlands gives us the action for general
	F} into account, it now suffices to show that the action of $R_p^{\loc}$ on 
	$\sigma^\circ\otimes_{\cO[[K_0]]}P$ factors through 
	$R_p^{\loc}(\sigma)$, and that if~$h\in\cH(\sigma^\circ)$ is such that 
	$\eta(h)\in
R_p^{\loc}(\sigma)$, then~$h$ acts
on~$\sigma^\circ\otimes_{\cO[[K_0]]}P$ via~$\eta(h)$.

 We have
 $\sigma^\circ\otimes_{\cO[[K_0]]}P=\otimes_{v|p}(\sigma_v^\circ\otimes_{\cO[[K_v]]}P_v)$,
 so it suffices to show that the action of~$\Rdef$
 on~$\sigma_v^\circ\otimes_{\cO[[K_v]]}P_v$ factors
 through~$\Rdef_v(\sigma_v)$, and that if~$h_v\in\cH(\sigma_v^\circ)$ is
 such that~$\eta(h_v)\in\Rdef_v(\sigma_v)$, then~$h_v$ acts
 on~$\sigma_v^\circ\otimes_{\cO[[K_v]]}P_v$ via~$\eta(h_v)$. By
 Lemma~\ref{lem: tensor dual Brumer} we have a natural
 isomorphism \[(\sigma_v^\circ\otimes_{\cO[[K_v]]}P_v)^\vee\cong
 \Hom_{\cO[[K_v]]}^{cts}(P_v,(\sigma_v^\circ)^\vee)\] where we note that since 
 $\sigma_v^\circ$ is a finitely generated $\cO[[K_v]]$-module we do not need to 
 take a completed tensor product. Lemma~\ref{brumercont} 
 implies that this is isomorphic to 
 $\dirlim_s\Hom_{\cO[[K_v]]}^{cts}(P_v,(\sigma_v^\circ/\varpi^s)^\vee)$ so we 
 deduce that \[(\sigma_v^\circ\otimes_{\cO[[K_v]]}P_v) \cong 
 \invlim_s\left(\Hom_{\cO[[K_v]]}^{cts}(P_v,(\sigma_v^\circ/\varpi^s)^\vee)\right)^\vee.\]
  \cite[Lem.~4.14]{CEGGPSBreuilSchneider} then shows that we have an 
 isomorphism \[\sigma_v^\circ\otimes_{\cO[[K_v]]}P_v \cong 
 \Hom_{\cO[[K_v]]}^{cts}(P_v,(\sigma_v^\circ)^d)^d\] where $(-)^d$ denotes the 
 Schikhof dual (as defined in \emph{loc.~cit.}). The 
 result now follows
 from~\cite[Cor.\ 6.4, 6.5]{paskunasBM} and~\cite[Prop.\
 6.17]{CEGGPSGL2}.
\end{proof}
\begin{rem}\label{rem: doubledual}
	It follows from the argument appearing at the end of the above proof that 
	if $P$ is a projective pseudocompact $\cO[[K_0]]$-module then we have a 
	natural isomorphism
	\[\sigma^\circ\otimes_{\cO[[K_0]]}P \cong 
	\Hom_{\cO[[K_0]]}^{cts}(P,(\sigma^\circ)^d)^d.\]
\end{rem}
We can also deduce the following modularity lifting theorem from Conjecture~\ref{conj: p-adic local Langlands gives us the action for general
		F}.
\begin{cor}\label{cor: modularity lifting}
	Assume \emph{(}in
	addition to our running assumptions\emph{)}
	Conjecture~\ref{conj: p-adic local Langlands gives us the action for general
		F}. Then, for any irreducible $E$-representation $\sigma$ of $K_0$, 
	$H_{q_0}(X_{K_0U^p},\sigma)_\m$ is a free module of rank $m$ \emph{(}where 
	$m$ is the multiplicity in the 
	statement of 
	Conjecture~\ref{conj: p-adic local Langlands gives us the action for general
		F}\emph{)} over 
	$R_{\overline{\cS}}\otimes_{R^{\loc}_p}R^{\loc}_p(\sigma)[1/p]$ \emph{(}if 
	this 
	ring is 
	non-zero\emph{)}. 
	
	In 
	particular, all characteristic $0$ points of the global crystalline 
	deformation ring
	$R_{\overline{\cS}(\sigma)}:=R_{\overline{\cS}}\otimes_{R^{\loc}_p}R^{\loc}_p(\sigma)$
	 are 
	automorphic, and the maximal $\varpi$-torsion free quotient of 
	$R_{\overline{\cS}(\sigma)}$ is isomorphic to 
	a Hecke algebra acting faithfully on  $H_{q_0}(X_{K_0U^p},\sigma)_\m$.
	
	Moreover, the annihilator of $H_{q_0}(X_{K_0U^p},\sigma^\circ)_\m$ in 
	$R_{\overline{\cS}(\sigma)}$ is nilpotent, and $R_{\overline{\cS}(\sigma)}$ 
	is a finite $\cO$-algebra.
\end{cor}
\begin{proof}
	By \cite[Cor.~6.5]{paskunasBM}, $P(\sigma^\circ) = \sigma^\circ 
	\otimes_{\cO[[K_0]]} P$ is a 
	maximal Cohen--Macaulay module with full support over $R^{\loc}_p(\sigma)$. 
	Since $R^{\loc}_p(\sigma)[1/p]$ is regular it follows that 
	$P(\sigma^\circ)[1/p]$ is locally free with full support over 
	$R^{\loc}_p(\sigma)[1/p]$. In fact, as explained in the proof of 
	\cite[Prop.~6.14]{CEGGPSGL2}, it follows from \cite[Prop.~4.14, 
	2.22]{paskunasBM} that $P(\sigma^\circ)[1/p]$ is locally free of rank one over 
	$R^{\loc}_p(\sigma)[1/p]$. We deduce from 
	Conjecture~\ref{conj: p-adic local Langlands 
		gives us the action for general
		F} that $\sigma^\circ \otimes_{\cO[[K_0]]} 
		H_{q_0}(\widetilde{\cC}(\infty))[1/p]$ is locally
	free of rank $m$ over  
	$\Rbar_\infty\otimes_{R^{\loc}_p}R^{\loc}_p(\sigma)[1/p]$. Reducing mod 
	$\ba$ 
	(and noting that $\Rbar_\infty/\ba \cong R_{\overline{\cS}}$ by Proposition 
	\ref{prop: big R equals T and describing completed homology}) we deduce 
	that 
	$\sigma^\circ \otimes_{\cO[[K_0]]} 
	\widetilde{H}_{q_0}(X_{U^p},\cO)_\m[1/p]$ is locally free of rank $m$ 
	over 
	$R_{\overline{\cS}}\otimes_{R^{\loc}_p}R_p(\sigma)[1/p]$. We complete
	the proof by noting that we have a natural isomorphism \[\sigma^\circ 
	\otimes_{\cO[[K_0]]} \widetilde{H}_{q_0}(X_{U^p},\cO)_\m \cong 
	H_{q_0}(X_{K_0U^p},\sigma^\circ)_\m\]  so 
	$R_{\overline{\cS}}\otimes_{R^{\loc}_p}R_p(\sigma)[1/p]$ is a 
	finite--dimensional algebra  (hence semi-local) and therefore the locally 
	free module of rank $m$, $H_{q_0}(X_{K_0U^p},\sigma^\circ)_\m$, is in fact 
	free of rank $m$.
	
	The moreover part follows from \cite[Lem.~2.2]{tay}, since $\sigma^\circ 
	\otimes_{\cO[[K_0]]} 
	H_{q_0}(\widetilde{\cC}(\infty))$ is a nearly faithful 
	$\Rbar_\infty(\sigma)$-module so reducing mod $\ba$ shows that 
	$H_{q_0}(X_{K_0U^p},\sigma^\circ)_\m$ is a 
	nearly faithful 
	$R_{\overline{\cS}(\sigma)}$-module, as well as a finite $\cO$-module.
\end{proof}
\begin{rem}
	As discussed in Remark~\ref{rem: could allow types as well as weights}, we 
	could 
	work with general potentially semistable types, and then the
	proof of Corollary~\ref{cor: modularity lifting} goes through
	unchanged to give an 
	automorphy 
	lifting theorem for arbitrary potentially semistable lifts of $\rhobar_\m$ with distinct Hodge--Tate weights, 
	which satisfy the conditions imposed by $\overline{\cS}$ at places $v \nmid 
	p$.
	\end{rem}
	
	\begin{rem}\label{emertonfmarg}
		Using Proposition~\ref{prop: big R equals T and describing completed 
		homology}, we can give an alternative argument to show that 
		Conjecture~\ref{conj: p-adic local Langlands gives us the action for 
		general
			F} implies many cases of the Fontaine--Mazur conjecture, in exactly 
			the same way that Emerton deduces \cite[Corollary 
			1.2.2]{emerton2010local} from his local-global compatibility 
			result. If we assume Conjecture~\ref{conj: p-adic local Langlands 
			gives us the action for general
			F} then any characteristic zero point of $R_{\overline{\cS}}$ whose 
			associated Galois representation is de Rham with distinct 
			Hodge--Tate weights at each place $v|p$ is automorphic, in the 
			sense that its associated system of Hecke eigenvalues appears in 
			$H_{q_0}(X_{KU^p},\sigma)_\m$ for some compact open $K \subset K_0$ 
			and some irreducible $E$-representation $\sigma$ of $K_0$.
		
		Moreover, again assuming Conjecture~\ref{conj: p-adic local Langlands 
		gives us the action for general
			F} and following Emerton's argument, we can show that any 
			characteristic zero point of $R_{\overline{\cS}}$ whose associated 
			Galois representation is trianguline at each place $v|p$ arises 
			from an overconvergent $p$-adic automorphic form of finite slope, 
			in the sense that its associated system of Hecke eigenvalues 
			appears in the Emerton--Jacquet module
			$J_B(((\widetilde{H}_{q_0}(X_{U^p},\cO)_\m)^d[\frac{1}{p}])^{an})$.
	\end{rem}
 \begin{rem}
   \label{rem: derivedgal}Assuming Conjecture~\ref{conj: p-adic local Langlands 
   gives us the action for general
   	F}, we 
   obtain an action of the graded $R_{\overline{\cS}(\sigma)}$-algebra 
   $\Tor_*^{R^{\loc}_p}(R_{\overline{\cS}},R_p^{\loc}(\sigma)) = 
   \Tor_*^{R_\infty}(R_\infty/\ba,R_\infty(\sigma))$ on the graded module
   \[H_{*}(X_{K_0U^p},\sigma^\circ)_\m = 
   H_*\left(R_\infty/\ba\otimes^\LL_{R_\infty}\left(\sigma^\circ\otimes_{\cO[[K_0]]}H_{q_0}(\widetilde{\cC}(\infty))\right)\right).\]
When $R_p^{\loc}(\sigma)$ is the representing object of a Fontaine--Laffaille 
moduli problem, the groups  
$\Tor_i^{R^{\loc}_p}(R_{\overline{\cS}},R_p^{\loc}(\sigma))$ are the 
homotopy groups of a derived Galois deformation ring (since 
$R_{\overline{\cS}}$ 
is a complete intersection of the predicted dimension, see the discussion in 
\cite[\S 1.3]{galvenk}) and the action of the graded algebra on 
$H_{*}(X_{K_0U^p},\sigma^\circ)_\m$ is free. This is an example of the main 
theorem of \cite{galvenk}. Note that it is not obvious that the action of 
$\Tor_*^{R^{\loc}_p}(R_{\overline{\cS}},R_p^{\loc}(\sigma))$ on 
$H_{*}(X_{K_0U^p},\sigma^\circ)_\m$ is independent of the choice of 
non-principal ultrafilter made to carry out the patching. Under some additional 
hypotheses, this independence is shown in \cite{galvenk}, by comparing the 
action of the derived Galois deformation ring with the action of a derived 
Hecke 
algebra.  
\end{rem}

 Proposition~\ref{prop: local global compatibility for finite level} shows that 
Conjecture~\ref{conj: p-adic local Langlands 
gives us the action for general
	F} implies a local--global compatibility statement at $p$. We are now going 
	to formulate a conjectural local--global compatibility statement which will 
	be sufficiently strong to imply Conjecture~\ref{conj: p-adic local 
	Langlands gives us the action for general
	F}. 

Note that for any 
Taylor--Wiles datum $(Q,(\gamma_{v,1},\dots,\gamma_{v,n})_{v\in Q})$,
 (\ref{eqn:R to T}) gives
an action of~$R_p^{\loc}$ on the complex \[\widetilde{\cC}(Q):= 
\invlim_{U_p,s}\cC(U_pU_1^p(Q),s)_{\m_{Q,1}}\] in $D(\cO[\Delta_Q])$. For 
any $\sigma$, the complex $\sigma^\circ\otimes_{\cO[[K_0]]}\widetilde{\cC}(Q)$ 
is naturally quasi-isomorphic (in particular, the quasi-isomorphism is 
$\cO[\Delta_Q]$-equivariant) to 
$\cC(K_0U_1^p(Q),\sigma^\circ)_{\m_{Q,1}}$ Again, 
this is deduced 
from Lemma~\ref{lem:flatnessinlimit}. We therefore obtain an action of 
$R_p^{loc}$ on $\cC(K_0U_1^p(Q),\sigma^\circ)_{\m_{Q,1}}$ in 
$D(\cO[\Delta_Q])$. We also have a natural action of $\cH(\sigma^\circ)$ on 
$\sigma^\circ\otimes_{\cO[[K_0]]}\widetilde{\cC}(Q)$ in 
$D(\cO[\Delta_Q])$, as described in section \ref{sec:actionsofhecke}. To 
apply the construction of that section, we must note that 
$\widetilde{\cC}(Q)$ has homology concentrated in degree $q_0$. Indeed, 
assumption $(a)$ in Proposition \ref{patchedCM} implies that the minimal 
resolution $\cF$ of $\widetilde{\cC}(Q)$ as a complex of 
$\cO[\Delta_Q][[K_1]]$-modules is concentrated in degrees $[q_0,q_0+l_0]$. We 
also have $j_{\cO[[K_0]]}(H_{q_0}(\widetilde{\cC}(Q)) \ge l_0$ because the 
quotient module $\cO\otimes_{\cO[\Delta_Q]}H_{q_0}(\widetilde{\cC}(Q)) \cong 
\widetilde{H}_{q_0}(X_{U^p},\cO)_\m$ has grade $l_0$ (by Proposition 
\ref{patchedCM}). Applying Lemma \ref{lem:CG} to the complex $\cF[-q_0]$, we 
deduce that $\widetilde{\cC}(Q)$ has homology concentrated in degree $q_0$.

Proposition~\ref{prop: local global compatibility for finite level}
motivates the following conjecture, which is a further refinement of
Conjectures~\ref{galconj} and~\ref{conj: compatibility at Q}.

\begin{conj}\label{conj: crystalline local global}For any Taylor--Wiles datum 
  $(Q,(\gamma_{v,1},\dots,\gamma_{v,n})_{v\in Q})$, and any irreducible 
  $E$-representation of $K_0$, $\sigma$,
  the action of~$R_p^{\loc}$ 
  on~$H_*(X_{K_0U_1^p(Q)},\sigma^\circ)_{\m_{Q,1}}$ factors through 
  $R_p^{\loc}(\sigma)$. Furthermore,
  if~$h\in\cH(\sigma^\circ)$ is such that $\eta(h)\in R_p^{\loc}(\sigma)$, 
  then~$h$ acts
  on~$H_{q_0}(X_{K_0U_1^p(Q)},\sigma^\circ)_{\m_{Q,1}}$ via~$\eta(h)$. 
  \end{conj}
\begin{rem}
  \label{rem: surprisingly strong local-global compatibility
    conjecture}The reader may be surprised by Conjecture~\ref{conj:
    crystalline local global}, which in particular implies that the
  factors at places dividing~$p$ of the Galois representations associated to 
  torsion
  classes in the homology groups $H_*(X_{K_0U_1^p(Q)},\sigma^\circ)_{\m_{Q,1}}$ 
  are controlled by the crystalline deformation
  rings, which are defined purely in terms of representations
  over~$p$-adic fields (and are $p$-torsion free by
  fiat). Nonetheless, since we believe that Conjecture~\ref{conj: p-adic local Langlands gives us the action for general
    F} is reasonable, Proposition~\ref{prop: local global
    compatibility for finite level} gives strong evidence for
  Conjecture~\ref{conj: crystalline local global}; similarly,
  \cite[Cor.\ 6.5]{paskunasBM} shows that the crystalline deformation
  rings can be reconstructed from~$P$, and this alternative
  construction makes it more plausible that they can also control
  integral phenomena. We are also optimistic that the
  natural analogues of Conjecture~\ref{conj: crystalline local global}
  should continue to hold beyond the case of~$\GL_2(\Qp)$.
  \end{rem}

\begin{rem}
  \label{rem: could allow types as well as weights}We have avoided
the notational clutter that would result from allowing non-trivial inertial types, but the
  natural generalisation of Proposition~\ref{prop: local global
    compatibility for finite level} to more general potentially
  crystalline (or even potentially semistable) representations can be
  proved in the same way. The axioms in Section~\ref{subsec:
    arithmetic actions} below only refer to crystalline representations;
  accordingly, Corollary~\ref{cor: local global for PGL2 Qp assuming
    crystalline local global} below shows that (in conjunction with our
  other assumptions) Conjecture~\ref{conj: crystalline local global}
  implies a local-global compatibility result for general potentially
  semistable representations. (It is perhaps also worth remarking that
  rather than assuming Conjecture~\ref{conj: crystalline local
    global}, we could instead assume a variant for arbitrary
  potentially Barsotti--Tate representations, or indeed any variant to
  which we can apply the ``capture'' machinery of~\cite[\S 2.4]{MR3272011}.)\end{rem}

In the rest of this section we will explain
(following~\cite{CEGGPSGL2}) that Conjecture~\ref{conj: crystalline local 
global} implies Conjecture~\ref{conj: p-adic local Langlands gives us the 
action for general
	F}.\subsection{Arithmetic actions}\label{subsec: arithmetic
  actions}We now introduce variants of the axioms of~\cite[\S 3.1]{CEGGPSGL2},
and prove Proposition~\ref{prop: all arithmetic actions are p from
  Colmez}, which shows that if the axioms are satisfied 
  for~$H_{q_0}(\widetilde{\cC}(\infty))$,
 then Conjecture~\ref{conj: p-adic local
  Langlands gives us the action for general F} holds. We will show in
Section~\ref{subsec: local global following CEGGPS} that (under our
various hypotheses)~$H_{q_0}(\widetilde{\cC}(\infty))$ indeed satisfies
these axioms. 

Fix an integer $g\ge 0$ and set
$\Rbar_\infty=R_p^{\loc}\wotimes_\cO\cO[[x_1,\dots,x_g]]$. (Of course, in our application
to~$H_{q_0}(\widetilde{\cC}(\infty))$ we will take~$g$ as in
Section~\ref{subsec: local global compatibility conjecture PGL2}.)

Then an
\emph{$\cO[G]$-module with an arithmetic action of $\Rbar_\infty$} 
is by definition a non-zero $\Rbar_\infty[G]$-module $M_\infty$
satisfying the  following axioms (AA1)--(AA4).
\begin{itemize}
\item[(AA1)] $M_\infty$ is a finitely generated
  $\Rbar_\infty[[K_0]]$-module.
\item[(AA2)] $M_\infty$ is projective in the category of pseudocompact
  $\cO[[K_0]]$-modules.
\end{itemize}

Set \[M_\infty(\sigma^\circ):=\sigma^\circ\otimes_{\cO[[K_0]]}M_\infty. \]This is
a finitely generated $\Rbar_\infty$-module by (AA1). For each $\sigma^\circ$, we have a natural action of
$\cH(\sigma^\circ)$ on $M_\infty(\sigma^\circ)$, and thus of
$\cH(\sigma)$ on $M_\infty(\sigma^\circ)[1/p]$.

\begin{itemize}
\item[(AA3)]For any $\sigma$, the action of $\Rbar_\infty$ on
  $M_\infty(\sigma^\circ)$ factors through
  $\Rbar_\infty(\sigma)$. Furthermore,
  $M_\infty(\sigma^\circ)$ is maximal Cohen--Macaulay over
  $\Rbar_\infty(\sigma)$. \end{itemize}

\begin{itemize}
\item[(AA4)] For any $\sigma$, the action of $\cH(\sigma)$ on
  $M_\infty(\sigma^\circ)[1/p]$ is given by the
  composite \[\cH(\sigma)\stackrel{\eta}{\to}R_p^{\loc}(\sigma)[1/p]\to \Rbar_\infty(\sigma)[1/p].\]
 
   \end{itemize}\begin{rem}
     \label{rem: differences in definition of AA}Our axioms are not
     quite the obvious translation of the axioms of~\cite[\S
     3.1]{CEGGPSGL2} to our setting. Firstly, our definition of
     $M_\infty(\sigma^\circ)$ is different; however, by Remark~\ref{rem: 
     doubledual} it is equivalent to the
     definition given there. More significantly, in~(AA3) we do not require
     that $M_\infty(\sigma^\circ)[1/p]$ is locally free of rank one
     over its support.

     Since $R_p^{\loc}(\sigma)[1/p]$ is equidimensional and regular (by~\cite[Thm.\ 
     3.3.8]{kisindefrings} and~\cite[Lem.\
     3.3]{blght}), $M_\infty(\sigma^\circ)[1/p]$ is (being maximal
     Cohen--Macaulay by~(AA3)) locally free over its support.  (This
     is standard, but for completeness we give an argument. Write
     $R=R_p^{\loc}(\sigma)[1/p]$, $M=M_\infty(\sigma^\circ)[1/p]$, and
     let $\frakp \in \mathrm{Supp}_{R}(M)$. By \cite[Ch.~0,
     Cor.~16.5.10]{EGAIV1}, $M_\frakp$ is Cohen--Macaulay over
     $R_\frakp$ and we have
     \[\dim_R(M) = \dim_R(M/\frakp M) + \dim_{R_\frakp}(M_\frakp).\]
     By \cite[Ch.~0, Prop.~16.5.9]{EGAIV1} we have
     $\dim_R(M/\frakp M) = \dim R/\frakp$ and since $M$ is maximal
     Cohen--Macaulay over $R$ we have $\dim_R(M) = \dim R$. Since
     $\dim R_\frakp + \dim R/\frakp \le \dim R$ (\cite[Ch.~0,
     16.1.4.1]{EGAIV1}) we deduce that
     $\dim_{R_\frakp}(M_\frakp) \ge \dim R_\frakp$ and therefore
     $\dim_{R_\frakp}(M_\frakp) = \dim R_\frakp$. So $M_\frakp$ is
     maximal Cohen--Macaulay over $R_\frakp$. Since $R$ is regular, and maximal Cohen--Macaulay modules over 
regular local rings are free 
\cite[\href{http://stacks.math.columbia.edu/tag/00NT}{Tag 
00NT}]{stacks-project}, we deduce that $M[1/p]$ is locally free over 
$\Rbar(\sigma)[1/p]$.)

     We do not make any prescription on the rank of
     $M_\infty(\sigma^\circ)[1/p]$ over its support (or even require this
     rank to be constant), and this is reflected in the
     multiplicity~$m$ in Proposition~\ref{prop: all arithmetic actions
       are p from Colmez} below.
   \end{rem}

We now follow the approach
of~\cite{CEGGPSGL2} to show that any~$\cO[G]$-module with an arithmetic
action of~$\Rbar_\infty$ is obtained from the $p$-adic local Langlands
correspondence for~$\GL_2(\Qp)$. The following result shows that in order to establish
Conjecture~\ref{conj: p-adic local Langlands gives us the action for
  general F}, it is enough to show that the action
of~$\Rbar_\infty[G]$ on~$H_{q_0}(\widetilde{\cC}(\infty))$ is
arithmetic. We will follow the proof of~\cite[Thm.\ 4.30]{CEGGPSGL2}
very closely, indicating what changes are necessary to go from
their~$G$ (which equals $\GL_2(\Qp)$) to our~$G$ (which is a product
of copies of~$\PGL_2(\Qp)$). We also need to make some additional
adjustments due to the absence of a rank one assumption in axiom~(AA3).

\begin{prop}
  \label{prop: all arithmetic actions are p from Colmez}Let~$M_\infty$
  be an $\cO[G]$-module with an arithmetic action
  of~$\Rbar_\infty$. Then for some integer $m\ge 1$ there is an isomorphism of
  $\Rbar_\infty\left[G\right]$-modules
  \[M_\infty\cong \Rbar_\infty\wotimes_{R_p^{\loc}}P^{\oplus
      m}.\]
\end{prop}
\begin{proof}
  As we have already remarked, we will
  closely follow the arguments of~\cite[\S 4]{CEGGPSGL2}. To orient
  the reader unfamiliar with~\cite{CEGGPSGL2}, we make some brief
  preliminary remarks. As a consequence of the results
  of~\cite{paskunasimage,paskunasBM}, it is not hard to show that the natural
  action of~$\Rbar_\infty[G]$ on $\Rbar_\infty\wotimes_{R_p^{\loc}}P^{\oplus
      m}$ is an arithmetic action. We show that~$M_\infty$ is a
    projective object of~$\mathfrak{C}_G(\cO)$, and that its cosocle only contains
    copies of~$\pi^\vee:=\wotimes_{v|p}\pi_v^\vee$. From this we can
    deduce the existence of an isomorphism
    of~$\cO[[x_1,\dots,x_g]]$-modules of the required kind, and we
    need only check that it is~$R_p^{\loc}$-linear. By a density
    argument, we reduce to showing that the corresponding isomorphism
    for~$M_\infty(\sigma)$ is $R_p^{\loc}(\sigma)$-linear (for
    each~$\sigma$). This in turn follows from~(AA4) (and the fact that
    ~$\eta:\cH(\sigma)\to R_p^{\loc}(\sigma)[1/p]$ becomes an
    isomorphism upon passing to completions at maximal ideals, cf.\
    \cite[Prop.\ 2.13]{CEGGPSGL2}; this is due to the uniqueness of
    the Hodge filtration for crystalline representations, which is a
    phenomenon unique to the case of~$\GL_2(\Qp)$).

We now begin the proof
proper. Set~$\pi^\vee:=\wotimes_{v|p}\pi_v^\vee$; by Lemma~\ref{lem:
  tensor products projective envelopes}, $P$ is a projective envelope
of~$\pi^\vee$ in~$\mathfrak{C}_G(\cO)$. The argument of~\cite[Prop.\
4.2]{CEGGPSGL2} goes through essentially unchanged, and shows that for each Serre
weight~$\sigmabar$ with corresponding lift~$\sigma^\circ$, we have:
\begin{enumerate}\item if~$M_\infty(\sigma^\circ)\ne 0$, then it is a
  free~$\Rbar_\infty(\sigma)$-module of some rank~$m$. Furthermore, the
  action of~$\cH(\sigmabar)$ on~$M_\infty(\sigmabar)$ factors through
  the natural map~$R_p^{\loc}(\sigma)/\varpi\to
  \Rbar_\infty(\sigma)/\varpi$, and~$M_\infty(\sigmabar)$ is a flat~$\cH(\sigmabar)$-module.

\item If~$M_\infty(\sigma^\circ)\ne 0$, then there is a homomorphism
  $\cH(\sigmabar)\to k$ such that
  $\pi\cong\cInd_{K_0}^G\sigmabar\otimes_{\cH(\sigmabar)}k$. Accordingly,
  $\Hom_G(\pi,M_\infty^\vee)^\vee\cong M_\infty(\sigmabar)\otimes_{\cH(\sigmabar)}k$.
\item If~$\pi'$ is an irreducible smooth $k$-representation of~$G$
  then~$\Hom_G(\pi',M_\infty^\vee)\ne 0$ if and only
  if~$\pi'\cong\pi$.
\end{enumerate}
Since~$\cH(\sigmabar)=\otimes_{v|p}\cH(\sigmabar_v)\cong
k[T_v]_{v|p}$, the proofs of~\cite[Lem.\ 4.10, Lem.\ 4.11, Thm.\ 4.15]{CEGGPSGL2} go
through with only notational changes, so that~$M_\infty$ is a
projective object of~$\mathfrak{C}_G(\cO)$. 

Write~$A=\cO[[x_1,\dots,x_g]]$, and choose a
homomorphism~$A\to\Rbar_\infty$ inducing an isomorphism
$R_p^{\loc}\wotimes_\cO A\cong\Rbar_\infty$. We claim that there is
an isomorphism
in~$\mathfrak{C}_G(A)$\numequation\label{eqn: Minfty iso}
M_\infty\cong A\wotimes_{\cO} P^{\oplus m}.\end{equation} By~(3)
above, all of the irreducible subquotients
of~$\operatorname{cosoc}_{\mathfrak{C}_G(\cO)}M_\infty$ are isomorphic
to~$\pi^\vee$, so by~\cite[Prop.\ 4.19, Rem.\ 4.21]{CEGGPSGL2} it is
enough to show that~$\Hom_G(\pi,M_\infty^\vee)^\vee$ is a
free~$A/\varpi$-module of rank~$m$. To see this, note that by ~(2)
above we have
$\Hom_G(\pi,M_\infty^\vee)^\vee\cong
M_\infty(\sigmabar)\otimes_{\cH(\sigmabar)}k$, which by~(1) is a free
$\Rbar_\infty(\sigma)\otimes_{\cH(\sigmabar)}k$-module of
rank~$m$. By~(1) again (together with~\cite[Lem.\ 2.14]{CEGGPSGL2}),
the map $A\to\Rbar_\infty$ induces an isomorphism
$A/\varpi\cong \Rbar_\infty(\sigma)\otimes_{\cH(\sigmabar)}k$, as
required.

It remains to show that~(\ref{eqn: Minfty iso})
is~$R_p^{\loc}$-linear. We claim that the action of~$\Rbar_\infty$
on~$ A\wotimes_{\cO} P^{\oplus m}$ is arithmetic; admitting this claim, the
proofs of~\cite[Thm.\ 4.30, 4.32]{CEGGPSGL2} go over with only minor
notational changes to show the required $R_p^{\loc}$-linearity. 

It is obviously enough to show that the action of~$R_p^{\loc}$ on~$P$
is an arithmetic action (with~$g=0$). (AA1)
holds by the topological version of Nakayama's lemma (since
$\wotimes_{v|p}\kappa_v^\vee$ is a finitely
generated~$k[[K_0]]$-module), while (AA2) holds by~\cite[Cor.\
5.3]{paskunasBM}. (AA3) holds by~\cite[Cor.\ 6.4, 6.5]{paskunasBM},
while~(AA4) follows from the main result of~\cite{paskunasimage}
exactly as in the proof of~\cite[Prop.\ 6.17]{CEGGPSGL2}.
\end{proof}

\subsection{Local-global compatibility}\label{subsec: local global
  following CEGGPS}We now discuss the axioms (AA1)--(AA4) in the case
$M_\infty=H_{q_0}(\widetilde{\cC}(\infty))$.

\begin{prop}
  \label{prop: arithmetic action on patched module}Assume \emph{(}in
  addition to our running assumptions\emph{)}
  Conjecture~\ref{conj: crystalline local global}. Then the action of ~$\Rbar_\infty[G]$
  on~$H_{q_0}(\widetilde{\cC}(\infty))$ is arithmetic.
\end{prop}
\begin{proof}
  Certainly~$H_{q_0}(\widetilde{\cC}(\infty))$ is finitely generated
  over~$R_\infty[[K_0]]$, by
  Proposition~\ref{completionofpatchedcomplexes}~(2) and
  Remark~\ref{liftOtoR}, so axiom~(AA1) holds. 
  
  Next we show that the $\Rbarinfty$ action on  
  $H_{i}(\sigma^{\circ}\otimes_{\cO[[K_0]]}\widetilde{\cC}(\infty))$ factors 
  through $\Rbarinfty(\sigma)$ for all $i$. Indeed, by \ref{goodinvlim}, we 
  have natural isomorphisms  
  \[H_{i}(\sigma^{\circ}\otimes_{\cO[[K_0]]}\widetilde{\cC}(\infty)) \cong 
  \invlim_{U_p,J}H_{i}(\sigma^{\circ}\otimes_{\cO[[K_0]]}{\cC}(U_p,J,\infty))\] 
  where the inverse limit is taken over pairs $(J,U_p)$ such that $U_p$ acts 
  trivially on $\sigma^{\circ}\otimes_\cO \cO_\infty/J$. Each homology group 
  $H_{i}(\sigma^{\circ}\otimes_{\cO[[K_0]]}{\cC}(U_p,J,\infty))$ can be 
  obtained by applying the ultraproduct construction to the groups 
  $H_{i}(\sigma^{\circ}\otimes_{\cO[[K_0]]}{\cC}(U_p,J,N))$, and it follows 
  from Conjecture~\ref{conj: crystalline local global} that the action of 
  $\Rbarinfty$ on all these groups factors through $\Rbarinfty(\sigma)$. It 
  follows in the same way from Conjecture~\ref{conj: crystalline local global} 
  that if~ $h\in\cH(\sigma^\circ)$ is such that
  $\eta(h)\in R_p^{\loc}(\sigma)$, then~$h$ acts on
  $H_{q_0}(\sigma^{\circ}\otimes_{\cO[[K_0]]}\widetilde{\cC}(\infty))$
  via~$\eta(h)$, so axiom~(AA4) holds. 
  
  We can now apply Lemma~\ref{lem:CG} (or
  \cite[Lem.~6.2]{1207.4224}) to the
  complex 
  of~$\cO_\infty$-modules~${\sigma^\circ}\otimes_{\cO[[K_0]]}\widetilde{\cC}(\infty)$
  (more precisely, we replace $\widetilde{\cC}(\infty)$ by a quasi-isomorphic 
  complex of finite projective modules in degrees $[q_0,q_0+l_0]$, which we can 
  do by Proposition \ref{patchedCM}(2)). As in the  proof of~\cite[Thm.\
  6.3]{1207.4224}, since the action of $\cO_\infty$ on
  $H_*({\sigma^\circ}\otimes_{\cO[[K_0]]}\widetilde{\cC}(\infty))$ factors
  through~$\Rbar_\infty(\sigma)$, and   
  $\dim
  \Rbar_\infty(\sigma)=\dim\cO_\infty-l_0$, we have 
  $j_{\cO_{\infty}}(H_*({\sigma^\circ}\otimes_{\cO[[K_0]]}\widetilde{\cC}(\infty)))\ge
  l_0$. We deduce that the complex 
  $\sigma^{\circ}\otimes_{\cO[[K_0]]}\widetilde{\cC}(\infty)$ has non-zero 
  homology only in degree $q_0$, and that 
  $H_{q_0}(\sigma^{\circ}\otimes_{\cO[[K_0]]}\widetilde{\cC}(\infty)) = 
  \sigma^{\circ}\otimes_{\cO[[K_0]]}H_{q_0}(\widetilde{\cC}(\infty))$ is 
  maximal 
  Cohen--Macaulay over $\Rbarinfty(\sigma)$. We have now established that axiom 
  (AA3) holds. 

Finally, it remains to check~(AA2). By~\cite[Prop.\ 3.1]{brumer}, it is enough 
to show
  that for each Serre weight~$\overline{\sigma}$ we
  have~$\Tor_1^{\cO[[K_0]]}(\overline{\sigma},H_{q_0}(\widetilde{\cC}(\infty)))=0$.
   Once again, we apply
  Lemma~\ref{lem:CG} (or
    \cite[Lem.~6.2]{1207.4224}) --- this time to the
  complex 
  of~$\cO_\infty/\varpi$-modules~$\overline{\sigma}\otimes_{\cO[[K_0]]}\widetilde{\cC}(\infty)$.
  We see that it suffices to prove that
  $j_{\cO_{\infty}/\varpi}(H_*(\overline{\sigma}\otimes_{\cO[[K_0]]}\widetilde{\cC}(\infty)))\ge
  l_0$. We let $\sigma^\circ$ be the lift of $\overline{\sigma}$. From what we 
  have already shown about the complex 
  ${\sigma^\circ}\otimes_{\cO[[K_0]]}\widetilde{\cC}(\infty)$ we deduce that 
  we have  
  \[H_{q_0}(\overline{\sigma}\otimes_{\cO[[K_0]]}\widetilde{\cC}(\infty)) = 
  \cO/\varpi\otimes_\cO 
  H_{q_0}({\sigma^\circ}\otimes_{\cO[[K_0]]}\widetilde{\cC}(\infty))\] and 
  \[H_{q_0+1}(\overline{\sigma}\otimes_{\cO[[K_0]]}\widetilde{\cC}(\infty)) = 
  \Tor_1^{\cO}(\cO/\varpi, 
  H_{q_0}({\sigma^\circ}\otimes_{\cO[[K_0]]}\widetilde{\cC}(\infty)))\] with 
  all 
  other homology groups vanishing. 
  
  The action of~$\cO_\infty/\varpi$ on
  these two groups factors
  through~$\Rbar_\infty(\sigma)/\varpi$, since the action on 
  $H_{q_0}({\sigma^\circ}\otimes_{\cO[[K_0]]}\widetilde{\cC}(\infty))$ factors 
  through $\Rbar_\infty(\sigma)$, and $\dim
  \Rbar_\infty(\sigma)/\varpi=\dim\cO_\infty/\varpi-l_0$, so we deduce the 
  desired inequality for
  $j_{\cO_{\infty}/\varpi}(H_*(\overline{\sigma}\otimes_{\cO[[K_0]]}\widetilde{\cC}(\infty)))$.
    
\end{proof}

\begin{cor}
  \label{cor: local global for PGL2 Qp assuming crystalline local global}Assume \emph{(}in
  addition to our running assumptions\emph{)}
  Conjecture~\ref{conj: crystalline local global}. Then Conjecture~\ref{conj: p-adic local Langlands gives us the action for general
    F} holds. In particular, we obtain as consequences the `big $R = \TT$' 
    result of Proposition~\ref{prop: big R equals T and describing completed 
    homology} and the automorphy lifting result of Corollary \ref{cor: 
    modularity lifting}.
\end{cor}
\begin{proof}
  This is immediate from Propositions~\ref{prop: all arithmetic
    actions are p from Colmez} and~\ref{prop: arithmetic action on patched module}.
\end{proof}

\subsection{The totally real and imaginary quadratic
  cases}\label{subsec: totally real or imaginary quadratic}We conclude
by discussing the cases in which unconditional results seem most in
reach. If~$F$ is totally real, then~$l_0=0$, and the existence of
Galois representations is known; the only assumption that is not
established is assumption~(a) of Proposition~\ref{patchedCM}, that the
homology groups ${H}_i(X_{U^pK_1},k)_\m$ vanish for~$i\ne q_0$. It
might be hoped that a generalisation of the results
of~\cite{2015arXiv151102418C} to non-compact Shimura varieties could
establish this. Of course the totally real cases where~$l_0=0$ are
less interesting from the point of view of this paper, as they could
already have been studied using the methods
of~\cite{CEGGPSBreuilSchneider}.

If~$F$ is imaginary quadratic, then the biggest obstacle to
unconditional results is Conjecture~\ref{conj: crystalline local
  global}; indeed, as explained in Remarks~\ref{rem: existence of
  Galois repns in CM case} and~\ref{rem: compatibility at Q in CM
  case}, the other hypotheses on the Galois representations seem to be
close to being known, and as explained in Remark~\ref{rem: assumptions
  abc for IQF}, assumptions~(a) and~(b) of Proposition~\ref{patchedCM}
are known in this case.

\renewcommand{\theequation}{\Alph{section}.\arabic{subsection}}

\appendix
\section{Non-commutative algebra}\label{sec: non commutative
  algebra}In this section we make some definitions and establish some results for non-commutative Iwasawa algebras which generalise standard facts about complete regular local rings. Section \ref{basicdefs} contains the basic definitions which will be needed for discussing our results on patching completed homology. 
\subsection{Depth and dimension}\label{basicdefs}
\begin{adefn}\label{def: gradedepth}
	Let $A$ be a ring and let $M$ be a left or right $A$-module. We denote the 
	\emph{projective dimension} of $M$ over $A$ by $\pd_A(M)$. We define the 
	\emph{grade} $j_A(M)$ of $M$ over $A$ by \[j_A(M) = 
	\inf\{i: \Ext^i_A(M,A) \ne 0\}.\] If all the $\Ext^i_A(M,A)$ vanish we have 
	$j_A(M)=\infty$. If $A$ is local with maximal ideal 
	$\m_A$, then we define the \emph{depth} $\dph_A(M)$ of $M$ by \[\dph_A(M) = 
	\inf\{i:\Ext^i_A(A/\m_A,M)\ne 0\}.\] Similarly, if all the 
	$\Ext^i_A(A/\m_A,M)$ vanish we set $\dph_A(M) = \infty$.
	
	A Noetherian ring $A$ is called \emph{Auslander--Gorenstein} if it has 
	finite left and right injective dimension and if for any finitely generated 
	left or right $A$-module $M$, any integer $m$, and any submodule $N \subset 
	\Ext^m_A(M,A)$, we have $j_A(N) \ge m$. 
	
	An Auslander--Gorenstein ring is called \emph{Auslander regular} if it has finite global dimension.
	
	Finally, let $A$ be an Auslander regular ring and let $M$ be a finitely generated left $A$-module. We define the \emph{dimension} $\delta_A(M)$ of $M$ over $A$ by \[\delta_A(M) = \mathrm{gld}(A) - j_A(M),\] where $\mathrm{gld}(A)$ is the global dimension of $A$.
\end{adefn}

Let $K$ be a compact $p$-adic analytic group. We are going to apply
the above definitions for $A = \cO[[K]]$, the Iwasawa algebra
of $K$ with coefficients in $\cO$. Note that taking inverses of group
elements induces an isomorphism between $\cO[[K]]$ and its
opposite ring, so there is an equivalence between the categories of
left and right $\cO[[K]]$-modules. 

$\cO[[K]]$ is Noetherian, and when $K$ is moreover a pro-$p$ group, $\cO[[K]]$ is a
local ring with $\cO[[K]]/\m_{\cO[[K]]} =
k$.
\begin{aremark}
  \label{rem: grade and codimension}If~$M$ is an $\cO[[K]]$-module,
  then $j_{\cO[[K]]}(M)$ is sometimes referred to as the
  \emph{codimension} of~$M$ (cf.\ \cite[\S 1.2]{MR2905536}).
\end{aremark}

 When~$K$ is pro-$p$ and torsion-free, Venjakob \cite{Venjakob} has established that $\cO[[K]]$ has nice homological properties, which are summarised in the next proposition.

\begin{aprop}[Venjakob]\label{venja}
	Let $K$ be compact $p$-adic analytic group which is torsion free and 
	pro-$p$. Let $\Lambda = \cO[[K]]$ and let $M$ be a finitely generated 
	$\Lambda$-module. 
	\begin{enumerate}
		\item $\Lambda$ is Auslander regular with global
                  dimension $\mathrm{gld}(\Lambda)$ and depth $\dph_\Lambda(\Lambda)$ both equal to $1+\dim(K)$.
		\item The Auslander--Buchsbaum equality holds for $M$: \[\pd_\Lambda(M) + \dph_\Lambda(M) = \dph_\Lambda(\Lambda) = 1 + \dim(K).\]
		\item We have \[\pd_\Lambda(M) = \max\{i : \Ext^i_\Lambda(M,\Lambda) \ne 0\}.\] In particular, we have $\pd_\Lambda(M) \ge j_\Lambda(M)$.
	\end{enumerate}
	\end{aprop}
\begin{proof}
	All these statements are contained in \cite{Venjakob}. For the
        first part of the proposition, Auslander regularity is
        \cite[Theorem 3.26]{Venjakob}. The depth of $\Lambda$ is equal
        to its global dimension by \cite[Lemma 5.5
        (iii)]{Venjakob}. The computation of the global dimension of
        $\Lambda$ follows from results of Brumer \cite[Theorem 4.1]{brumer}, Lazard \cite[Th\'{e}or\`{e}me V.2.2.8]{lazard} and Serre \cite{serrecd}.
	
	The Auslander--Buchsbaum equality is \cite[Theorem 6.2]{Venjakob}. Finally, the formula for $\pd_\Lambda(M)$ is \cite[Corollary 6.3]{Venjakob}.
\end{proof}

\begin{adefn}
  \label{defn: CM module}If $K$ is a compact $p$-adic analytic group,
  then a non-zero finitely generated~$\cO[[K]]$-module $M$ is
  \emph{Cohen--Macaulay} if $\Ext^i_{\cO[[K]]}(M,\cO[[K]])$ is non-zero for 
  just one 
  degree $i$. 
\end{adefn}

\begin{aremark}\label{rem: cm}
If~$K$ is furthermore torsion-free and pro-$p$, then by Proposition \ref{venja}, a finitely generated $\cO[[K]]$-module $M$ 
	is Cohen--Macaulay if and only if $\mathrm{depth}_{\cO[[K]]}(M) = 
	\delta_{\cO[[K]]}(M)$.
\end{aremark}

If $K$ is an arbitrary compact $p$-adic analytic group, then $\cO[[K]]$ is not necessarily local (although it is semilocal), and is not necessarily Auslander regular. But the notions of grade and projective dimension are still well-behaved, because we can apply the following lemma with $H$ a normal compact open subgroup of $K$ which is torsion free and pro-$p$.
\begin{alemma}\label{lem:samepd}
	Suppose $K$ is a compact $p$-adic analytic group and let $H \subset K$ be a 
	normal compact open subgroup. Let $M$ be a $\cO[[K]]$-module. 
	\begin{itemize}
		\item For all $i \ge 0$ we have an isomorphism of $\cO[[H]]$-modules 
		\[\Ext^i_{\cO[[H]]}(M,\cO[[H]])\cong \Ext^i_{\cO[[K]]}(M,\cO[[K]]).\] 
		In particular, we have $j_{\cO[[K]]}(M) = j_{\cO[[H]]}(M)$.
		\item Suppose $M$ is finitely generated and of finite projective dimension over $\cO[[K]]$. Suppose that $H$ is torsion free and pro-$p$. Then \[\mathrm{pd}_{\cO[[K]]}(M) = \mathrm{pd}_{\cO[[H]]}(M).\]\end{itemize}
\end{alemma}
\begin{proof}
	The first item follows from \cite[Lemma 5.4]{Ardakov-Brown-2}. The second 
	item is a combination of the first with the fact that we have 
	\[\pd_\Lambda(M) = \max\{i : \Ext^i_\Lambda(M,\Lambda) \ne 0\}\] for 
	$\Lambda = \cO[[H]]$ by Proposition \ref{venja} and we also have the 
	same equality for $\Lambda = \cO[[K]]$ by \cite[Remark 
	6.4]{Venjakob}. 
\end{proof}

From now on in this subsection we fix a compact $p$-adic analytic group $K$ and assume that $K$ is torsion free and pro-$p$. We let $\Lambda = \cO[[K]]$, and let $d = 1 + \dim(K)$, so $d$ is the global dimension of $\Lambda$. 

We use the following fundamental fact (again due to Venjakob) in this section:
\begin{alemma}\label{gradeses}
	If we have a short exact sequence of finitely generated $\Lambda$-modules $0 \rightarrow L \rightarrow M \rightarrow N \rightarrow 0$ then $j_\Lambda(M) = \min(j_\Lambda(L),j_\Lambda(N))$.
\end{alemma}
\begin{proof}
	This is \cite[Proposition 3.6]{Venjakob}.
\end{proof}

The next two lemmas are generalisations of \cite[Lemmas 6.1, 6.2]{1207.4224}:

\begin{alemma}\label{cg61}If $N$ is a finitely generated
  $\Lambda$-module with projective dimension $j$, and $0 \ne M
  \subseteq N$, then $j_\Lambda(M) \le j$. 
  
 \end{alemma}
\begin{proof}
Since $\Lambda$ is Auslander regular, this follows immediately 
from \cite[Proposition 3.10]{Venjakob}. 
\end{proof}

\begin{alemma}\label{lem:CG}
	Suppose $l_0$ is an integer with $0 \le l_0 \le d$. Let $P_{\bullet}$ be a chain 
	complex of finite free $\Lambda$-modules, concentrated in degrees 
	$0,\ldots,l_0$. Assume that $H_*(P_{\bullet}) \ne 0$. Then $j_\Lambda(H_*(P_{\bullet})) \le 
	l_0$ and if equality occurs then:
	\begin{enumerate}
		\item $P_{\bullet}$ is a projective resolution of $H_0(P_{\bullet})$.
		\item We have $\pd_\Lambda(H_0(P_{\bullet})) = j_\Lambda(H_0(P_{\bullet})) = l_0$.
	\end{enumerate}
	
	We have the same statements if we replace $\Lambda$ with $\Omega := 
	\Lambda/\varpi = k[[K]]$.
\end{alemma}
\begin{proof}
Let $m \ge 0$ be the largest integer such that $H_m(P_{\bullet}) \ne 0$. Consider the 
complex \[P_{l_0}  \rightarrow \cdots \rightarrow 
P_{m+1}\overset{d_{m+1}}{\rightarrow}P_m.\] By the definition of $m$, this 
complex is a projective resolution of $K_m:= P_m/\im(d_{m+1})$. It follows that 
$\pd_\Lambda(K_m) \le l_0-m$. 

Since $H_m(P_{\bullet})$ is a non-trivial submodule of $K_m$, by Lemmas~\ref{gradeses}
 and~\ref{cg61}  we have \[j_\Lambda(H_*(P_{\bullet})) \le
  j_\Lambda(H_m(P_{\bullet})) \le \pd_\Lambda(K_m)\le l_0 - m\le l_0,\]as claimed. 

If we have the equality $j_\Lambda(H_*(P_{\bullet})) = l_0$, then
equality holds in all the above inequalities, so that in particular $m
= 0$, $K_m=H_0(P_{\bullet})$, and the other claims follow immediately.

The proof with $\Lambda$ replaced by $\Omega$ is identical, using the fact that 
the relevant lemmas all hold with $\Lambda$ replaced by $\Omega$ (which is 
again Auslander regular). \end{proof}

We finish this subsection with a Lemma computing the codimension of a tensor 
product of two modules. 

\begin{alemma}\label{tensorgrade}
	Let $G, H$ be compact $p$-adic analytic groups. Let 
	$M, N$ be finitely generated $k[[G]]$- and $k[[H]]$-modules. Then 
	$j_{k[[G\times H]]}(M\wotimes_k N) = j_{k[[G]]}(M) + j_{k[[H]]}(N)$.
\end{alemma}
\begin{proof}
	By Lemma \ref{lem:samepd} we can assume that $G$ and $H$ are torsion free 
	pro-$p$.
	
	Set $\Omega =  k[[G\times H]]$, $\Omega_1 = k[[G]]$ and $\Omega_2 = 
	k[[H]]$. Note that we can naturally identify $\Omega$ with the completed 
	tensor product $\Omega_1\wotimes_k \Omega_2$.
	Let $P_\bullet \rightarrow M$ and $Q_\bullet \rightarrow N$ be finite free 
	resolutions of $M$ and $N$ respectively (they exist since $\Omega_1$ and 
	$\Omega_2$ have finite global dimension). 
	
	We denote by $P_\bullet \wotimes_k Q_\bullet$ the finite free complex of 
	$\Omega$ modules obtained from totalizing the double complex $(P_i 
	\wotimes_k P_j)_{i,j}$. This is a finite free resolution of 
	$M\wotimes_k N$. We have natural isomorphisms \[\Hom_\Omega(P_\bullet 
	\wotimes_k Q_\bullet,\Omega) = 
	\Hom_{\Omega_1}(P_\bullet,\Omega_1)\wotimes_k\Hom_{\Omega_2}(Q_\bullet,\Omega_2).\]
	The equality $j_{k[[G\times H]]}(M\wotimes_k N) = j_{k[[G]]}(M) + 
	j_{k[[H]]}(N)$ follows immediately. Indeed, we have a spectral sequence 
	\[\Ext^i_{\Omega_1}(M,\Omega_1)\wotimes_k\Ext^j_{\Omega_2}(N,\Omega_2) 
	\Rightarrow \Ext^{i+j}_{\Omega}(M\wotimes_k N,\Omega).\qedhere\]
\end{proof}

\subsection{Gelfand--Kirillov dimension}In this section we assume that
$K$ is a compact $p$-adic analytic group which is uniform
pro-$p$. (Note that any compact $p$-adic analytic group contains a
normal open subgroup which is uniform pro-$p$, so this will not be a
problematic assumption in our applications.) We again let $\Lambda = \cO[[K]]$, and set $d = 1 + \dim(K)$. We let $\Omega = \Lambda/\varpi\Lambda$. We denote by $J_\Omega$ the Jacobson radical of $\Omega$. The ring $\Omega$ is again Auslander regular, and for finitely generated $\Omega$ modules the dimension $\delta_\Omega$ (or equivalently the grade $j_\Omega$) can be computed as a Gelfand--Kirillov dimension:

\begin{aprop}\label{GKdim}
	Let $M$ be a finitely generated $\Omega$-module. We have \[\delta_\Omega(M) = \limsup \log_n \dim_k M/J_\Omega^n M.\]
\end{aprop}
\begin{proof}
	This is \cite[Prop.\ 5.4~(3)]{Ardakov-Brown}.
\end{proof}

\subsection{Comparing dimensions}
We again assume that $K$ is uniform pro-$p$ and let
$\Lambda = \cO[[K]]$. Fix a topological generating set
$a_1,\ldots,a_m$ for $K$. We consider two more Auslander regular rings
$A = \Lambda\hat{\otimes}_{\cO}\cO[[x_1,\ldots,x_r]]$ and
$B = \Lambda\hat{\otimes}_{\cO}\cO[[y_1,\ldots,y_s]]$ together with a
map $A \rightarrow B$ induced from a (local $\cO$-algebra) map
$\cO[[x_1,\ldots,x_r]] \rightarrow \cO[[y_1,\ldots,y_s]]$. 

Note that we can 
think of $A$ and $B$ as the Iwasawa algebras $\Lambda_\cO[[K\times\Z_p^r]]$ for 
appropriate $r$, and $K \times \Z_p^r$ is uniform pro-$p$, so we can
apply the results of the previous subsections to~$A$ and~$B$.

We set
$\overline{A} = A/\varpi A$ and $\overline{B} = B/\varpi B$. The goal
of this subsection is Lemma \ref{samedim} which shows that if $M$ is a
finitely generated $B$-module, which is also finitely generated as an
$A$-module, then $\delta_A(M) = \delta_B(M)$. This generalises a
well-known fact in commutative algebra \cite[Ch.~0, Prop.~16.1.9]{EGAIV1}.

\begin{alemma}\label{changeofrings}
Suppose $M$ is a finitely generated $A$-module, and let $x$ be one of 
	$\varpi,x_1,\ldots,x_r$. Then \begin{itemize}
		\item if $M$ is killed by $x$, $\delta_{A}(M) = 
		\delta_{A/x}(M)$.
		\item if $M$ is $x$-torsion free, $\delta_{A}(M) = 1 + 
		\delta_{A/x}(M/x M)$.
	\end{itemize}
\end{alemma}
\begin{proof}
	First we assume that $M$ is killed by $x$. The base
	change spectral sequence~\cite[Ex.\ 5.6.3]{weibel} for
	$\mathrm{Ext}$ is \[E_2^{i,j}:
	\mathrm{Ext}^i_{A/x}(M,\mathrm{Ext}^j_A(A/x,A))\implies
	\mathrm{Ext}_A^{i+j}(M,A)\] and
	$\mathrm{Ext}^j_A(A/x,A)$ is zero unless $j=1$,
	when we have $\mathrm{Ext}^1_A(A/x,A) =
	A/x$. Since $M$ is killed by $x$, 
	$j_A(M)>0$, and we have \[\mathrm{Ext}^i_{A/x}(M,A/x) =
	\mathrm{Ext}^{i+1}_{A}(M,A)\] for $i \ge 0$. We
	deduce that $j_A(M) = 1+j_{A/x}(M)$, and therefore $\delta_A(M) = 
	\delta_{A/x}(M)$.
	
	Now we assume that $M$ is $x$-torsion free. \cite[Thm.\
	4.3]{Levasseur}  implies that $j_{A}(M/x M) \ge 1 +
	j_{A}(M)$, so we have an exact sequence \[0 \rightarrow
	\Ext_A^{j_A(M)}(M,A)\overset{\times x}{\rightarrow}
	\Ext_A^{j_A(M)}(M,A) \rightarrow \Ext_A^{1 +
		j_{A}(M)}(M/x M,A)\] and
	$\Ext_A^{j_A(M)}(M,A)$ is a non-zero finitely
	generated $A$-module. By Nakayama's lemma we see that
	$\Ext_A^{j_A(M)}(M,A)/x
	\Ext_A^{j_A(M)}(M,A)$ is non-zero, and so
	$\Ext_A^{1 + j_{A}(M)}(M/x M)$ is also non-zero. This 
	implies that $j_{A}(M/x M) = 1 + j_{A}(M)$. The first part 
	of the lemma then gives $j_{A/x}(M/x M) = j_{A}(M)$ and so 
	$\delta_{A}(M) = 1 + \delta_{A/x}(M/x M)$.
\end{proof}

\begin{alemma}\label{lem:obviouscodimineq}
	Suppose $M$ is a finitely generated $A$-module and let $x$ be one of 
	$\varpi, x_1,\ldots,x_r$. Then \[j_A(M) \ge j_{A/x}(M/xM).\] In particular, 
	\[j_A(M) \ge j_{\Lambda}(M/(x_1,\ldots,x_r)M)\] and \[j_A(M) \ge 
	j_{\Omega}(M/(\varpi,x_1,\ldots,x_r)M).\]
\end{alemma}
\begin{proof}
	The `in particular' part of the lemma follows from the first part by 
	induction.

Applying Lemma \ref{changeofrings}, we see that if $M$
is $x$-torsionfree, then $j_A(M) = j_{A/x}(M/xM)$. In general, we
have an exact sequence \[0\rightarrow M[x^\infty] \rightarrow M
  \rightarrow M/M[x^\infty]\rightarrow 0\] where $M/M[x^\infty]$ is
$x$-torsionfree, so we have a short exact
sequence \[0\rightarrow A/x\otimes_A M[x^\infty] \rightarrow
  A/x\otimes_AM \rightarrow A/x\otimes_A(M/M[x^\infty])\rightarrow
  0.\] By Lemma~\ref{gradeses}, it now suffices to show that if $M$ is killed by $x^N$ for some
$N \ge 1$, then $j_A(M) \ge j_{A/x}(M/xM)$. Consider the filtration $
\{0\} = x^NM \subset x^{N-1}M \subset \cdots \subset xM \subset M$. We
have $j_A(M) = \min_i(j_A(x^iM/x^{i+1}M))$ by a repeated application of 
Lemma~\ref{gradeses} and  we therefore have $j_A(M)= 1 +
\min_i(j_{A/x}(x^iM/x^{i+1}M))$ by another application of Lemma
\ref{changeofrings}. Multiplication by $x^i$ gives a surjective
$A$-linear map $M/xM \rightarrow x^iM/x^{i+1}M$, so $j_{A/x}(M/xM) \le
j_{A/x}(x^iM/x^{i+1}M)$ for all $i$ (by Lemma~\ref{gradeses} again). In particular we 
have $j_A(M) = 1
+ j_{A/x}(M/xM)$ which gives the desired conclusion.\end{proof}

\begin{alemma}\label{commuting}We have $J_{\overline{A}}\overline{B} = \overline{B}J_{\overline{A}}$ and $J_{\overline{A}}J_{\overline{B}} = J_{\overline{B}}J_{\overline{A}}$.
\end{alemma}
\begin{proof}
	$J_{\overline{A}}$ is the (right, left, two-sided) ideal of $\overline{A}$ generated by $a_1-1,\ldots,a_m-1,x_1,\ldots,x_r$ and $J_{\overline{B}}$ is the (right, left, two-sided) ideal of $\overline{B}$ generated by $a_1-1,\ldots,a_m-1,y_1,\ldots,y_s$. The lemma is now easy, since the $x_i$ map to central elements in $\overline{B}$.
\end{proof}
The next lemma is a mild variation on \cite[Lemma 3.1]{Wadsley}.
\begin{alemma}\label{GKcompare}
	Suppose $M$ is a finitely generated $\overline{B}$-module, which is also 
	finitely generated as an $\overline{A}$-module. Then 
	$\delta_{\overline{A}}(M) = \delta_{\overline{B}}(M)$.
\end{alemma}
\begin{proof}
	We show the lemma by comparing Gelfand--Kirillov dimensions. Since $M$ is a finitely generated $\overline{A}$-module, $M/J_{\overline{A}}M$ is a finite dimensional $k$-vector space. By Lemma \ref{commuting}, $J_{\overline{A}}M$ is a $\overline{B}$-submodule of $M$. So $M/J_{\overline{A}}M$ is an Artinian $\overline{B}$-module. Therefore $J_{\overline{B}}^k (M/J_{\overline{A}}M) = 0$ for some positive integer $k$. So $J_{\overline{B}}^kM\subset J_{\overline{A}}M \subset J_{\overline{B}}M$. Using the fact that $J_{\overline{A}}J_{\overline{B}} = J_{\overline{B}}J_{\overline{A}}$ (Lemma \ref{commuting}) an induction shows that \[J_{\overline{B}}^{kN}M\subset J_{\overline{A}}^NM \subset J_{\overline{B}}^NM\] for all $N \ge 1$. Using Proposition \ref{GKdim} we conclude that $\delta_{\overline{A}}(M) = \delta_{\overline{B}}(M)$. 
\end{proof}

\begin{alemma}\label{samedim}
	Suppose $M$ is a finitely generated $B$-module, which is also finitely generated as an $A$-module. Then $\delta_{{A}}(M) = \delta_{{B}}(M)$.
\end{alemma}
\begin{proof}
	$M$ has a finite filtration by $B$-submodules $\{0\} = M_0
        \subset M_1 \subset \cdots \subset M_l = M$ such that each
        $M_i/M_{i-1}$ is either $\varpi$-torsionfree or killed by
        $\varpi$. Each $M_i$ is also a finitely generated
        $A$-module. By Lemma~\ref{gradeses}, we have $\delta_A(M) =
        \max_i(\delta_A(M_i/M_{i-1}))$ and $\delta_B(M) =
        \max_i(\delta_B(M_i/M_{i-1}))$, so we may assume that $M$ is either $\varpi$-torsionfree or killed 
        by $\varpi$. Applying Lemma \ref{changeofrings} and Lemma 
        \ref{GKcompare} gives $\delta_A(M) = \delta_B(M)$.
\end{proof}

\subsection{Comparing depths}
We retain the assumptions and notation of the previous
subsection. Recall that we have two $\Lambda$-algebras $A =
\Lambda\hat{\otimes}_{\cO}\cO[[x_1,\ldots,x_r]]$ and $B =
\Lambda\hat{\otimes}_{\cO}\cO[[y_1,\ldots,y_s]]$. The goal of this
subsection is Lemma \ref{depthineq}, which shows that if $M$ is a
finitely generated $B$-module, which is also finitely generated as an
$A$-module, then $\mathrm{depth}_{A}(M) \le \mathrm{depth}_{B}(M)$. In
fact, we can show that $\mathrm{depth}_{A}(M) = \mathrm{depth}_{B}(M)$
(which again generalises a well-known result in commutative algebra 
\cite[Ch.~0, Prop.~16.4.8]{EGAIV1})
but proving the inequality suffices for our applications and is
already sufficiently painful.

We set $\overline{R} = k[[x_1,\ldots,x_r]] = A/J_\Lambda A$ and $\overline{S} = k[[y_1,\ldots,y_s]] = B/J_\Lambda B$. We have a map of local $k$-algebras $\overline{R} \rightarrow \overline{S}$. 

\begin{alemma}\label{presinj}
	Suppose $I$ is an injective left $B$-module. Then $I$ is injective as a left $\Lambda$-module.
\end{alemma}
\begin{proof}
	Suppose $0 \rightarrow L \rightarrow M \rightarrow N \rightarrow 0$ is a short exact sequence of left $\Lambda$-modules. Since $B$ is a flat right $\Lambda$-module, we have an exact sequence of left $B$-modules \[0 \rightarrow B\otimes_\Lambda L \rightarrow B\otimes_\Lambda M \rightarrow B\otimes_\Lambda N \rightarrow 0\] and hence an exact sequence \[0 \rightarrow \Hom_B(B\otimes_\Lambda N,I) \rightarrow \Hom_B(B\otimes_\Lambda M,I) \rightarrow \Hom_B(B\otimes_\Lambda L,I) \rightarrow 0.\] Finally, the tensor-hom adjunction implies that \[0 \rightarrow \Hom_\Lambda(N,I) \rightarrow \Hom_\Lambda(M,I) \rightarrow \Hom_\Lambda(L,I) \rightarrow 0\] is exact.
\end{proof}

For any left $B$-module $M$, note that $\Hom_\Lambda(k,M)=\{m \in M: J_\Lambda 
m = 0\}$ is naturally a left $\overline{S}$-module. We denote by 
$\RHom^{\overline{S}}_\Lambda(k,M)$ the object of ${D}^+(\overline{S})$ given 
by taking an injective $A$-module resolution of $M$ and applying 
$\Hom_\Lambda(k,-)$ to get a complex of $\overline{S}$-modules. By Lemma 
\ref{presinj}, we have natural isomorphisms of Abelian groups 
$H^i(\RHom^{\overline{S}}_\Lambda(k,M)) = \mathrm{Ext}^i_\Lambda(k,M)$.

\begin{aremark}
	Note that the natural $\overline{S}$-module structure on $\mathrm{Ext}^i_\Lambda(k,M)$ can also be defined using the facts that $\mathrm{Ext}^i_\Lambda(k,M) = \mathrm{Ext}^i_B(B\otimes_\Lambda k,M)$ (extension of scalars) and that $B\otimes_\Lambda k$ is a $(B,\overline{S})$-bimodule.
\end{aremark}
\begin{aremark}
	For an $A$-module $M$, we can similarly define $\RHom^{\overline{R}}_\Lambda(k,M)$. 
\end{aremark}
\begin{alemma}\label{forgetful}
	For a $B$-module $M$, there is a natural isomorphism \[\RHom^{\overline{R}}_\Lambda(k,M) = \iota^{\overline{S}}_{\overline{R}}\RHom^{\overline{S}}_\Lambda(k,M),\] where $\iota^{\overline{S}}_{\overline{R}}$ is the derived functor of the \emph{(}exact\emph{)} forgetful functor from $\overline{S}$-modules to $\overline{R}$-modules. 
\end{alemma}
\begin{proof}
	We can compute $\RHom^{\overline{R}}_\Lambda(k,M)$ using an injective 
	$B$-module resolution of $M$, since an injective $B$-module is acyclic for 
	the functor $\Hom_\Lambda(k,-)$ from $A$-modules to $\overline{R}$-modules. 
	Computing $\RHom^{\overline{S}}_\Lambda(k,M)$ using the same injective 
	resolution gives the desired isomorphism.
\end{proof}

\begin{alemma}\label{depthreducecomm}
	For $B$-modules $M$, we have natural isomorphisms \[\RHom_B(k,M) = \RHom_{\overline{S}}(k,\RHom^{\overline{S}}_\Lambda(k,M))\] and \[\RHom_A(k,M) = \RHom_{\overline{R}}(k,\iota^{\overline{S}}_{\overline{R}}\RHom^{\overline{S}}_\Lambda(k,M)).\]
\end{alemma}
\begin{proof}
	Consider the functor
        $\Hom_\Lambda(k,-)$ from $B$-modules to $\overline{S}$-modules. This 
          takes injectives to injectives, since for an $\overline{S}$-module 
          $X$ we have $\Hom_{\overline{S}}(X,\Hom_\Lambda(k,M)) 
          = \Hom_B(X,M)$. 
	
	The functor $\Hom_{\overline{S}}(k,\Hom_\Lambda(k,-))$ from $B$-modules to 
	Abelian groups is naturally equivalent to the functor $\Hom_B(k,M)$. The 
	derived functor of the composition of functors is given by 
	$\RHom_{\overline{S}}(k,\RHom^{\overline{S}}_\Lambda(k,-))$, and this gives 
	the first collection of natural isomorphisms.
	
	Applying the same argument to $A$-modules, together with Lemma \ref{forgetful}, we get the second collection of natural isomorphisms.
\end{proof}

At this point we recall that for a commutative Noetherian local ring
$X$ with maximal ideal $\m_X$ there is a good notion of depth for
objects in $D^+(X)$
\cite{Iyengar-depth}\footnote{In fact one needn't restrict to bounded
  complexes, see \cite{depth-unbounded}}. 

\begin{adefn}
	For $M \in {D}^+(X)$ we define \[\mathrm{depth}_X(M) = 
	\mathrm{inf}\{i:\mathrm{Ext}^i_X(X/\m_X,M) \ne 0 \}.\]
\end{adefn}

\begin{alemma}\label{commdepthineq}
	Let $M \in  {D}^+(\overline{S})$. We have 
	\[\mathrm{depth}_{\overline{R}}(\iota^{\overline{S}}_{\overline{R}}M) \le 
	\mathrm{depth}_{\overline{S}}(M).\]
\end{alemma}
\begin{proof}Combine \cite[Thm.\ 6.1]{Iyengar-depth} 
  (which shows that our definition of depth coincides with the
  definition given in~\cite[\S2]{Iyengar-depth}) with
  \cite[Prop.\ 5.2~(2)]{Iyengar-depth}.
\end{proof}
\begin{alemma}\label{depthineq}
	Let $M$ be a $B$-module. We have \[\mathrm{depth}_{A}(M) \le \mathrm{depth}_{B}(M).\]
\end{alemma}
\begin{proof}
	By Lemma \ref{depthreducecomm} and Lemma \ref{commdepthineq} we have \begin{align*}\mathrm{depth}_{A}(M) &= \mathrm{depth}_{\overline{R}}(\iota^{\overline{S}}_{\overline{R}}\RHom^{\overline{S}}_\Lambda(k,M))\\& \le \mathrm{depth}_{\overline{S}}(\RHom^{\overline{S}}_\Lambda(k,M)) = \mathrm{depth}_{B}(M).\qedhere\end{align*}
\end{proof}
\begin{acor}\label{cor:stayCM}
	Suppose $M$ is a finitely generated $B$-module, which is also finitely generated as an $A$-module. Moreover, suppose that $M$ is a Cohen--Macaulay $A$-module. Then $M$ is a Cohen--Macaulay $B$-module, with $\mathrm{depth}_{B}(M) = \delta_B(M) = \delta_A(M)$.
\end{acor}
\begin{proof}
	By Lemma \ref{depthineq} we have $\delta_A(M) =
        \mathrm{depth}_{A}(M) \le \mathrm{depth}_{B}(M)$. We also have
        $\mathrm{depth}_B(M) \le \delta_B(M)$, by parts ~(2) and~(3) of
        Proposition~\ref{venja} (or by local duality). Since $\delta_A(M) = \delta_B(M)$ (by Lemma \ref{samedim}), all these inequalities are equalities.
\end{proof}
\begin{aprop}[Miracle Flatness]\label{miracleflatness}
	Let $M$ be a finitely generated Cohen--Macaulay $A$-module. 
	
	Then $M$ is a flat $\cO[[x_1,\ldots,x_r]]$-module if and only if \[j_A(M) = 
	j_{\Omega}(M/(\varpi,x_1,\ldots,x_r)M).\]
\end{aprop}
\begin{proof}
	We let $R = \cO[[x_1,\ldots,x_r]]$ and $\m_R = (\varpi,x_1,\ldots,x_r) 
	\subset R$. First suppose $M$ is a flat 
	$\cO[[x_1,\ldots,x_r]]$-module. Then 
	$(\varpi,x_1,\ldots,x_r)$ is an $M$-regular sequence (using Nakayama's 
	lemma for finitely generated $A$-modules to see that 
	$M/(\varpi,x_1,\ldots,x_r) \ne 0$; we are assuming $M \ne 0$ since 
	Cohen--Macaulay modules are by definition non-zero). It follows from Lemma 
	\ref{changeofrings} that we have the desired equality of
        codimensions.

	Conversely, suppose that $j_A(M) = 
	j_{\Omega}(M/(\varpi,x_1,\ldots,x_r)M)$. We claim that 
	$(\varpi,x_1,\ldots,x_r)$ is an $M$-regular sequence. To prove the claim, 
	it suffices (by induction on~$r$) to show that for $x \in \{\varpi,x_1,\ldots,x_r\}$ we have the 
	following \begin{enumerate}
		\item $j_A(M) = j_{A/x}(M/xM).$
		\item $x$ is $M$-regular.
		\item $M/xM$ is a Cohen--Macaulay $A/x$-module.
	\end{enumerate} By Lemma \ref{lem:obviouscodimineq}, 
	we have $j_A(M)\ge j_{A/x}(M/xM)\ge j_{\Omega}(M/\m_RM)$, so our assumption 
	implies that (1) holds. 
	
	Next we check that $x$ is $M$-regular. As in the 
	proof of Lemma \ref{lem:obviouscodimineq}, we have a short exact sequence 
	\[0 \rightarrow M[x^\infty]\rightarrow M \rightarrow 
	M/M[x^\infty]\rightarrow 0\] where $M/M[x^\infty]$ is $x$-torsion free. 
	Suppose for a contradiction that $M[x^\infty]$ is nonzero. By 
	\cite[Prop.~3.9, Prop.~3.5(v)]{Venjakob} $M$ has pure $\delta$-dimension 
	$\dim_A(M)$. By \cite[Prop.~3.5(vi)(b)]{Venjakob} we therefore have 
	$j_A(M[x^\infty]) = j_A(M)$ (if a module has pure $\delta$-dimension, all 
	its non-zero submodules have the same dimension). As in the proof of Lemma \ref{lem:obviouscodimineq} we also have 
	$j_A(M[x^\infty])  = 1 + j_{A/x}(M[x^\infty]/xM[x^\infty])$. Combining the 
	two equalities, we get $j_{A/x}(M[x^\infty]/xM[x^\infty]) = j_A(M) - 1$, 
	which (by Lemma~\ref{gradeses}) contradicts (1), since $M[x^\infty]/xM[x^\infty]$ is a submodule of 
	$M/xM$. This completes the proof that (2) holds.
	
	Now we must show that $M/xM$ is a Cohen--Macaulay $A/x$-module. By 
	Lemma \ref{changeofrings} we have $j_A(M/xM) = 1 + j_{A/x}(M/xM) = 1 + 
	j_A(M)$. By (2), we have a short exact sequence \[0 \rightarrow M 
	\overset{\times 
	x}{\rightarrow} M \rightarrow M/xM \rightarrow 0.\]  Considering the long 
	exact sequence for $\Hom_A(-,A)$ we see that $\Ext_A^i(M/xM,A) = 0$ for all 
	$i \ne 1 + j_A(M)$. The argument of the first paragraph of the
        proof of Lemma \ref{changeofrings} now implies 
	that  $\Ext_{A/x}^i(M/xM,A/x) = 0$ for all 
	$i \ne j_A(M)$, and this shows that $M/xM$ is
        Cohen--Macaulay (by Remark \ref{rem: cm}).
	
	Finally, we have established the claim that $(\varpi,x_1,\ldots,x_r)$ is an 
	$M$-regular sequence. It follows that $\mathrm{Tor}_1^R(R/\m_R,M) = 0$. If 
	$I$ is an ideal in $R$ then $I \otimes_R M$ is naturally a finitely 
	generated $A$-module and is therefore separated for the $\m_R$-adic 
	topology. Now \cite[Theorem 22.3]{MR1011461} implies that $M$ is a flat 
	$R$-module (the 
	previous sentence shows that $M$ is $\m_R$-adically ideal-separated, in 
	Matsumura's terminology).	
	\end{proof}
\subsection{An application of the Artin--Rees lemma}We now recall a version of the Artin--Rees Lemma.\begin{alem}\label{lem:artinrees}
	Let $K$ be a compact $p$-adic analytic group, and let $M$ be a 
	$\cO[[K]]$-submodule of $\cO[[K]]^{\oplus t}$, for 
	some $t \ge 1$. Let $K'$ be an open uniform pro-$p$ subgroup of $K$, and 
	let $\cJ$ denote the two-sided ideal of $\cO[[K]]$ generated by the 
	maximal ideal $\m$ of the local ring $\cO[[K']]$. Then there is a 
	constant $c \ge 0$ such that $M \cap (\cJ^{m+c})^{\oplus t} \subset \cJ^mM$ 
	for all $m \ge 0$.
\end{alem}
\begin{proof}
	The associated graded of $\cO[[K]]$ 
	for the $\cJ$-adic filtration is finite over the Noetherian ring 
	$\mathrm{gr}_{\m}\cO[[K']]$, so it is itself 
	Noetherian. Now we can apply \cite[Prop.~II.2.2.1, 
	Thm.~II.2.1.2(2)]{zariskian}. This shows that the $\cJ$-adic filtration on 
	$\cO[[K]]$ has the Artin--Rees property (defined in 
	\cite[Defn.~II.1.1.1]{zariskian}), and the statement of the Lemma is a 
	special case of this property.
\end{proof}
\begin{alem}\label{lem:flatnessinlimit}
	Keep the same notation as in the previous Lemma. Suppose we have flat 
	$\cO[[K]]/\cJ^m$-modules $M_m$ for each $n\ge 1$, with $M_m = 
	M_{m+1}/\cJ^mM_{m+1}$. Then $M:=\invlim_m M_m$ is a flat 
	$\cO[[K]]$-module and \[Q \otimes_{\cO[[K]]} M = \invlim Q 
	\otimes_{\cO[[K]]} M_m\] for 
	every finitely generated (right) $\cO[[K]]$-module $Q$. 
	
	In particular, we have $M/\cJ^mM = M_m$.
\end{alem}
\begin{proof}
	This follows from \cite[\href{http://stacks.math.columbia.edu/tag/0912}{Tag 
	0912}]{stacks-project}. The reference assumes 
	that the rings in question are commutative, so we will write out the proof 
	in our setting. Set $A = \cO[[K]]$ to abbreviate our notation.
	
	We first show that $Q \otimes_A M = \invlim Q \otimes_A M_m$ for every 
	finitely generated (right)
	$A$-module $Q$. Since~$A$ is Noetherian, we may choose a 
	resolution $F_2 \to F_1 \to F_0 \to Q \to 0$
	by finite free $A$-modules $F_i$.  Then
	$$
	F_2 \otimes_A M_m \to F_1 \otimes_A M_m \to F_0 \otimes_A M_m
	$$
	is a chain complex whose homology in degree $0$ is $Q \otimes_A M_m$
	and whose homology in degree $1$ is
	$$
	\text{Tor}_1^A(Q, M_m) = \text{Tor}_1^A(Q, A/\cJ^m) \otimes_{A/\cJ^m} M_m
	$$
	as $M_m$ is flat over $A/\cJ^m$. Set $K= \ker(F_0\to Q)$. We have \[\text{Tor}_1^A(Q, A/\cJ^m) = 
	(K\cap(\cJ^m F_0))/\cJ^mK\] so Lemma \ref{lem:artinrees} implies that there 
	exists a $c \ge 0$ such that the map \[\text{Tor}_1^A(Q, A/\cJ^{n+c})\to 
	\text{Tor}_1^A(Q, A/\cJ^m)\] is zero for all $m$.
	
	It follows from \cite[\href{http://stacks.math.columbia.edu/tag/070E}{Tag 
	070E}]{stacks-project}
	that $\invlim Q \otimes_A M_m =
	\coker(\invlim F_1 \otimes_A M_m \to \invlim F_0 \otimes_A M_m)$.
	Since the $F_i$ are finite free this equals
	$\coker(F_1 \otimes_A M \to F_0 \otimes_A M) = Q \otimes_A M$,
        as claimed.  Taking $Q = A/\cJ^m$, we obtain $M/\cJ^mM = M_m$.
	
	It remains to show that $M$ is flat. Let $Q \to Q'$ be an injective map of finitely generated right 
	$A$-modules; we must show that $Q \otimes_A M \to Q' \otimes_A 
	M$ is injective. By the above we see
	$$
	\ker(Q \otimes_A M \to Q' \otimes_A M) =
	\ker(\invlim Q \otimes_A M_m \to \invlim Q' \otimes_A M_m).
	$$
	For each $m$ we have an exact sequence
	$$
	\text{Tor}_1^A(Q', M_m) \to \text{Tor}_1^A(Q'', M_m) \to
	Q \otimes_A M_m \to Q' \otimes_A M_m
	$$
	where $Q'' = \coker(Q \to Q')$. Above we have seen that the
	inverse systems of Tor's are essentially constant with value $0$.
	It follows from
	\cite[\href{http://stacks.math.columbia.edu/tag/070E}{Tag 
	070E}]{stacks-project}
	that the inverse limit of the right most maps is injective, as
        required.\end{proof}

\section{Tensor products and projective covers}\label{appendix: tensor products projective covers}
\subsection{Tensor products}We recall from \cite[\S2]{brumer} that if $R$ is 
a pseudocompact ring and $M, N$ are pseudocompact (right, resp.~left) 
$R$-modules, then the 
completed tensor product $M \wotimes_R N$ is a pseudocompact $R$-module, which 
satisfies the usual universal property for the tensor product in the category 
of pseudocompact $R$-modules. $M\wotimes_R N$ is the completion of $M\otimes_R 
N$ in the topology induced
by taking $\mathrm{Im}(M\otimes_R V + U \otimes_R N)$ as a fundamental system 
of 
open neighborhoods of $0$, where $U$ (resp.~$V$) runs through the open 
submodules of $M$ (resp.~$N$).

If $A$ and $B$ are pseudocompact $R$-algebras, and $M, N$ (respectively) 
are pseudocompact $A$ and $B$-modules, then $M\wotimes_R N$ is naturally a 
pseudocompact $A\wotimes_R B$-module.

\begin{alemma}\label{brumercont} Let $M$, $N$ be pseudocompact $\cO$-modules.
	Suppose $M = \invlim_i M_i$ and $N = \invlim_j N_j$, where $M_i$ and 
	$N_j$ are also pseudocompact $\cO$-modules. Suppose that the 
	transition maps $M_j \rightarrow M_i$ and $N_j \rightarrow M_i$ are 
	surjective. Then the natural map 
	\[\dirlim_{i,j}\Hom_\cO^{cts}(M_i,N_j^\vee) \rightarrow 
	\Hom_\cO^{cts}(M,N^\vee)\]
	is an isomorphism.
	
	The natural map \[M\wotimes_\cO N \rightarrow \invlim_{i,j}M_i\wotimes_\cO 
	N_j \] is also an isomorphism.
\end{alemma}
\begin{proof}
	The first claim is (a special case of) \cite[Lem.~A.3]{brumer}. The second 
	claim is a special case of 
	\cite[Lem.~A.4]{brumer}.\end{proof}

\begin{alemma}\label{lem: tensor dual Brumer}Let $M, N$ be pseudocompact 
$\cO$-modules.
	There is a natural isomorphism \[\left(M \wotimes_\cO N\right)^\vee \cong 
	\Hom_\cO^{cts}(M,N^\vee)\] where $N^\vee$ has the discrete topology.
\end{alemma}
\begin{proof}
	By Lemma \ref{brumercont} we may assume that $M$ and $N$ are finite length 
	$\cO$-modules. By the universal property of the tensor product, we have 
	\[(M\otimes_\cO N)^\vee = \Hom_\cO(M,N^\vee).\qedhere\]
\end{proof}

We now recall some terminology about categories of smooth representations of 
$p$-adic analytic groups from \cite{emordone}. Let $G$ be a $p$-adic analytic 
group, with a compact open subgroup $K_0$ (all the notions recalled below will 
be independent of the choice of $K_0$). We let $A$ denote
a complete Noetherian local $\cO$-algebra with finite residue field and maximal 
ideal $\m_A$. In particular, $A$ is a pseudocompact $\cO$-algebra. 
$\Mod_G^{sm}(A)$ denotes the abelian category of smooth 
$G$-representations with coefficients in $A$ \cite[Defn.~2.2.5]{emordone}. 
Pontryagin duality gives an anti-equivalence of categories between 
$\Mod_G^{sm}(A)$ and the category of pseudocompact $A[[K_0]]$-modules with a 
compatible $G$-action \cite[(2.2.8)]{emordone}. Here we write $A[[K_0]]$ for 
$A\wotimes_\cO \cO[[K_0]]$.

An object $V \in 
\Mod_G^{sm}(A)$ is admissible if $V^\vee$ is a finitely generated 
$A[[K_0]]$-module (we take this as the definition, but see 
\cite[Lem.~2.2.11]{emordone}). An element $v \in V$ is called locally 
admissible if the $G$-subrepresentation of $V$ generated by $v$ is admissible, 
and $V$ is called locally admissible if every element of $V$ is locally 
admissible. 

Similarly, an element $v\in V$ is called locally finite if the 
$G$-subrepresentation of $V$ generated by $v$ is a finite length object in 
$\Mod_G^{sm}(A)$, and $V$ is called locally finite if every element of $V$ is 
locally 
finite.

\begin{alemma}\label{tensorisadm}
	Let $G, H$ be $p$-adic analytic groups and suppose that $V \in 
	\Mod_G^{sm}(\cO)$ and $W \in \Mod_H^{sm}(\cO)$. Suppose that $V$ and $W$ 
	are locally admissible. Then $(V^\vee \wotimes_\cO W^\vee)^\vee = 
	\Hom_\cO^{cts}(V^\vee,W)$ is a locally admissible object of $\Mod_{G\times 
	H}^{sm}(\cO)$. \end{alemma}
\begin{proof}
Let $M = V^\vee$ and $N = W^\vee$. Since $V$ and $W$ are locally admissible, we 
can write $M = 
\invlim_i 
M_i$ and $N = \invlim_j N_j$ where the $M_i^\vee$ and $N_j^\vee$ are admissible 
and the transition maps in the inverse systems are surjective. It 
follows from Lemma \ref{brumercont} that it suffices to prove the Lemma 
under the additional assumption that $V$ and $W$ are admissible.

Let $K_1$ and $K_2$ be compact open subgroups of $G$ and $H$ respectively. We 
may assume that $M$ and $N$ are finitely generated $\cO[[K_1]]$- and 
$\cO[[K_2]]$-modules respectively. In particular, we have (continuous) 
surjections $\cO[[K_1]]^{\oplus a} \rightarrow M$ and $\cO[[K_2]]^{\oplus b} 
\rightarrow N$. Therefore, we have a surjective map of 
$\cO[[K_1]]\wotimes_\cO\cO[[K_2]] = \cO[[K_1\times K_2]]$-modules:

\[\cO[[K_1]]^{\oplus a}\wotimes_\cO \cO[[K_2]]^{\oplus b} = \cO[[K_1\times 
K_2]]^{\oplus ab} \rightarrow M\wotimes_\cO N.\] In particular, $(M\wotimes_\cO 
N)^\vee$ is admissible. \end{proof}

We recall that an irreducible admissible object $V$ of $\Mod_G^{sm}(k)$ is 
called 
\emph{absolutely irreducible} if $V\otimes_k{k'}$ is irreducible in 
$\Mod_G^{sm}(k')$ for every field extension $k'/k$ (or equivalently for every 
finite extension). See \cite[\S4.1]{emordtwo} for this definition and the 
following facts. If $V$ is an 
admissible irreducible in $\Mod_G^{sm}(k)$ then $k' = \End_G(V)$ is a finite 
extension of $k$ and $V\otimes_k k'$ is a finite direct sum of admissible 
absolutely 
irreducible 
objects of $\Mod_G^{sm}(k')$.

\begin{alemma}\label{tensorissimple}
	Let $G, H$ be $p$-adic analytic groups and suppose that $V \in 
	\Mod_G^{sm}(\cO)$ and $W \in \Mod_H^{sm}(\cO)$. Suppose that $V$ and $W$ 
	are locally finite and locally 
	admissible. Then $(V^\vee \wotimes_\cO W^\vee)^\vee = 
	\Hom_\cO^{cts}(V^\vee,W)$ is a locally finite object of $\Mod_{G\times 
		H}^{sm}(\cO)$. 
	
	If $V$ and $W$ are abmissible absolutely irreducible then 
	$(V^\vee \wotimes_\cO W^\vee)^\vee = 
	V\otimes_k W$ is an admissible absolutely irreducible representation of $G 
	\times H$. 
\end{alemma}
\begin{proof}
	Let $M = V^\vee$ and $N = W^\vee$. Since $V$ and $W$ are locally finite, we 
	can write $M = 
	\invlim_i 
	M_i$ and $N = \invlim_j N_j$ where the $M_i^\vee$ and $N_j^\vee$ are finite 
	length and the transition maps in the inverse systems are surjective. It 
	follows from Lemma \ref{brumercont} that it suffices to prove the Lemma 
	under the additional assumption that $V$ and $W$ are finite 
	length. By induction on the length, we can assume that $V$ and $W$ are 
	irreducible admissible. In this case (since $V$ and 
	$W$ are killed by $\varpi$), $\Hom_\cO^{cts}(M,N^\vee) = 
	\dirlim_U\Hom_k(M/U,N^\vee) = V \otimes_k W$, where $U$ runs over 
	open submodules of $M$, and the first equality follows from Lemma 
	\ref{brumercont}. 
	
	Now it remains to show that if $V$ and $W$ are irreducible admissible then 
	$V \otimes_k W$ has finite length, and if moreover $V$ and $W$ are 
	absolutely irreducible then $V \otimes_k W$ is absolutely irreducible. By 
	extending scalars to a finite extension of $k$ over which both $V$ and $W$ 
	are direct sums of absolutely irreducible representations, we can 
	reduce to the case where $V$ and $W$ are absolutely irreducible  
	(descending back, 
	we see that $V \otimes_k W$ is a finite direct sum of irreducibles which 
	can be obtained by Galois descent from a direct sum of absolutely 
	irreducible represnetations in the extension of scalars).

	We have 
	\[\Hom_G(V,V\otimes_k 
	W) = \Hom_G(V,V)\otimes_kW\] since $V$ has finite length. By Schur's 
	lemma we can identify 
	$\Hom_G(V,V\otimes_k W)$ with $W$.  
	
	Suppose $U \subset V \otimes_k W$ is a nonzero $G\times 
	H$-subrepresentation. Then 
	$\Hom_G(V,U)$ is an $H$-subrepresentation of $\Hom_G(V,V\otimes_k W) 
	= W$. Since $V\otimes_k W$ is locally finite as a $G$-representation, with 
	every 
	simple submodule isomorphic to $V$, we have $\Hom_G(V,U) \ne 0$ and 
	therefore $\Hom_G(V,U) = W$. This says that for all $w \in W$, the map $v 
	\mapsto v\otimes w$ lies in $\Hom_G(V,U)$. In other words, $v\otimes w \in 
	U$ for all $v \in V, w\in W$. So $U = V\otimes W$. The same argument 
	applies after any extension of scalars $k'/k$, so we deduce that 
	$V\otimes_k W$ is absolutely irreducible.
\end{proof}

\begin{alemma}\label{simpleprodweak}
	Let $G, H$ be $p$-adic analytic groups. Suppose that both $G$ and $H$ have the 
	property that locally admissible representations are locally finite. 
	Let $X$ be an admissible absolutely irreducible 
	object of $\Mod_{G\times 
		H}^{sm}(k)$. Then there is a finite extension $k'/k$ such that the 
		extension of scalars $X_{k'} \in \Mod_{G\times 
		H}^{sm}(k')$
	is isomorphic to $V\otimes_{k'} W$, for some admissible absolutely 
	irreducible representations
	$V 
	\in\Mod_{G}^{sm}(k')$ 
	and $W \in \Mod_{H}^{sm}(k')$.
\end{alemma}
\begin{proof}
	Since $X$ is admissible as a $G \times H$-representation, it is locally 
	admissible as a $G$-representation. Indeed for every $x \in X$ there is a 
	compact 
	open subgroup $K_2 \subset H$ such that $x \in X^{K_2}$, and $X^{K_2}$ is a 
	locally admissible $G$-representation. It follows from 
	our assumptions that $X$ is a locally finite $G$-representation.
	
	So, there is a simple admissible $V 
	\in\Mod_{G}^{sm}(\cO)$ with $\Hom_G(V,X) \ne 0$. The $H$-representation 
	$\Hom_G(V,X)$ is admissible, and hence locally finite. Indeed, if $K_2 
	\subset H$ is compact open, then $X^{K_2}$ is an admissible 
	$G$-representation and $\Hom_G(V,X)^{K_2} = \Hom_G(V,X^{K_2})$ is a 
	finitely generated $\cO$-module by \cite[Lem.~2.3.10]{emordone}. We 
	conclude that there is a simple admissible $W 
	\in \Mod_{H}^{sm}(\cO)$ with a injective $H$-linear map $W \rightarrow 
	\Hom_G(V,X)$. It follows that we have a non-zero $G\times H$-linear map 
	$V\otimes_k W 
	\rightarrow X$. There is a finite extension $k'/k$ such that the extensions 
	of scalars $V_{k'}$ and $W_{k'}$ are direct sums of absolutely irreducible 
	representations. By Lemma \ref{tensorissimple}, $X_{k'}$ is isomorphic to 
	the tensor product of two of these absolutely irreducible representations.
\end{proof}
\begin{alemma}\label{locfinlocadm}
	Let $G = \prod_{i = 1}^m G_i$, where $G_i = \PGL_2(\Q_p)$. Let $V \in 
	\Mod_G^{sm}(\cO)$ be admissible and finitely generated over 
	$\cO[G]$. Then $V$ is of finite length. In particular, locally admissible 
	$G$-representations are locally finite.
	
	If $V$ is absolutely irreducible as a $G$-representation, there 
	is a finite extension $k'/k$ such that $V_{k'}$ is 
	isomorphic to $\otimes_{i=1}^m V_i$, where the $V_i$ are absolutely 
	irreducible 
	$G_i$-representations over $k'$.
\end{alemma}\begin{proof}
	Let $K_0 = \prod_{i=1}^m \PGL_2(\Z_p)$. 
	Following the argument of \cite[Thm.~2.3.8]{emordone}, it suffices to show 
	that every admissible quotient $V$ of $\cInd_{K_0}^GW$ is of finite length, 
	where $W$ is a finite dimensional absolutely irreducible representation of 
	$K_0$ over $k$. After extending scalars if necessary, $W$ decomposes as a 
	tensor product $W = \otimes_{i=1}^m 
	W_i$ 
	of representations of $PGL_2(\Z_p)$. As in \emph{loc.~cit.}~we consider 
	$\Hom_{k[G]}(\cInd_{K_0}^GW,V)$ which is a finite dimensional $k$-vector 
	space and a module over $\cH(W):= \End_{k[G]}(\cInd_{K_0}^GW)$. We have a 
	surjective map 	\[\Hom_{k[G]}(\cInd_{K_0}^GW,V) \otimes_{\cH(W)} 
	\cInd_{K_0}^GW \rightarrow V.\]
	
	The Hecke algebra $\cH(W)$ is isomorphic to the convolution algebra of 
	compactly supported functions $f: G \rightarrow \End_k(W)$ such that 
	$f(h_1g h_2) = h_1\circ f(g) \circ h_2$ for all $h_1,h_2 \in K_0$ and $g 
	\in 
	G$. With this description, one can show that \[\cH(W) \cong 
	\otimes_{i=1}^m\cH_i(W_i)\] where $\cH_i(W_i) = 
	\End_{k[G_i]}(\cInd_{\PGL_2(\Z_p)}^{G_i}W_i)$. By 
	\cite[Prop.~8]{barthel-livne}, we  have $\cH_i(W_i) \cong k[T_i]$ 
	and therefore we have $\cH(W)\cong k[T_1,\ldots,T_m]$. 
	
	Now it suffices to show that \[X \otimes_{\cH(W)} 
	\cInd_{K_0}^GW\] is of finite length, where $X$ is a finite dimensional 
	$\cH(W)$-module. By induction on the dimension of $X$, extending scalars if 
	necessary, we may assume that $X 
	\cong 
	\cH(W)/(T_1-\lambda_1,\ldots,T_m-\lambda_m)$, with $\lambda_i \in k$. 
	
	Since $\cInd_{K_0}^GW \cong \otimes_{i=1}^m 
	\cInd_{\PGL_2(\Z_p)}^{G_i}W_i$ we need to show 
	that 
	\[\otimes_{i=1}^m 
	\cInd_{\PGL_2(\Z_p)}^{G_i}W_i/(T_i - \lambda_i)\] has finite length, 
	which follows from Lemma \ref{tensorissimple} and the results of 
	\cite{barthel-livne,breuil1}.
	
	Finally, we repeatedly apply Lemma \ref{simpleprodweak} to show that if $V$ 
	is absolutely irreducible it factors as a tensor product after an extension 
	of scalars.
\end{proof}

\begin{alemma}\label{lem: tensor products projective envelopes}
	Let $G = \prod_{i = 1}^m G_i$, where $G_i = \PGL_2(\Q_p)$. Let $V = 
	\otimes_{i=1}^m V_i$ be an absolutely irreducible admissible representation 
	of $G$ (which factorises as shown). 
	Let $V_i \hookrightarrow I_i$, $i = 1,\ldots m$ be injective envelopes of 
	$V_i$ in 
	$\Mod^{\mathrm{loc~adm}}_{G_i}(\cO)$ (the category of locally admissible 
	representations). Dually, set $M_i = V_i^\vee$ 
	and $P_i = I_i^\vee$. 
	
	Then $\wotimes_{i=1}^m P_i \rightarrow \wotimes_{i=1}^mM_i$ is a 
	projective 
	envelope in $\mathfrak{C}_{G}(\cO)$ (see Definition~\ref{defn: gothic C 
	category}).\end{alemma}
\begin{proof}
	First we show that $P:=\wotimes_{i=1}^m P_i$ is projective in 
	$\mathfrak{C}_{G}(\cO)$. Note that it follows from Lemma \ref{tensorisadm} 
	and 
	Lemma \ref{locfinlocadm} that $P^\vee$ is locally admissible and locally 
	finite. 
	Let $M = \wotimes_{i=1}^mM_i \in 
	\mathfrak{C}_{G}(\cO)$. We induct on $m$. Let $P' = \wotimes_{i=2}^mP_i$ 
	and 
	$G' = \prod_{i=2}^{m}G_i$. By the universal property of the completed 
	tensor product we have \anumequation\label{tensorhomident}\Hom_{G_1\times 
	G'}^{cts}(P_1 \wotimes P', M)=   
	\Hom^{cts}_{G_1} (P_1, \Hom^{cts}_{G'}(P', M)),\end{equation} so 
	projectivity of $P$ 
	follows from projectivity of $P'$ and $P_1$.
	
	Now we prove that $P \rightarrow M$ is an essential surjection. Since 
	$P^\vee$ is locally finite, it suffices to show that $M = 
	\mathrm{cosoc}(P)$ (see \cite[Lem.~4.6]{CEGGPSGL2}). Again we proceed by 
	induction on $m$. So we assume that $\mathrm{cosoc}(P') = 
	\wotimes_{i=2}^mM_i$. Let $N \not\cong M$ be a simple object of 
	$\mathfrak{C}_{G}(\cO)$. We want to show that $\Hom_{G}^{cts}(P, N) = 0$. 
	Extending scalars to a field where $N^\vee$ is a direct sum of absolutely 
	irreducible representations, we reduce (using Lemma \ref{locfinlocadm}) to 
	the case where $N^\vee$ is absolutely irreducible and we have a 
	factorisation $N \cong \wotimes_{i=1}^mN_i$ 
	where the $N_i^\vee$ are absolutely irreducible. Let 
	$N'=
	\wotimes_{i=2}^mN_i$. By (\ref{tensorhomident}), we have 
	\[\Hom_{G}^{cts}(P, N)=   
	\Hom^{cts}_{G_1} (P_1, \Hom^{cts}_{G'}(P', N)).\] As an object of 
	$\mathfrak{C}_{G'}(\cO)$, we have $N = N_1 \wotimes_\cO N' = (\invlim N_1/U) 
	\wotimes_\cO N'$ where the limit runs over open submodules of $N_1$ and so 
	$N_1/U$ is a finite length $\cO$-module. In fact, since $N_1$ is simple, 
	$N_1/U$ is just a finite dimensional $k$-vector space. It follows from 
	Lemma \ref{brumercont} that, in $\mathfrak{C}_{G'}(\cO)$, we have an isomorphism $N 
	\cong \invlim (N_1/U \otimes_\cO N')$ and so we obtain isomorphisms 
	\[\Hom^{cts}_{G'}(P', N) \cong 
	\invlim \Hom^{cts}_{G'}(P', N_1/U \otimes_\cO N') = \invlim 
	\Hom^{cts}_{G'}(P', N')\otimes_\cO N_1/U\]
	
	Applying a similar argument, we conclude that \[\Hom^{cts}_{G}(P, N) \cong 
	\Hom^{cts}_{G_1}(P_1,N_1)\wotimes_\cO\Hom^{cts}_{G'}(P', 
	N').\]
	
We immediately deduce (from our inductive hypothesis) that  
$\Hom^{cts}_{G}(P, N)=0$. On the other hand, the same argument shows that we 
have \[\Hom^{cts}_{G}(P, 
M)=\Hom^{cts}_{G_1}(M_1,M_1)\wotimes_\cO\Hom^{cts}_{G'}(M', 
M') = \Hom^{cts}_{G}(M, 
M) = k.\] We deduce that $\mathrm{cosoc}(P) = M$, as desired.
\end{proof}

\emergencystretch=3em
\printbibliography

\end{document}